\documentclass[11pt]{amsart}

\usepackage[utf8]{inputenc}
\usepackage{amssymb}
\usepackage{graphicx}   
\usepackage{mathrsfs}
\usepackage{amsmath}
\usepackage{amsthm}
\usepackage{amsfonts}
\usepackage{latexsym}
\usepackage{hyperref}
\usepackage[top=3.5cm,bottom=3.5cm,left=3.5cm,right=3.5cm]{geometry}

\newtheorem{thm}{Theorem}[section]
\newtheorem{prop}[thm]{Proposition}
\newtheorem{lem}[thm]{Lemma}
\newtheorem{cor}[thm]{Corollary}

\theoremstyle{definition}
\newtheorem{definition}[thm]{Definition}

\theoremstyle{remark}
\newtheorem{remark}[thm]{Remark}

\numberwithin{equation}{section}

\newcommand{\R}{\mathbb{R}}  % The real numbers.

\newcommand{\cb}{{\mathcal B}} % Balls in curved metrics on T_x\Sigma
\newcommand{\mrestr}{\mathbin{\vrule height 1.6ex depth 0pt width
	0.13ex\vrule height 0.13ex depth 0pt width 1.3ex}} % Restriction of measures
\renewcommand{\tilde}{\widetilde} % Always use wide tilde
\renewcommand{\hat}{\widehat} % Always use wide hat
\renewcommand{\bar}{\overline} % Always use wide bar

\begin{document}

\title{Minimal submanifolds from the abelian Higgs model}

\author{Alessandro Pigati}
\address{ETH Z\"urich, Department of Mathematics,
R\"amistrasse 101, 8092 Z\"urich}
\email{alessandro.pigati@math.ethz.ch}
\author{Daniel Stern}
\address{Princeton University, Department of Mathematics, 
Princeton, NJ 08544}
\email{dls6@math.princeton.edu}

\begin{abstract}
Given a Hermitian line bundle $L\to M$ over a closed, oriented Riemannian manifold $M$,
we study the asymptotic behavior, as $\epsilon\to 0$, of couples $(u_\epsilon,\nabla_\epsilon)$ critical for the rescalings
\begin{align*}
	&E_\epsilon(u,\nabla)=\int_M\Big(|\nabla u|^2+\epsilon^2|F_\nabla|^2+\frac{1}{4\epsilon^2}(1-|u|^2)^2\Big)
\end{align*}
of the self-dual Yang--Mills--Higgs energy, where $u$ is a section of $L$ and $\nabla$ is a Hermitian connection on $L$ with curvature $F_{\nabla}$.

Under the natural assumption $\limsup_{\epsilon\to 0}E_\epsilon(u_\epsilon,\nabla_\epsilon)<\infty$, we show that the energy measures
converge subsequentially to (the weight measure $\mu$ of) a stationary integral $(n-2)$-varifold.
Also, we show that the $(n-2)$-currents dual to the curvature forms converge subsequentially to $2\pi\Gamma$, for an integral $(n-2)$-cycle $\Gamma$ with $|\Gamma|\le\mu$.

Finally, we provide a variational construction of nontrivial critical points $(u_\epsilon,\nabla_\epsilon)$ on arbitrary line bundles, satisfying a uniform energy bound. As a byproduct, we obtain a PDE proof, in codimension two, of Almgren's existence result of (nontrivial) stationary integral $(n-2)$-varifolds in an arbitrary closed Riemannian manifold.
\end{abstract}

\maketitle

%%
%% LaTeX can automatically make a table of contents.  This is done by
%% uncommenting the following:
%%

%%\tableofcontents

\section{Introduction}

A level set approach for the variational construction of minimal hypersurfaces was born from the work of Modica--Mortola \cite{MM}, Modica \cite{Mod}, and Sternberg \cite{Ste}. Starting from a suggestion by De Giorgi \cite{DG},
they highlighted a deep connection between minimizers $u_\epsilon:M\to\R$ of the Allen--Cahn functional
\begin{align*}
	&F_\epsilon(v):=\int_M\Big(\epsilon |dv|^2+\frac{1}{4\epsilon}(1-v^2)^2\Big),
\end{align*}
and two-sided minimal hypersurfaces in $M$, showing essentially that the functionals $F_{\epsilon}$ $\Gamma$-converge to ($\frac{4}{3}$ times) the perimeter functional on Caccioppoli sets. Several years later, Hutchinson and Tonegawa \cite{HT} initiated  the asymptotic study of critical points $v_{\epsilon}$ of $F_{\epsilon}$ with bounded energy, without the energy-minimality assumption. They showed, in particular, that their energy measures concentrate along a stationary, integral $(n-1)$-varifold, given by the limit of the level sets $v_{\epsilon}^{-1}(0)$. 

These developments, together with the deep regularity work by Tonegawa and Wickramasekera on stable solutions \cite{TonWic}, opened the doors to a fruitful min-max approach to the construction of minimal hypersurfaces, providing a PDE alternative to the rather involved discretized min-max procedure implemented by Almgren and Pitts (\cite{Almvar}, \cite{Pitts}) in the setting of geometric measure theory. This promising min-max approach based on the Allen--Cahn functionals was recently developed by Guaraco and Gaspar--Guaraco \cite{Gua,GasGua}, and has been used successfully to attack some deep questions concerning the structure of min-max minimal hypersurfaces---most notably in Chodosh and Mantoulidis's work on the multiplicity one conjecture \cite{ChoMan}.

The initial motivation for this paper is to find, in a similar vein, a natural way to construct minimal varieties of codimension two through PDE methods.
Recently, other attempts in this direction have been made by Cheng \cite{Cheng} and the second-named author \cite{Stern}, based on the study of the Ginzburg--Landau functionals 
$$F_{\epsilon}(v):=\frac{1}{|\log\epsilon|}\int \Big(|dv|^2+\frac{1}{4\epsilon^2}(1-|v|^2)^2\Big)$$
on complex-valued maps $v:M\to \mathbb{C}$. While the Ginzburg--Landau approach can be employed successfully to produce nontrivial stationary \emph{rectifiable} $(n-2)$-varifolds (building on the analysis of \cite{LRgl}, \cite{BBO}, and others), and leads to existence results of independent interest for solutions of the Ginzburg--Landau equations, it is not yet known whether the varifolds produced in this way are \emph{integral}, nor is it known whether the full energies $F_{\epsilon}(v_{\epsilon})$ of the min-max critical points converge to the mass of the limiting minimal variety in the case $b_1(M)\neq 0$.

While it is possible that these and other technical difficulties may be overcome with sufficient effort---and establishing integrality in particular remains a fascinating open problem---they point to the deeper fact that the Ginzburg--Landau functionals, though intimately related to the $(n-2)$-area, do \emph{not} provide a straightforward regularization of the codimension-two area functional. Indeed, we stress that the Ginzburg--Landau energies should be understood first and foremost as a relaxation of the Dirichlet energy for singular maps to $S^1$, and while the limiting singularities of critical points may coincide with minimal varieties, the associated variational problems exhibit substantial qualitative differences at both large and small scales.

In the present paper, we consider instead the self-dual Yang--Mills--Higgs energy
\begin{equation}\label{ymh}
E(u,\nabla):=\int_M \Big(|\nabla u|^2+|F_{\nabla}|^2+W(u)\Big)
\end{equation}
and its rescalings (for $\epsilon\in (0,1]$)
\begin{equation}\label{ymhe}
E_{\epsilon}(u,\nabla):=\int_M\Big(|\nabla u|^2+\epsilon^2|F_{\nabla}|^2+\epsilon^{-2}W(u)\Big),
\end{equation}
for couples $(u,\nabla)$ consisting of a section $u$ of a given Hermitian line bundle $L\to M$, and a metric connection $\nabla$ on $L$. Here, the nonlinear potential $W: L\to \mathbb{R}$ is given by
\begin{equation}\label{wfix}
W(u):=\frac{1}{4}(1-|u|^2)^2,
\end{equation}
while $F_{\nabla}\in \Omega^2(\operatorname{End}(L))$ denotes the curvature of $\nabla$.

For the trivial bundle $L=\mathbb{C}\times \mathbb{R}^2$ on the plane $M=\mathbb{R}^2$, a detailed study of the functional \eqref{ymh} and its critical points can be found in the doctoral work of Taubes \cite{Taubes1,Taubes2}. In \cite{Taubes2}, all finite-energy critical points $(u,\nabla)$ of \eqref{ymh} in the plane are shown to solve the first order system\footnote{Here and elsewhere, we implicitly identify $F_{\nabla}$ with the two-form $\omega$ given by $F_{\nabla}(X,Y)=-i\omega(X,Y)$.}
\begin{equation}\label{vort}
\nabla_{\partial_1}u\pm i\nabla_{\partial_2}u=0;\quad*F_{\nabla}=\pm \frac{1}{2}(1-|u|^2)
\end{equation}
known as the \emph{vortex equations}---a two-dimensional counterpart of the instanton equations in four-dimensional Yang--Mills theory. In particular, all such solutions $(u,\nabla)$ minimize energy among pairs $(u,\nabla)$ with fixed vortex number
$$N:=\frac{1}{2\pi}\int_{\mathbb{R}^2}*F_{\nabla}\in \mathbb{Z},$$
and carry energy exactly $E(u,\nabla)=2\pi |N|$. In \cite{Taubes1}, Taubes shows moreover that there exist solutions of \eqref{vort} with any prescribed zero set 
$$u^{-1}(0)=\{z_1,\ldots,z_N\}\subset \mathbb{R}^2,$$
which are unique up to gauge equivalence, so that \cite{Taubes1} and \cite{Taubes2} together give a complete classification of finite-energy critical points of \eqref{ymh} in the plane. 

In \cite{HJS}, Hong, Jost, and Struwe initiate the study of the rescaled functionals \eqref{ymhe} in the limit $\epsilon\to 0$ for line bundles $L\to \Sigma$ over a closed Riemann surface $\Sigma$. The main result of \cite{HJS} shows that, for solutions $(u_{\epsilon},\nabla_{\epsilon})$ of the rescaled vortex equations (given by replacing $\frac{1}{2}(1-|u|^2)$ with $\frac{1}{2\epsilon^2}(1-|u_\epsilon|^2)$ in \eqref{vort}), the curvature $*\frac{1}{2\pi}F_{\nabla_{\epsilon}}$ converges as $\epsilon\to 0$ to a finite sum of Dirac masses of total mass $|\operatorname{deg}(L)|$, away from which $\nabla_{\epsilon}$ converges to a flat connection $\nabla_0$, and $u_{\epsilon}$ to a unit section $u_0$ with $\nabla_0u_0=0$. While the authors of \cite{HJS} focus on the vortex equations over Riemann surfaces, they suggest that the asymptotic analysis of the rescaled functionals $E_{\epsilon}$ may also yield interesting results in higher dimension, pointing to similarities with the Allen--Cahn functionals for scalar-valued functions.

In the present paper, we develop the asymptotic analysis as $\epsilon\to 0$ for critical points of $E_{\epsilon}$ associated to line bundles $L\to M$ over Riemannian manifolds $M^n$ of arbitrary dimension $n\geq 2$. The bulk of the paper is devoted to the proof of the following theorem, which describes the limiting behavior as $\epsilon\to 0$ of the energy measures
$$\mu_{\epsilon}:=\frac{1}{2\pi}e_{\epsilon}(u_{\epsilon},\nabla_{\epsilon})\,\operatorname{vol}_g$$
and curvatures $F_{\nabla_{\epsilon}}$ for critical points $(u_{\epsilon},\nabla_{\epsilon})$ satisfying a uniform energy bound.

\begin{thm}\label{limthm} Let $L\to M$ be a Hermitian line bundle over a closed, oriented Riemannian manifold $M^n$ of dimension $n\geq 2$, and let $(u_{\epsilon},\nabla_{\epsilon})$ be a family of critical points for $E_{\epsilon}$ satisfying a uniform energy bound
$$E_{\epsilon}(u_{\epsilon},\nabla_{\epsilon})\leq \Lambda<\infty.$$
Then, as $\epsilon\to 0$, the energy measures 
$$\mu_{\epsilon}:=\frac{1}{2\pi}e_{\epsilon}(u_{\epsilon},\nabla_{\epsilon})\,\operatorname{vol}_g$$
converge subsequentially, in duality with $C^0(M)$, to the weight measure of a stationary, integral $(n-2)$-varifold $V$. Also, for all $0\le\delta<1$,
$$\operatorname{spt}(V)=\lim_{\epsilon \to 0}\,\{|u_{\epsilon}|\le\delta\}$$
in the Hausdorff topology. The $(n-2)$-currents dual to the curvature forms $\frac{1}{2\pi}F_{\nabla_{\epsilon}}$ converge subsequentially to an integral $(n-2)$-cycle $\Gamma$, with $|\Gamma|\le\mu$.
\end{thm}

Roughly speaking, Theorem \ref{limthm} says that the energy of the critical points concentrates near the zero sets $u_{\epsilon}^{-1}(0)$ of $u_{\epsilon}$ as $\epsilon\to 0$, which converge to a (possibly rather singular) minimal submanifold of codimension two. In the case $\dim(M)=3$, for instance, it follows from the results above and work of Allard and Almgren \cite{AA} that energy concentrates along a stationary geodesic network with integer multiplicities. The convergence of the curvature, moreover, to an integral cycle Poincar\'{e} dual to $c_1(L)$, with mass bounded above by $\lim_{\epsilon\to 0}E_{\epsilon}(u_{\epsilon},\nabla_{\epsilon})$, provides a higher dimensional analog to the limiting behavior described in two dimensions by Hong--Jost--Struwe \cite{HJS}.

At first glance, the obvious advantages of Theorem \ref{limthm} over analogous results for the complex Ginzburg--Landau equations (cf., e.g., \cite{BBO}, \cite{Stern}) are the \emph{integrality} of the limit varifold $V$, and the concentration of the \emph{full energy measure} to $V$, independent of the topology of $M$. Indeed, Theorem \ref{limthm} and the analysis leading to its proof align much more closely with the work of Hutchinson and Tonegawa \cite{HT} on the Allen--Cahn equations than they do with related results (e.g. \cite{LRems}, \cite{BBO}) for the complex Ginzburg--Landau equations. The parallels between the analysis presented here and that of the Allen--Cahn equations in \cite{HT} are in fact quite striking in places---a point to which we will draw the reader's attention throughout the paper.

\begin{remark} We warn the reader, however, that while the qualitative analysis of the Allen--Cahn functionals does not depend on the precise choice of the double-well potential $W$, the analysis of the abelian Yang--Mills--Higgs functionals \eqref{ymh}--\eqref{ymhe} seems to depend \emph{quite strongly} on the choice $W(u)=\frac{1}{4}(1-|u|^2)^2$. Indeed, already in two dimensions, replacing $W$ with a potential $W_{\lambda}(u):=\frac{\lambda}{4}(1-|u|^2)^2$ for some $\lambda\neq 1$ yields a dramatically different qualitative behavior, breaking the symmetry which leads to the first-order equations \eqref{vort}, and introducing interactions between disjoint components of the zero set (see, e.g., \cite[Chapters~I--III]{JT}). This should serve as one indication that the analysis of the abelian Higgs model is somewhat more delicate than that of related semilinear scalar equations, in spite of the strong parallels.
\end{remark}

To get some idea of the role played by gauge invariance, note that unit sections of a Hermitian line bundle are indistinguishable up to change of gauge (when no preferred connection has been selected), and for a given unit section $u$ of $L$, one can always choose locally a connection with respect to which $u$ appears constant. Thus, while most of the energy of solutions $v_{\epsilon}$ to the complex Ginzburg--Landau equations falls on annular regions---relatively far from the zero set---where $v_{\epsilon}$ resembles a harmonic $S^1$-valued map, the energy $e_{\epsilon}(u_{\epsilon},\nabla_{\epsilon})$ of a critical pair $(u_{\epsilon},\nabla_{\epsilon})$ for the abelian Yang--Mills--Higgs energy instead concentrates near the zero set $u_{\epsilon}^{-1}(0)$, with $|\nabla_{\epsilon}u_{\epsilon}|$ vanishing rapidly outside this region.

Of course, the results of Theorem \ref{limthm} would be of limited interest if nontrivial critical points $(u_{\epsilon},\nabla_{\epsilon})$ could be found only in a few special settings. After completing the proof of Theorem \ref{limthm}, we therefore establish the following general existence result, showing that nontrivial families satisfying the hypotheses of Theorem \ref{limthm} arise naturally on any line bundle (including, importantly, the trivial bundle) over any oriented Riemannian manifold $M^n$, from variational constructions.

\begin{thm}\label{existthm} For any Hermitian line bundle $L\to M$ over an arbitrary closed base manifold $M^n$, there exists a family $(u_{\epsilon},\nabla_{\epsilon})$ satisfying the hypotheses of Theorem \ref{limthm}, with nonempty zero sets $u_{\epsilon}^{-1}(0)\neq \varnothing$. In particular, the energy $\mu_{\epsilon}$ of these families concentrates (subsequentially) on a nontrivial stationary integral $(n-2)$-varifold $V$ as $\epsilon\to 0$.
\end{thm}

For nontrivial bundles $L\to M$, this follows from a fairly simple argument, showing that the minimizers $(u_{\epsilon},\nabla_{\epsilon})$ of $E_{\epsilon}$ satisfy uniform energy bounds as $\epsilon\to 0$. For these energy-minimizing solutions, we expect moreover that the limiting minimal variety $\mu=\theta \mathcal{H}^{n-2}\mrestr\Sigma$, i.e. the weight measure $|V|$ of $V$, coincides with the weight measure $|\Gamma|$ of the limiting $(n-2)$-cycle $\Gamma=\lim_{\epsilon\to 0} *\frac{1}{2\pi}F_{\nabla_{\epsilon}}$, and that $\Gamma$ minimizes $(n-2)$-area in its homology class. While we do not take up this question here, we believe that it would be very interesting to study the convergence of the functionals \eqref{ymhe} to the $(n-2)$-area functional in a $\Gamma$-convergence framework. Let us mention that an asymptotic study for \emph{minimizers} of the Ginzburg--Landau functional, on a domain with boundary, was successfully carried out by Lin and Rivi\`ere \cite{LRems}, who were able to identify the concentration measure with the weight of an integral current. (See also \cite{ABO}, \cite{JS} for related $\Gamma$-convergence results in that setting.)

\begin{remark}
We remark that a very special class of minimizers for $E_{\epsilon}$ are given by solutions $(u_{\epsilon},\nabla_{\epsilon})$ of the first-order vortex equations in K\"{a}hler manifolds $(M^{2n},\omega_K)$ of higher dimension; these generalize the system (\ref{vort}) from the two-dimensional setting by replacing $*F_{\nabla}$ in (\ref{vort}) by the inner product $\langle F_{\nabla},\omega_K\rangle$ with the K\"{a}hler form $\omega_K$, and requiring additionally that $F_{\nabla}^{0,2}=0$. As in the two-dimensional setting, solutions of this first-order system minimize the energy $E_{\epsilon}$ in appropriate line bundles on K\"{a}hler manifolds, and it was shown by Bradlow\footnote{The precise form of the energies considered by Bradlow in \cite{Brad} differs slightly from the functionals $E_{\epsilon}$ considered here, but the analysis is essentially the same.} \cite{Brad} that the moduli space of solutions corresponds to the space of complex subvarieties in $M$ (of complex codimension one) via the zero locus $(u_{\epsilon},\nabla_{\epsilon})\mapsto u_{\epsilon}^{-1}(0)$. 

In particular, the zero loci $u_{\epsilon}^{-1}(0)$ in this case are already area-minimizing subvarieties, before passing to the limit $\epsilon\to 0$. Note that the analysis of the vortex equations plays a key role in the study of Seiberg--Witten invariants of K\"{a}hler surfaces \cite{Wit}, and a similar analysis figures crucially into Taubes's work relating the Seiberg--Witten and Gromov--Witten invariants of symplectic four-manifolds \cite{TaubesSW}. For a concise introduction to the higher-dimensional vortex equations and connections to Seiberg--Witten theory, we refer the interested reader to the survey \cite{GP} by Garc\'{i}a-Prada.
\end{remark}

For the trivial bundle $L\cong \mathbb{C}\times M$, we prove Theorem \ref{existthm} by applying min-max methods to the functionals \eqref{ymhe}, to produce nontrivial families $(u_{\epsilon},\nabla_{\epsilon})$ satisfying a uniform energy bound as $\epsilon\to 0$. While we consider only one min-max construction in the present paper, we remark that many more may be carried out in principle, due to the rich topology of the space 
$$\mathcal{M}:=\{(u,\nabla) : 0\not\equiv u\in \Gamma(\mathbb{C}\times M),\text{ }\nabla\text{ a Hermitian connection}\}/\mathcal{G},$$
where $\mathcal{G}:=\operatorname{Maps}(M,S^1)$ is the gauge group. Indeed, on a closed oriented manifold $M$, one can show that the homotopy groups $\pi_i(\mathcal{M})$ are given by
$$\pi_1(\mathcal{M})\cong H^1(M;\mathbb{Z}),\text{ }\pi_2(\mathcal{M})\cong \mathbb{Z},\text{ and }\pi_i(\mathcal{M})=0\text{ for }i\geq 3;$$
it may be of interest to note that these are isomorphic to the homotopy groups of the space $\mathcal{Z}_{n-2}(M;\mathbb{Z})$ of integral $(n-2)$-cycles in $M$, as computed by Almgren \cite{Almdiss}. 

As an application of Theorem \ref{existthm}, we obtain a new proof of the existence of stationary integral $(n-2)$-varifolds in an arbitrary Riemannian manifold---a result first proved by Almgren in 1965 \cite{Almvar} using a powerful, but rather involved geometric measure theory framework. As already mentioned, similar constructions for the Allen--Cahn equations have been carried out successfully by Guaraco \cite{Gua} and Gaspar--Guaraco \cite{GasGua}, yielding new proofs of the existence of minimal hypersurfaces of optimal regularity, and leading to other recent breakthroughs in the min-max theory of minimal hypersurfaces (e.g., \cite{ChoMan}). 

In \cite{ChoMan} and \cite{Gua} (building on results of \cite{TonWic}), the stability properties of the min-max critical points for the Allen--Cahn functionals play a central role in controlling the regularity and multiplicity of the limit hypersurface. To obtain an improved understanding of min-max families $(u_{\epsilon},\nabla_{\epsilon})$ and the associated minimal varieties in the abelian Higgs setting, it would likewise be very interesting to refine the conclusions of Theorem \ref{limthm} under the assumption that the families $(u_{\epsilon},\nabla_{\epsilon})$ satisfy a uniform Morse index bound as $\epsilon\to 0$. We hope to take up this line of investigation in future work. 

\subsection{Organization of the paper}
In Section \ref{preliminary.sec} we fix notation and record some basic properties satisfied by critical pairs $(u_{\epsilon},\nabla_{\epsilon})$ for the energies $E_{\epsilon}$. 

In Section \ref{bochsec}, we record some useul Bochner identities for the gauge-invariant quantities $|u|^2$, $|F_{\nabla}|^2$, and $|\nabla u|^2$, and use them to establish an initial rough estimate on $\xi_\epsilon:=\epsilon|F_\nabla|-\frac{(1-|u|^2)}{2\epsilon}$, whose role should be compared to that of the \emph{discrepancy function} in the Allen--Cahn setting. Under suitable assumptions on the curvature of $M$, the fact that $\xi_\epsilon\le 0$ follows quickly from the aforementioned Bochner identities and the maximum principle. Without the curvature assumptions, some nontrivial additional work is required to obtain the pointwise upper bound $\xi_\epsilon\le C(M,E_\epsilon(u,\nabla))$. This estimate is the key ingredient to obtain the sharp $(n-2)$-monotonicity of the energy.

In Section \ref{monosec} we derive the stationarity equation for inner variations, from which an obvious $(n-4)$-monotonicity property of the energy follows rather immediately. Using our rough initial bounds on $\xi_\epsilon$ from Section \ref{bochsec}, we deduce an intermediate $(n-3)$-monotonicity; we use this to reach the pointwise bound $\xi_\epsilon\le C(M,E_\epsilon(u,\nabla))$, from which we finally infer the sharp $(n-2)$-monotonicity.

In Section \ref{decaysec} we show that, similar to the Allen--Cahn setting, the energy density $e_{\epsilon}(u,\nabla)$ decays exponentially away from the set $u^{-1}(0)$---more precisely, away from $\{|u|^2\ge 1-\beta_d\}$ for some $\beta_d$ independent of $\epsilon$.

Section \ref{varifold.section}, which constitutes the main part of the paper, contains an initial description of the limiting varifold, showing that it is stationary, $(n-2)$-rectifiable, and has a lower density bound on the support. Then we establish its integrality with a blow-up analysis, employing the estimates from the preceding sections to reduce the problem to a statement for entire planar solutions, already contained in the work of Jaffe and Taubes \cite{JT}.
We then use this analysis to show that the level sets $u_\epsilon^{-1}(0)$ converge to the support of $V$ in the Hausdorff topology, and conclude the section with a discussion of the asymptotics for the curvature forms $\frac{1}{2\pi}F_{\nabla_{\epsilon}}$.

In Section \ref{minmaxsec}, we show that $E_\epsilon$ satisfies a variant of the Palais--Smale property on suitable function spaces, allowing us to produce critical points via classical min-max methods. We provide a variational construction to get nontrivial critical points satisfying the assumptions of our main theorem, with energy bounded from above and below, both for nontrivial and trivial line bundles.

Finally, the Appendix addresses the issue of obtaining regularity of critical points, as obtained from Section \ref{minmaxsec}, when they are read in a local or global Coulomb gauge.

\section*{Acknowledgements} A hearty thank-you goes to Tristan Rivi\`{e}re for introducing the authors to each other, and for suggesting the line of investigation taken up in the present paper. D.S. also thanks Fernando Cod\'{a} Marques for his interest in this work, and Francesco Lin for pointing him to the reference \cite{Wit}. A.P. is partially supported by SNSF grant 172707. During the completion of this work, D.S. was supported in part by NSF grant DMS-1502424.

\section{The Yang--Mills--Higgs equations on \texorpdfstring{$U(1)$}{U(1)} bundles}\label{preliminary.sec}

Let $M$ be a closed, oriented Riemannian manifold, and let $L\to M^n$ be a complex line bundle over $M$, endowed with a Hermitian structure $\langle\cdot,\cdot\rangle$. Denote by $W: L\to \mathbb{R}$ the nonlinear potential
$$W(u):=\frac{1}{4}(1-|u|^2)^2.$$
For a Hermitian connection $\nabla$ on $L$, a section $u\in \Gamma(L)$ and a parameter $\epsilon>0$, denote by $E_{\epsilon}(u,\nabla)$ the scaled Yang--Mills--Higgs energy
\begin{equation}
E_{\epsilon}(u,\nabla):=\int_M\Big(|\nabla u|^2+\epsilon^2|F_{\nabla}|^2+\epsilon^{-2}W(u)\Big),
\end{equation}
where $F_{\nabla}$ is the curvature of $\nabla$. Throughout, we will identify the curvature $F_{\nabla}$ with a closed real two-form $\omega$ via
\begin{equation}
F_{\nabla}(X,Y)u=[\nabla_X,\nabla_Y]u-\nabla_{[X,Y]}u=-i\omega(X,Y)u.
\end{equation}
In computing inner products for two-forms, we follow the convention
\begin{align}\label{two.form.conv}
	&|\omega|^2=\sum_{1\leq j<k\leq n}\omega(e_j,e_k)^2=\frac{1}{2}\sum_{j,k=1}^n\omega(e_j,e_k)^2
\end{align}
with respect to a local orthonormal basis $\{e_j\}_{j=1}^n$ for $TM$.

It is easy to check that the smooth pair $(u,\nabla)$ gives a critical point for the energy $E_{\epsilon}$, with respect to smooth variations, if and only if it satisfies the system
\begin{align}
\label{ueq} \nabla^*\nabla u&=\frac{1}{2\epsilon^2}(1-|u|^2)u, \\[5pt]
\label{feq} \epsilon^2d^*\omega&=\langle \nabla u,iu\rangle.
\end{align}
Note that, in our convention, the adjoint to $d:\Omega^1(M)\to\Omega^2(M)$ is
$$(d^*\omega)(e_k)=-\sum_{j=1}^n(\nabla_{e_j}\omega)(e_j,e_k).$$
Since the curvature form $\omega$ is closed, taking the exterior derivative of \eqref{feq} gives
\begin{align*}
\epsilon^2(\Delta_H\omega)(e_j,e_k)&=(d\langle \nabla u,iu\rangle)(e_j,e_k)\\
&=\langle i \nabla_{e_j}u,\nabla_{e_k}u\rangle-\langle i\nabla_{e_k}u,\nabla_{e_j}u\rangle\\
&\quad+\langle iu,F_{\nabla}(e_j,e_k)u\rangle\\
&=\psi(u)(e_j,e_k)-|u|^2\omega(e_j,e_k);
\end{align*}
i.e.,
\begin{equation}\label{lapomega}
\epsilon^2\Delta_H\omega=-|u|^2\omega+\psi(u),
\end{equation}
where
$$\psi(u)(e_j,e_k):=2\langle i\nabla_{e_j}u,\nabla_{e_k}u\rangle.$$
For future reference, we record the simple bound
\begin{equation}\label{psinormbd}
|\psi(u)|\leq |\nabla u|^2.
\end{equation}
To confirm \eqref{psinormbd}, fix $x\in M$ and note that the linear map $\nabla u(x):T_xM\to L_x$ has a kernel of dimension at least $n-2$. Take an orthonormal basis $\{e_j\}$ of $T_xM$ such that $e_j\in\operatorname{ker}\nabla u(x)$ for $j>2$. We compute at $x$ that
\begin{align*}
	&|\psi(u)|=2|\langle i\nabla_{e_1}u,\nabla_{e_2}u\rangle|\le 2|\nabla_{e_1}u||\nabla_{e_2}u|\le|\nabla_{e_1}u|^2+|\nabla_{e_2}u|^2,
\end{align*}
which gives \eqref{psinormbd}.

\section{Bochner identities and preliminary estimates}\label{bochsec}

From the equations \eqref{lapomega} and \eqref{ueq}, we apply the standard Bochner--Weitzenb\"{o}ck formulas to obtain some identities which will play a central role in our analysis. For the curvature two-form $\omega$, it will be useful to record the Bochner identity
\begin{equation}\label{omega2boch}
\Delta \frac{1}{2}|\omega|^2=|D\omega|^2+\epsilon^{-2}(|u|^2|\omega|^2-\langle \psi(u),\omega\rangle)+\mathcal{R}_2(\omega,\omega),
\end{equation}
where $\mathcal{R}_2$ denotes the Weitzenb\"{o}ck curvature operator for two-forms on the base Riemannian manifold $M$.
For any $\delta>0$ we have
\begin{align*}
	&(|\omega|^2+\delta^2)^{1/2}\Delta(|\omega|^2+\delta^2)^{1/2}+|D|\omega||^2\ge\Delta\frac{1}{2}(|\omega|^2+\delta^2)=\Delta \frac{1}{2}|\omega|^2.
\end{align*}
Since $|D|\omega||^2\le|D\omega|^2$, \eqref{omega2boch} implies
\begin{align*}
	&(|\omega|^2+\delta^2)^{1/2}\Delta(|\omega|^2+\delta^2)^{1/2}\ge\epsilon^{-2}(|u|^2|\omega|^2-\langle \psi(u),\omega\rangle)+\mathcal{R}_2(\omega,\omega).
\end{align*}
Dividing by $(|\omega|^2+\delta^2)^{1/2}$ and letting $\delta\to 0$, we obtain
\begin{equation}\label{omega1boch}
%\Delta|\omega|\geq |\omega|^{-1}\epsilon^{-2}(|u|^2|\omega|^2-\langle \psi(u),\omega\rangle)+|\omega|^{-1}\mathcal{R}_2(\omega,\omega),
\Delta|\omega|\geq\epsilon^{-2}(|u|^2|\omega|-|\psi(u)|)-|\mathcal{R}_2^-||\omega|,
\end{equation}
in the distributional sense (and classically on $\{|\omega|>0\}$).
%(Where \eqref{omega1boch} holds pointwise on , and globally in a distributional sense.)
Note that, by \eqref{psinormbd}, the relation \eqref{omega1boch} also gives us the cruder subequation
\begin{equation}\label{omega1boch2}
\Delta |\omega|\geq \epsilon^{-2}|u|^2|\omega|-\epsilon^{-2}|\nabla u|^2-|\mathcal{R}_2^-||\omega|.
\end{equation}

For the modulus $|u|^2$ of the Higgs field $u$, we record
\begin{equation}\label{moduboch}
\Delta \frac{1}{2}|u|^2=|\nabla u|^2-\frac{1}{2\epsilon^2}(1-|u|^2)|u|^2,
\end{equation}
and observe that a simple application of the maximum principle yields the pointwise bound
$$|u|^2\leq 1\text{ on }M.$$
For the energy density $|\nabla u|^2$ of the Higgs field $u$, we see that
\begin{align*}
\Delta \frac{1}{2}|\nabla u|^2&=|\nabla^2u|^2-\langle\nabla(\nabla^*\nabla u),\nabla u\rangle+\langle d^*\omega,\langle iu,\nabla u\rangle\rangle\\
&\quad-2\langle\omega,\psi(u)\rangle+\mathcal{R}_1(\nabla u,\nabla u)\\
&=|\nabla^2u|^2-2\langle \omega,\psi(u)\rangle+\frac{1}{\epsilon^2}|\langle iu,\nabla u\rangle|^2\\
&\quad-\frac{1}{2\epsilon^2}(1-|u|^2)|\nabla u|^2+\frac{1}{\epsilon^2}|\langle u,\nabla u\rangle|^2+\mathcal{R}_1(\nabla u,\nabla u)\\
&=|\nabla^2u|^2+\frac{1}{2\epsilon^2}(3|u|^2-1)|\nabla u|^2-2\langle \omega,\psi(u)\rangle+\mathcal{R}_1(\nabla u,\nabla u),
\end{align*}
where at $p\in M$ we let $\mathcal{R}_1(\nabla u,\nabla u)=\operatorname{Ric}(e_i,e_j)\langle\nabla_{e_i}u,\nabla_{e_j}u\rangle$ and $\nabla^2_{e_i,e_j}u=\nabla_{e_i}(\nabla_{e_j}u)$, for any local orthonormal frame $\{e_i\}$ with $De_i(p)=0$.

Next, we introduce the function
\begin{equation}\label{xidef}
\xi_{\epsilon}:=\epsilon|F_{\nabla}|-\frac{1}{2\epsilon}(1-|u|^2),
\end{equation}
and combine \eqref{omega1boch2} with \eqref{moduboch} to see that
\begin{align*}
\Delta \xi_{\epsilon}%&\geq\epsilon^{-1}|u|^2|\omega|-\epsilon^{-1}|\nabla u|^2+\epsilon|\omega|^{-1}\mathcal{R}_2(\omega,\omega)\\
%&\quad+\epsilon^{-1}|\nabla u|^2-\frac{1}{2\epsilon^3}(1-|u|^2)|u|^2\\
%&=\epsilon^{-2}|u|^2\xi_{\epsilon}+\epsilon|\omega|^{-1}\mathcal{R}_2(\omega,\omega)\\
%&\geq\epsilon^{-2}|u|^2\xi_{\epsilon}-\epsilon\|\mathcal{R}_2^-\|_{L^{\infty}}|\omega|.
&\geq\epsilon^{-1}|u|^2|\omega|-\epsilon^{-1}|\nabla u|^2-\epsilon|\mathcal{R}_2||\omega|+\epsilon^{-1}|\nabla u|^2-\frac{1}{2\epsilon^3}(1-|u|^2)|u|^2 \\
&\ge\epsilon^{-2}|u|^2\xi_\epsilon-\epsilon\|\mathcal{R}_2\|_{L^\infty}|\omega|.
\end{align*}
From a simple application of the maximum principle, we see in particular that if $\mathcal{R}_2>0$, then $\xi_{\epsilon}\leq 0$ everywhere on $M$, and consequently (cf. \cite[Theorem~III.8.1]{JT})
\begin{equation}\label{ptwisebal}
\epsilon^2|F_{\nabla}|^2\leq \frac{W(u)}{\epsilon^2}\text{ pointwise, provided }\mathcal{R}_2> 0\text{ on }M.
\end{equation}
This balancing of the Yang--Mills and potential terms,
which should be compared with Modica's gradient estimate in the asymptotic analysis of the Allen--Cahn equations
(cf. \cite[Proposition~3.3]{HT}), will play a key role in our analysis, allowing us to upgrade the obvious ${(n-4)}$-monotonicity typical of Yang--Mills--Higgs problems to the much stronger ${(n-2)}$-monotonicity $\frac{d}{dr}(r^{2-n}\int_{B_r}e_{\epsilon}(u_{\epsilon},\nabla_{\epsilon}))\geq 0$.

Without the positive curvature assumption, we may still employ the subequation
\begin{equation}\label{balsubeq}
\Delta \xi_{\epsilon}\geq \frac{|u|^2}{\epsilon^2}\xi_{\epsilon}-C(M)\epsilon|F_{\nabla}|,
\end{equation}
to obtain strong estimates for the positive part $\xi_{\epsilon}^+$ of $\xi_{\epsilon}$. To begin, denote by $G(x,y)$ the nonnegative Green's function for the Laplacian on $M$, so that $\Delta_xG(x,y)=\frac{1}{\operatorname{vol}(M)}-\delta_y$, and set
\begin{equation}\label{hdef}
h_{\epsilon}(x):=\int_M G(x,y)\epsilon|F_{\nabla}|(y)\,dy\ge 0,
\end{equation}
so that 
\begin{equation}\label{hchar2}
\Delta h_{\epsilon}(x)=\frac{1}{\operatorname{vol}(M)}\|\epsilon F_{\nabla}\|_{L^1}-\epsilon |F_{\nabla}|(x).
\end{equation}
Taking $C'$ to be the constant appearing in \eqref{balsubeq}, for the difference $\xi_{\epsilon}-C'h_{\epsilon}$, we then have 
\begin{align*}
\Delta (\xi_{\epsilon}-C'h_{\epsilon})
&\geq\frac{|u|^2}{\epsilon^2}(\xi_{\epsilon}-C'h_{\epsilon})+C'\frac{|u|^2}{\epsilon^2}h_{\epsilon}-C'\frac{\|\epsilon F_{\nabla}\|_{L^1}}{\operatorname{vol(M)}}\\
&\geq\frac{|u|^2}{\epsilon^2}(\xi_{\epsilon}-C'h_{\epsilon})-C'\frac{\|\epsilon F_{\nabla}\|_{L^1}}{\operatorname{vol}(M)}.
\end{align*}
Observe that the $L^1$ norm of $\xi_\epsilon-C'h_\epsilon$ is bounded by the energy:
\begin{align*}
	\|\xi_\epsilon-C'h_\epsilon\|_{L^1}&\le\|\xi_\epsilon\|_{L^1}+C(M)\|h_\epsilon\|_{L^1} \\
	&\le\|\xi_\epsilon\|_{L^1}+C(M)\|\epsilon F_\nabla\|_{L^1} \\
	&\le C(M)E_{\epsilon}(u,\nabla)^{1/2}.
\end{align*}
Thus, applying Moser iteration to the positive part $(\xi_{\epsilon}-C'h_{\epsilon})^+$, we deduce that
\begin{equation}\label{xihbd}
\xi_{\epsilon}-C'h_\epsilon\leq C(M)E(u,\nabla)^{1/2}.
\end{equation}
(Where the constant $C(M)$ may of course change from line to line.)

As a simple application of \eqref{xihbd}, we note that by definition \eqref{hdef} of $h_{\epsilon}$ and the standard estimate (see, e.g., \cite[Chapter~4]{Aub})
$$G(x,y)\leq C(M)d(x,y)^{2-n}$$
if $n\ge 3$ (or $G(x,y)\leq -C(M)\log(d(x,y))+C(M)$ if $n=2$), we have the $L^{\infty}$ estimate
$$\|h_{\epsilon}\|_{L^{\infty}}\leq C(M) \|\epsilon F_{\nabla}\|_{L^{n-1}}. $$
If $n=2$, this inequality and \eqref{xihbd} give a pointwise bound
\begin{align*}
	&\|\xi_\epsilon^+\|_{L^\infty}\le C(M)\|\epsilon F_\nabla\|_{L^1}+C(M)E_{\epsilon}(u,\nabla)^{1/2}\le C(M)E_{\epsilon}(u,\nabla)^{1/2}.
\end{align*}
In the sequel, we assume $n\ge 3$ and aim for a similar pointwise bound. We have
$$\|h_{\epsilon}\|_{L^{\infty}}
\leq C(M)\|\epsilon F_{\nabla}\|_{L^{n-1}}
\leq C \epsilon \|F_{\nabla}\|_{L^{\infty}}^{\frac{n-3}{n-1}}\|F_{\nabla}\|_{L^2}^{\frac{2}{n-1}}.$$
Using this in \eqref{xihbd},  we compute at a maximum point for $|F_{\nabla}|$ to see that 
$$\|\epsilon F_{\nabla}\|_{L^{\infty}}-\frac{1}{2\epsilon}(1-|u|^2)=\xi_{\epsilon}\leq C\|\epsilon F_{\nabla}\|_{L^{\infty}}^{\frac{n-3}{n-1}}E_{\epsilon}(u,\nabla)^{\frac{1}{n-1}}+CE_{\epsilon}(u,\nabla)^{1/2},$$
and, by an application of Young's inequality, it follows that 
$$(1-C\delta)\|\epsilon F_{\nabla}\|_{L^{\infty}}\leq \frac{1}{2\epsilon}+C \delta^{\frac{3-n}{2}}E_{\epsilon}(u,\nabla)^{1/2}$$
for any $\delta\in (0,1)$. Taking $\delta=\epsilon^{2/n}$, we arrive at the crude preliminary estimate
\begin{align*}
\|\epsilon F_{\nabla}\|_{L^{\infty}}
&\leq\frac{1}{(1-C\epsilon^{2/n})}\left(\frac{1}{2\epsilon}+C \epsilon^{3/n}\epsilon^{-1}E_{\epsilon}(u,\nabla)\right)^{1/2}\\
&\leq\frac{1}{2\epsilon}+\frac{\alpha(\epsilon)}{2\epsilon}(1+E_{\epsilon}(u,\nabla)^{1/2}),
\end{align*}
where $\alpha(\epsilon)\to 0$ as $\epsilon\to 0$.
Now, consider the function 
$$f:=\epsilon |\omega|-\frac{1+\alpha(\epsilon)(1+E_{\epsilon}(u,\nabla)^{1/2})}{2\epsilon}(1-|u|^2).$$
By virtue of the preceding estimate for $\|F_{\nabla}\|_{L^{\infty}}$, we then see that
$$f\leq \frac{1+\alpha(\epsilon)(1+ E_{\epsilon}(u,\nabla)^{1/2})}{2\epsilon}|u|^2$$
pointwise. Appealing once again to \eqref{moduboch} and \eqref{omega1boch2}, we see that
$$\Delta f\geq \frac{|u|^2}{\epsilon^2}f-C\epsilon |F_{\nabla}|,$$
so at a point where $f$ achieves its maximum we have
$$\frac{|u|^2}{\epsilon^2} f\leq C\epsilon |F_{\nabla}|\leq \frac{C(1+E_{\epsilon}(u,\nabla)^{1/2})}{\epsilon}.$$
On the other hand, we know that $|u|^2\geq\frac{\epsilon}{C(1+E_{\epsilon}(u,\nabla)^{1/2})}f$ everywhere, so the preceding computations yield an estimate of the form
$$\frac{(\max f)^2}{\epsilon}\leq \frac{C(M,E_{\epsilon}(u,\nabla))}{\epsilon},$$
and we deduce that $f\leq C(M,E_{\epsilon}(u,\nabla))$ everywhere. Putting all this together, we arrive at the following lemma.
\begin{lem}\label{xiptwise1}
Let $(u,\nabla)$ solve \eqref{ueq} and \eqref{feq} on a line bundle $L\to M$, and suppose $E_{\epsilon}(u,\nabla)\leq \Lambda$. Then there exists a constant $C(M,\Lambda)$ and a function $\alpha(M,\Lambda,\epsilon)$, with $\alpha(\epsilon)\to 0$ as $\epsilon\to 0$, such that
\begin{equation}
\xi_{\epsilon}\leq \alpha(\epsilon)\frac{(1-|u|^2)}{\epsilon}+C.
\end{equation}
\end{lem}
In the next section, we will improve the rough preliminary estimate of Lemma \ref{xiptwise1} to a uniform pointwise bound of the form $\xi_{\epsilon}\leq C(M,\Lambda)$, but this will require some additional effort.

\section{Inner variations and improved monotonicity}\label{monosec}

In this section, we derive the inner variation equation for solutions of \eqref{ueq}--\eqref{feq}, and explore the scaling properties of the energy $E_{\epsilon}(u_{\epsilon},\nabla_{\epsilon})$ over balls of small radius. Under the assumption that the curvature operator $\mathcal{R}_2$ appearing in \eqref{omega1boch2} is positive-definite (so that \eqref{ptwisebal} holds), the analysis simplifies considerably, leading with little effort to the desired monotonicity of the $(n-2)$-energy density. Without this curvature assumption, more work is required, first building on the cruder estimates of the preceding section to obtain a uniform pointwise bound for $\xi_{\epsilon}$.

Fixing notation, with respect to a local orthonormal basis $\{e_i\}$ for $TM$, define the $(0,2)$-tensors $\nabla u^*\nabla u$ and $\omega^*\omega$ by
\begin{align}
(\nabla u^*\nabla u)(e_i,e_j)&:=\langle \nabla_{e_i}u,\nabla_{e_j}u\rangle, \\[5pt]
\omega^*\omega(e_i,e_j)&:=\sum_{k=1}^n\omega(e_i,e_k)\omega(e_j,e_k).
\end{align}
Note that $\operatorname{tr}(\nabla u^*\nabla u)=|\nabla u|^2$ and $\operatorname{tr}(\omega^*\omega)=2|\omega|^2$. Denote by $e_{\epsilon}(u,\nabla)$ the energy integrand
$$e_{\epsilon}(u,\nabla):=|\nabla u|^2+\epsilon^2|F_{\nabla}|^2+\frac{W(u)}{\epsilon^2}.$$
The fact that $d\omega=0$ reads
\begin{align*}
	&D\omega(e_i,e_j)=D_{e_i}\omega(\cdot,e_j)+D_{e_j}\omega(e_i,\cdot),
\end{align*}
where $D$ is the Levi--Civita connection of $M$.
Using this identity, it is straightforward to check that
\begin{align*}
d e_{\epsilon}(u,\nabla)&=2\operatorname{div}(\nabla u^*\nabla u)+2\langle \nabla u,\nabla^*\nabla u\rangle+d\frac{W(u)}{\epsilon^2}\\
&\quad+2\omega(\langle iu,\nabla u\rangle^\#,\cdot)+2\epsilon^2\operatorname{div}(\omega^*\omega)-2\epsilon^2\omega((d^*\omega)^\#,\cdot).
\end{align*}
In particular, defining the stress-energy tensor $T_{\epsilon}(u,\nabla)$ by
\begin{equation}
T_{\epsilon}(u,\nabla):=e_{\epsilon}(u,\nabla)g-2\nabla u^*\nabla u-2\epsilon^2\omega^*\omega,
\end{equation}
for $(u,\nabla)$ solving \eqref{ueq} and \eqref{feq} it follows that
\begin{equation}\label{stateq1}
\operatorname{div}(T_{\epsilon}(u,\nabla))=0,
\end{equation}
meaning that $\sum_i (D_{e_i}T)(e_i,\cdot)=0$.
Integrating \eqref{stateq1} against a vector field $X$ on some domain $\Omega\subseteq M$, we arrive at the usual inner-variation equation
\begin{equation}\label{stateq2}
\int_{\Omega}\langle T_{\epsilon}(u,\nabla),DX\rangle=\int_{\partial\Omega}T_{\epsilon}(u,\nabla)(X,\nu),
\end{equation}
where we identify $T_\epsilon(u,\nabla)$ with a $(1,1)$-tensor and denote by $\nu$ the outer unit normal to $\Omega$. Taking $\Omega=B_r(p)$ to be a small geodesic ball of radius $r$ about a point $p\in M$, and taking $X=\operatorname{grad}(\frac{1}{2}d_p^2)$, where $d_p$ is the distance function to $p$, \eqref{stateq2} gives
\begin{align*}
r\int_{\partial B_r(p)}(e_{\epsilon}(u,\nabla)-2|\nabla_{\nu}u|^2-2\epsilon^2|\iota_{\nu}\omega|^2)
&=\int_{B_r(p)}\langle T_{\epsilon}(u,\nabla),DX\rangle\\
&=\int_{B_r(p)}\langle T_{\epsilon}(u),g\rangle+\int_{B_r(p)}\langle T_{\epsilon}(u),DX-g\rangle\\
&=\int_{B_r(p)}(n e_{\epsilon}(u,\nabla)-2|\nabla u|^2-4\epsilon^2|F_{\nabla}|^2)\\
&\quad+\int_{B_r(p)}\langle T_{\epsilon}(u),DX-g\rangle.
\end{align*}
Now, by the Hessian comparison theorem, we know that
$$|DX-g|\leq C(M)d_p^2;$$
applying this in the relations above, we see that
\begin{align*}
r\int_{\partial B_r(p)}e_{\epsilon}(u,\nabla)
&\geq 2r\int_{\partial B_r(p)}(|\nabla_{\nu}u|^2+\epsilon^2|\iota_{\nu}\omega|^2)\\
&\quad +\int_{B_r(p)}\Big((n-2)|\nabla u|^2+(n-4)\epsilon^2|F_{\nabla}|^2+n\frac{W(u)}{\epsilon^2}\Big)\\
&\quad -C'(M) r^2\int_{B_r(p)}e_{\epsilon}(u,\nabla).
\end{align*}

Setting 
\begin{equation}
f(p,r):=e^{C' r^2}\int_{B_r(p)}e_{\epsilon}(u,\nabla),
\end{equation}
it follows from the computations above (temporarily throwing out the additional nonnegative boundary terms) that
\begin{equation}\label{fderivcomp}
\frac{\partial f}{\partial r}\geq \frac{e^{C' r^2}}{r}\int_{B_r(p)}\Big((n-2)|\nabla u|^2+(n-4)\epsilon^2|F_{\nabla}|^2+n\frac{W(u)}{\epsilon^2}\Big)
\end{equation}
At this point, one easily observes that the right-hand side of \eqref{fderivcomp} is bounded below by $\frac{n-4}{r}f(p,r)$, to obtain the monotonicity of the $(n-4)$-energy density
$$\frac{\partial}{\partial r}(r^{4-n}f(p,r))\geq 0.$$
For general Yang--Mills and Yang--Mills--Higgs problems, this codimension-four energy growth is well known to be sharp (cf., e.g., \cite{SmUhl}, \cite{Zhang}). For solutions of \eqref{ueq} and \eqref{feq} on Hermitian line bundles, however, we show now that this can be improved to (near-) monotonicity of the $(n-2)$-density $r^{2-n}f(p,r)$ on small balls, which constitutes a key technical ingredient in the proof of Theorem \ref{limthm}.

To begin, we rearrange \eqref{fderivcomp}, to see that
\begin{align*}
\frac{\partial f}{\partial r}
&\geq \frac{n-2}{r}f(r)+\frac{2e^{C' r^2}}{r}\int_{B_r(p)}\Big(\frac{W(u)}{\epsilon^2}-\epsilon^2|F_{\nabla}|^2\Big)\\
&=\frac{n-2}{r}f(r)-\frac{2e^{C' r^2}}{r}\int_{B_r(p)}\xi_{\epsilon}\Big(\epsilon |F_{\nabla}|+\frac{1}{2\epsilon}(1-|u|^2)\Big),
\end{align*}
recalling the notation $\xi_{\epsilon}:=\epsilon|F_{\nabla}|-\frac{1}{2\epsilon}(1-|u|^2)$. Now, by Lemma \ref{xiptwise1}, assuming $E_{\epsilon}(u,\nabla)\leq \Lambda$, we have the pointwise bound
\begin{align*}
\xi_{\epsilon}\Big(\epsilon |F_{\nabla}|+\frac{1}{2\epsilon}(1-|u|^2)\Big)
&\leq \left(C+\alpha(\epsilon)\frac{1-|u|^2}{\epsilon}\right)e_{\epsilon}(u,\nabla)^{1/2}\\
&\leq Ce_{\epsilon}(u,\nabla)^{1/2}+C\alpha(\epsilon)e_{\epsilon}(u,\nabla).
\end{align*}

Applying this in our preceding computation for $\frac{\partial f}{\partial r}$, we deduce that
\begin{align*}
\frac{\partial f}{\partial r}
&\geq \frac{n-2}{r}f(r)-\frac{e^{C'r^2}}{r}\int_{B_r(p)}Ce_\epsilon(u,\nabla)^{1/2}-\alpha(\epsilon)\frac{e^{C'r^2}}{r}\int_{B_r(p)}Ce_{\epsilon}(u,\nabla)\\
&\geq \frac{n-2-C\alpha(\epsilon)}{r}f(r)-\frac{e^{C'r^2}}{r}Cr^{n/2}\Big(\int_{B_r(p)}e_\epsilon(u,\nabla)\Big)^{1/2}\\
&\geq \frac{n-2-C''\alpha(\epsilon)}{r}f(r)-C''r^{n/2-1}f(r)^{1/2}
\end{align*}
for some constant $C''(M,\Lambda)$ and $0<r<c(M)$.
Taking $\epsilon$ sufficiently small, we arrive next at the following coarse estimate for the $(n-3)$-energy density, which we will then use to establish an improved bound for $\xi_{\epsilon}$.
\begin{lem}\label{crudefbd}
For $\epsilon\le\epsilon_m(M,\Lambda)$ sufficiently small, we have a uniform bound
\begin{equation}\label{crudedensest}
\sup_{r>0}r^{3-n}\int_{B_r(p)}e_{\epsilon}(u,\nabla)\leq C(M,\Lambda).
\end{equation}
\end{lem}
\begin{proof} The statement is trivial if $n=2,3$, so assume $n\geq 4$. In the preceding computation, take $\epsilon\le\epsilon_m(M,\Lambda)$ sufficiently small that $C''\alpha(\epsilon)<\frac{1}{2}$. Then the estimate gives
$$f'(r)\geq \frac{n-2-1/2}{r}f(r)-C''r^{n/2-1}f(r)^{1/2},$$
from which it follows that, for $0<r<c(M)$,
\begin{align*}
\frac{d}{dr}(r^{3-n}f(r))
&\geq r^{3-n}f'(r)+(3-n)r^{2-n}f(r)\\
&\geq r^{2-n}\Big(\Big(n-\frac{5}{2}\Big)f(r)-Cr^{n/2}f(r)^{1/2}+(3-n)f(r)\Big)\\
&\geq r^{2-n}\Big(\frac{1}{2}f(r)-Cr^{n/2}f(r)^{1/2}\Big).
\end{align*}
If $r^{3-n}f(r)$ has a maximum in $(0,\operatorname{inj}(M))$, it follows that $f(r)\leq C r^{n/2}f(r)^{1/2}$ there, and therefore $r^{3-n}f(r)\leq C r^3\le C$. Obviously the desired estimate holds at $r=0$ and $r=c(M)$, so \eqref{crudedensest} follows.
\end{proof}

With Lemma \ref{crudefbd} in hand, we can now improve the bounds of Lemma \ref{xiptwise1} to a uniform pointwise estimate, as follows.

\begin{prop}\label{xibdprop} Let $(u,\nabla)$ solve \eqref{ueq}--\eqref{feq} on a line bundle $L\to M$, with $\epsilon\le\epsilon_m$ and the energy bound $E_{\epsilon}(u,\nabla)\leq \Lambda$. Then there is a constant $C(M,\Lambda)$ such that
\begin{equation}\label{xiptwise2}
\xi_{\epsilon}:=\epsilon |F_{\nabla}|-\frac{1}{2\epsilon}(1-|u|^2)\leq C(M,\Lambda).
\end{equation}
\end{prop}
\begin{proof}
We can assume $n\ge 3$, as we already obtained the claim for $n=2$ in Section \ref{bochsec}. Recall from that section the function
$$h_{\epsilon}(x):=\int_MG(x,y)\epsilon |F_{\nabla}|(y)dy,$$
where $G$ is the nonnegative Green's function on $M$. As discussed in Section \ref{bochsec}, we can deduce from \eqref{balsubeq} a pointwise estimate of the form
\begin{equation}
\xi_{\epsilon}\leq C(M,\Lambda)h_{\epsilon}+C(M)E_{\epsilon}(u,\nabla)^{1/2}.
\end{equation}
Thus, to arrive at the desired bound \eqref{xiptwise2}, it will suffice to establish a pointwise bound of the same form for $h_{\epsilon}$. 

To this end, recall again that $G(x,y)\leq C(M)d(x,y)^{2-n}$, so that by definition we have
\begin{align*}
h_{\epsilon}(x)&\leq C \int_M d(x,y)^{2-n}\epsilon |F_{\nabla}|(y)\,dy\\
&\leq C\int_M d(x,y)^{2-n}e_{\epsilon}(u,\nabla)^{1/2}(y)\,dy\\
&\leq C\int_M (d(x,y)^{-n+1/2}+d(x,y)^{3-n+1/2}e_{\epsilon}(u,\nabla))\,dy,
\end{align*}
where the last line is a simple application of Young's inequality. Since the integral $\int_M d(x,y)^{-n+1/2}\,dy$ is finite, it follows that
\begin{align*}
h_{\epsilon}(x)&\leq C(M)+C(M)\Lambda+C(M)\int_0^{\operatorname{inj}(M)}r^{3-n+1/2}\Big(\int_{\partial B_r(x)}e_{\epsilon}(u,\nabla)\Big)\,dr\\
&=C(M,\Lambda)+C(M)\int_0^{\operatorname{inj}(M)}\frac{d}{dr}\Big(r^{-n+7/2}\int_{B_r(x)}e_{\epsilon}(u,\nabla)\Big)\,dr\\
&\quad+(n-7/2)C(M)\int_0^{\operatorname{inj}(M)}r^{3-n-1/2}\Big(\int_{B_r(x)}e_{\epsilon}(u,\nabla)\Big)\,dr\\
&\leq C(M,\Lambda)+C(M)\int_0^{\operatorname{inj}(M)}r^{3-n-1/2}\left(\int_{B_r(x)}e_{\epsilon}(u,\nabla)\right)dr.
\end{align*}
On the other hand, by Lemma \ref{crudefbd}, we know that $r^{3-n}\int_{B_r(x)}e_{\epsilon}(u,\nabla)\leq C(M,\Lambda)$ for every $r$, so we see finally that
\begin{align*}
h_{\epsilon}(x)\leq C(M,\Lambda)+C(M,\Lambda)\int_0^{\operatorname{inj}(M)}r^{-1/2}\,dr\leq C(M,\Lambda),
\end{align*}
as desired.
\end{proof}

Applying \eqref{xiptwise2} in our original computation for $f'(r)$, we see now that
\begin{align*}
\frac{\partial f}{\partial r}&\geq \frac{n-2}{r}f(r)-\frac{2e^{C'r^2}}{r}\int_{B_r(p)}\xi_{\epsilon}\Big(\epsilon |F_{\nabla}|+\frac{1}{2\epsilon}(1-|u|^2)\Big)\\
&\geq \frac{n-2}{r}f(r)-\frac{2e^{C'r^2}}{r}\int_{B_r(p)}C(M,\Lambda)e_{\epsilon}(u,\nabla)^{1/2}\\
&\geq \frac{n-2}{r}f(r)-C(M,\Lambda)r^{\frac{n-2}{2}}f(r)^{1/2}.
\end{align*}
In fact, bringing in the extra boundary terms that we have been neglecting, and applying Young's inequality to the term $r^{\frac{n-2}{2}}f(r)^{1/2}$, we see that
\begin{align*}
\frac{\partial f}{\partial r}&\geq 2e^{C'r^2}\int_{\partial B_r(p)}(|\nabla_{\nu}u|^2+\epsilon^2|\iota_{\nu}F_{\nabla}|^2)\\
&\quad +\frac{n-2}{r}f(r)-C r^{\frac{n-2}{2}}f(r)^{1/2}\\
&\geq 2e^{C'r^2}\int_{\partial B_r(p)}(|\nabla_{\nu}u|^2+\epsilon^2|\iota_{\nu}F_{\nabla}|^2)\\
&\quad+\frac{n-2}{r}f(r)-Cf(r)-Cr^{n-2}.
\end{align*}
With this differential inequality in place, a straightforward computation leads us finally to one of our key technical theorems, the monotonicity formula for the $(n-2)$-density.

\begin{thm}\label{monotonicity}
Let $(u,\nabla)$ solve \eqref{ueq}--\eqref{feq} on a Hermitian line bundle $L\to M$, with an energy bound $E_{\epsilon}(u,\nabla)\leq \Lambda$. Then there exists positive constants $\epsilon_m(M,\Lambda)$ and $C_m(M,\Lambda)$ such that the normalized energy density
\begin{equation}
\tilde{E}_{\epsilon}(x,r):=e^{C_m r}r^{2-n}\int_{B_r(x)}e_{\epsilon}(u,\nabla)
\end{equation}
satisfies
\begin{equation}\label{mono.ineq}
\tilde{E}_{\epsilon}'(r)\geq 2r^{2-n}\int_{\partial B_r(x)}(|\nabla_{\nu}u|^2+\epsilon^2|\iota_{\nu}F_{\nabla}|^2)-C_m,
\end{equation}
for $0<r<\operatorname{inj}(M)$ and $\epsilon\le\epsilon_m$.
\end{thm}

As a simple corollary of the monotonicity result (together with a pointwise bound for $|\nabla u|$ derived in the following section), we deduce that $(u,\nabla)$ must have positive $(n-2)$-energy density wherever $|u|$ is bounded away from $1$.

\begin{cor}[clearing-out]\label{clearing.out}
	Let $(u,\nabla)$ solve \eqref{ueq}--\eqref{feq} on a line bundle $L\to M$, with $E_{\epsilon}(u,\nabla)\leq \Lambda$ and $\epsilon\le\epsilon_m$. Given $\delta>0$, if
	\begin{align*}
		&r^{2-n}\int_{B_r(x)}e_\epsilon(u,\nabla)\le\eta(M,\Lambda,\delta)
	\end{align*}
	with $x\in M$ and $\epsilon<r<\operatorname{inj}(M)$, then we must have $|u(x)|>1-\delta$.
\end{cor}

\begin{proof}
	For $\epsilon\le\epsilon_m$, Theorem \eqref{monotonicity} gives
	\begin{align*}
		&\epsilon^{2-n}\int_{B_\epsilon(x)}e_\epsilon(u,\nabla)\le C(M,\Lambda)\eta+C(M,\Lambda)r.
	\end{align*}
	The gradient bound \eqref{nabwcontrol} in Proposition \ref{wcontrolprop} of the following section gives
	$|d|u||\le C\epsilon^{-1}$. Hence, if $|u(x)|\le 1-\delta$ then $|u(y)|<1-\frac{\delta}{2}$ on $B_{\epsilon\delta/(2C)}(x)$, so that $1-|u(y)|^2\ge 1-|u(y)|>\frac{\delta}{2}$. We deduce that
	\begin{align*}
		&\frac{\delta^2}{16}\operatorname{vol}(B_{\epsilon\delta/(2C)}(x))
		\le\int_{B_\epsilon(x)}W(u)
		\le\epsilon^2\int_{B_\epsilon(x)}e_\epsilon(u,\nabla)
		\le C\epsilon^{n}(\eta+r).
	\end{align*}
	Since $\operatorname{vol}(B_{\epsilon\delta/(2C)}(x))$ is bounded below by $c(M,\Lambda,\delta)\epsilon^n$, we can choose $\tilde\eta(M,\Lambda,\delta)\le\operatorname{inj}(M)$ so small that we get a contradiction if $r,\eta\le\tilde\eta$. On the other hand, if $r>\tilde\eta$ then
	\begin{align*}
		&\tilde\eta^{2-n}\int_{B_{\tilde\eta}(x)}e_\epsilon(u,\nabla)
		\le\Big(\frac{\operatorname{inj}(M)}{\tilde\eta}\Big)^{n-2}\eta.
	\end{align*}
	Hence, setting $\eta:=\Big(\frac{\tilde\eta}{\operatorname{inj}(M)}\Big)^{n-2}\tilde\eta\le\tilde\eta$, we can reduce to the previous case (replacing $r$ with $\tilde\eta$), reaching again a contradiction.
\end{proof}

\section{Decay away from the zero set}\label{decaysec}

Again, let $(u,\nabla)$ solve \eqref{ueq}--\eqref{feq} on a line bundle $L\to M$, with the energy bound $E_{\epsilon}(u,\nabla)\leq \Lambda$. In the preceding section, we obtained the pointwise estimate
\begin{equation}\label{bigfbd}
|F_{\nabla}|\leq \frac{1}{2\epsilon^2}(1-|u|^2)+\frac{1}{\epsilon}C(M,\Lambda)
\end{equation}
when $\epsilon\le\epsilon_m$.
As a first step toward establishing strong decay of the energy away from the zero set of $u$, we show in the following proposition that the full energy density $e_{\epsilon}(u,\nabla)$ is controlled by the potential $\frac{W(u)}{\epsilon^2}$. 
\begin{prop}\label{wcontrolprop}
For $(u,\nabla)$ as above, we have the pointwise estimates
\begin{equation}\label{fwcontrol}
\epsilon^2|F_{\nabla}|^2\leq C(M,\Lambda)\frac{W(u)}{\epsilon^2}+C(M,\Lambda)\epsilon
\end{equation}
and
\begin{equation}\label{nabwcontrol}
|\nabla u|^2\leq C(M,\Lambda)\frac{W(u)}{\epsilon^2}+C(M,\Lambda)\epsilon^2,
\end{equation}
provided $\epsilon\le\epsilon_d$, for some $\epsilon_d=\epsilon_d(M,\Lambda)$.
\end{prop}
\begin{proof} To begin, let $C_1=C_1(M,\Lambda)$ be the constant from \eqref{bigfbd}, and consider the function
$$f:=\epsilon |F_{\nabla}|-\frac{1+2C_1\epsilon}{2\epsilon}(1-|u|^2)=\xi_{\epsilon}-C_1+C_1|u|^2.$$
Similar to the proof of Lemma \ref{xiptwise1}, observe that $C_1|u|^2\geq f$ pointwise, by \eqref{bigfbd}, while the computations from Section \ref{bochsec} give
$$\Delta f\geq \frac{|u|^2}{\epsilon^2}f-C'(M)\epsilon |F_{\nabla}|.$$
By \eqref{xiptwise2} we have $|F_\nabla|\le\frac{1}{2\epsilon^2}+C(M,\Lambda)$, so at a maximum for $f$ it follows that
\begin{align*}
0&\geq \frac{|u|^2}{\epsilon^2}f-C'\epsilon |F_{\nabla}|\\
&\geq \frac{f^2}{C_1\epsilon^2}-\frac{C(M,\Lambda)}{\epsilon},
\end{align*}
so that
$$(\max f)^2\leq C\epsilon,$$
and consequently $f\leq C\epsilon^{1/2}$ everywhere. As a consequence, at any point, we have either $f<0$, in which case 
$$\epsilon^2|F_{\nabla}|^2\leq (1+2C_1\epsilon)^2\frac{W(u)}{\epsilon^2},$$
or $f\geq 0$, in which case
\begin{align*}
\epsilon^2|F_{\nabla}|^2&\leq 2f^2+2(1+2C_1\epsilon)^2\frac{W(u)}{\epsilon^2}\\
&\leq 2C_3^2\epsilon+2(1+2C_1\epsilon)^2\frac{W(u)}{\epsilon^2}.
\end{align*}
In either scenario, we obtain a bound of the desired form \eqref{fwcontrol}.

To bound $|\nabla u|^2$, recall from Section \ref{bochsec} the identity
\begin{equation}
\Delta \frac{1}{2}|\nabla u|^2=|\nabla^2u|^2+\frac{1}{2\epsilon^2}(3|u|^2-1)|\nabla u|^2-2\langle \omega,\psi(u)\rangle+\mathcal{R}_1(\nabla u,\nabla u).
\end{equation}
In view of the estimate \eqref{bigfbd} for $|F_{\nabla}|=|\omega|$ and \eqref{psinormbd}, we can estimate the term $2\langle \omega,\psi(u)\rangle$ from above by 
$$2|F_{\nabla}||\nabla u|^2\leq \frac{1}{\epsilon^2}(1-|u|^2)|\nabla u|^2+\frac{C}{\epsilon}|\nabla u|^2,$$
to obtain the existence of $C_1(M,\Lambda)$ such that
$$\frac{1}{2}\Delta |\nabla u|^2\geq |\nabla^2u|^2+\frac{1}{2\epsilon^2}(5|u|^2-3)|\nabla u|^2-\frac{C_1}{\epsilon}|\nabla u|^2.$$
For $\Delta |\nabla u|$, this then gives 
\begin{equation}
\Delta |\nabla u|\geq \frac{1}{2\epsilon^2}(5|u|^2-3)|\nabla u|-\frac{C_1}{\epsilon}|\nabla u|.
\end{equation}

Recalling once again the equation \eqref{moduboch} for $\Delta \frac{1}{2}|u|^2$, we define 
$$w:=|\nabla u|-\frac{1}{\epsilon}(1-|u|^2),$$
and observe that
\begin{align*}
\Delta w&\geq \frac{1}{2\epsilon^2}(5|u|^2-3)|\nabla u|-\frac{C_1}{\epsilon}|\nabla u|\\
&\quad+\frac{2}{\epsilon}|\nabla u|^2-\frac{1}{\epsilon^3}|u|^2(1-|u|^2)\\
&=\frac{|u|^2}{\epsilon^2}w+|\nabla u|\Big(\frac{2}{\epsilon}|\nabla u|-\frac{3}{2}\frac{(1-|u|^2)}{\epsilon^2}-\frac{C_1}{\epsilon}\Big)\\
&=\frac{|u|^2}{\epsilon^2}w+\frac{|\nabla u|}{\epsilon}\Big(2w+\frac{1}{2\epsilon}(1-|u|^2)-C_1\Big).
\end{align*}
We then have
\begin{equation}\label{lapwest}
\Delta w \geq \frac{|u|^2}{\epsilon^2}w+\frac{1}{\epsilon}\Big(w+\frac{1}{\epsilon}(1-|u|^2)\Big)\Big(2w+\frac{1}{2\epsilon}(1-|u|^2)-C_1\Big).
\end{equation}

If $w$ has a positive maximum, it follows that
$$2w+\frac{1}{2\epsilon}(1-|u|^2)\leq C_1$$
at this maximum point; in particular, we deduce then that
$$|u|^2\geq 1-2C_1\epsilon$$
at this point, and see from \eqref{lapwest} that here
$$0\geq \frac{1-2C_1\epsilon}{\epsilon^2}w-\frac{1}{\epsilon}\Big(w+\frac{1}{\epsilon}(1-|u|^2)\Big)C_1\ge\frac{1-3C_1\epsilon}{\epsilon^2}w-2\frac{C_1^2}{\epsilon}.$$
If $\epsilon\le\epsilon_d(M,\Lambda)$ is small enough, it follows that $\max w\leq C\epsilon$; as a consequence, we check that
$$|\nabla u|^2\leq C\frac{W(u)}{\epsilon^2}+C\epsilon^2,$$
completing the proof of \eqref{nabwcontrol}.
\end{proof}

As a simple consequence of the estimates in Proposition \ref{wcontrolprop}, we obtain the following corollary.

\begin{cor}\label{offzsubeqcor} There exist constants $0<\beta_d(M,\Lambda)<1$ and $C(M,\Lambda)$ such that, for $(u,\nabla)$ as above, we have
\begin{equation}\label{offzsubeq}
\Delta \frac{1}{2}(1-|u|^2)\geq \frac{1}{4\epsilon^2}(1-|u|^2)-C\epsilon^2
\end{equation}
on the set $\{|u|^2\geq 1-\beta_d\}$.
\end{cor}

\begin{proof} By the formula \eqref{moduboch} for $\Delta \frac{1}{2}|u|^2$, we know that
$$\Delta \frac{1}{2}(1-|u|^2)=\frac{1}{2\epsilon^2}|u|^2(1-|u|^2)-|\nabla u|^2.$$
Combining this with the estimate \eqref{nabwcontrol} for $|\nabla u|^2$, we then deduce the existence of a constant $\hat C(M,\Lambda)$ such that
$$\Delta \frac{1}{2}(1-|u|^2)\geq |u|^2\frac{1}{2\epsilon^2}(1-|u|^2)-\hat C(M,\Lambda)\frac{(1-|u|^2)^2}{2\epsilon^2}-C\epsilon^2.$$
By taking $\beta=\beta(M,\Lambda)>0$ sufficiently small, we can arrange that
$$|u|^2-\hat C(1-|u|^2)\geq 1-\beta-\hat C\beta\geq \frac{1}{2}$$
on $\{|u|^2\geq 1-\beta\}$, from which the claimed estimate follows.
\end{proof}

Next, we employ the result of Corollary \ref{offzsubeqcor} to show that the quantity $(1-|u|^2)$ vanishes rapidly away from from $Z_{\beta_d}(u)$ (compare \cite[Sections~III.7--III.8]{JT}).

\begin{prop}\label{maindecayest} Let $(u,\nabla)$ be as before,
with $\epsilon\le\epsilon_d$, and define the set
$$Z_{\beta_d}:=\{x\in M : |u(x)|^2\leq 1-\beta_d\},$$
where $\beta_d(M,\Lambda)$ is the constant provided by Corollary \ref{offzsubeqcor}.
Defining $r: M\to [0,\infty)$ by 
$$r(p):=\operatorname{dist}(p,Z_{\beta}),$$
we have an estimate of the form
\begin{equation}
(1-|u|^2)(p)\leq C e^{-a_d r(p)/\epsilon}+C\epsilon^4
\end{equation}
for some $C=C(M,\Lambda)$ and $a_d=a_d(M)>0$.
\end{prop}

\begin{proof}
Fix a point $p\in M$, and let $r=r(p)=\operatorname{dist}(p,Z_{\beta})$ as above. We can clearly assume $r(p)<\frac{1}{2}\operatorname{inj}(M)$. On the ball $B_r(p)$, for some constant $a=a_d>0$ to be chosen later, consider the function
$$\varphi(x):=e^{(a/\epsilon)(d_p(x)^2+\epsilon^2)^{1/2}},$$
where $d_p(x):=\operatorname{dist}(p,x)$.
A straightforward computation then gives
\begin{align*}
\Delta\varphi&=\frac{a}{\epsilon}\varphi\left(\frac{(a/\epsilon)d_p^2}{d_p^2+\epsilon^2}-\frac{d_p^2}{(d_p^2+\epsilon^2)^{3/2}}\right)\\
&\quad+\frac{a}{2\epsilon}\varphi\frac{\Delta d_p^2}{(d_p^2+\epsilon^2)^{1/2}}\\
&\leq \frac{a^2}{\epsilon^2}\varphi+\frac{a}{2\epsilon}\varphi\frac{\Delta d_p^2}{(d_p^2+\epsilon^2)^{1/2}}\\
&\leq \frac{a^2+C_1a}{\epsilon^2}\varphi
\end{align*}
for some $C_1=C_1(M)$.
Now, fix some constant $c_2>0$ to be chosen later, and let
$$f:=\frac{1}{2}(1-|u|^2)-c_2\varphi.$$
Combining the preceding computation with \eqref{offzsubeq}, we see that, on $B_r(p)$,
\begin{align*}
\Delta f&\geq \frac{1}{4\epsilon^2}(1-|u|^2)-C(M,\Lambda)\epsilon^2-\frac{a^2+C_1a}{\epsilon^2}c_2\varphi\\
&= \frac{1}{2\epsilon^2}f+\frac{1-2a^2-2C_1a}{2\epsilon^2}c_2\varphi-C(M,\Lambda)\epsilon^2.
\end{align*}
Choosing $a=a_d(M)>0$ sufficiently small, we can arrange that $2a^2+C_1a\leq 1$, so that the above computation gives
\begin{equation}\label{littlefsubeq}
\Delta f\geq \frac{f}{2\epsilon^2}-C\epsilon^2.
\end{equation}

On the boundary of the ball $\partial B_r(p)$, it follows from definition of $r=r(p)$ that $|u|^2\geq 1-\beta_d$, and therefore 
$$f(x)\leq \frac{\beta_d}{2}-c_2\varphi\leq \frac{\beta_d}{2}-c_2e^{ar/\epsilon}\text{ on }\partial B_r(p).$$
Taking $c_2(M,\Lambda):=\beta_d e^{-ar/\epsilon}$, it then follows that $f<0$ on $\partial B_r(p)$, so we can apply the maximum principle with \eqref{littlefsubeq} to deduce that
$$f\leq C\epsilon^4\text{ in }B_r(p).$$
Evaluating at $p$, this gives
$$C\epsilon^4\geq f(p)=\frac{1}{2}(1-|u|^2)(p)-\beta_d e^{-ar(p)/\epsilon}e^{a},$$
so that
$$(1-|u|^2)(p)\leq C(M,\Lambda)e^{-a r(p)/\epsilon}+C(M,\Lambda)\epsilon^4,$$
as desired.
\end{proof} 

Combining these estimates with those of Proposition \ref{fwcontrol}, we arrive immediately at the following decay estimate for the energy integrand $e_{\epsilon}(u,\nabla)$.

\begin{cor}\label{decaycor} Defining $Z_{\beta_d}$ and $r(p)=\operatorname{dist}(p,Z_{\beta_d})$ as in Corollary \ref{offzsubeqcor}, there exist $a_d(M)>0$ and $C_d(M,\Lambda)$ such that
\begin{equation}
e_{\epsilon}(u,\nabla)(p)\leq C_d\frac{e^{-a_dr(p)/\epsilon}}{\epsilon^2}+C_d\epsilon.
\end{equation}
\end{cor}

\section{The energy-concentration varifold}\label{varifold.section}

This section is devoted to the proof of the main result of the paper, which we recall now.

\begin{thm}\label{thm1} Let $(u_{\epsilon},\nabla_{\epsilon})$ be a family of solutions to \eqref{ueq}--\eqref{feq} satisfying a uniform energy bound $E_{\epsilon}(u_{\epsilon},\nabla_{\epsilon})\leq \Lambda$ as $\epsilon\to 0$. Then, as $\epsilon\to 0$, the energy measures 
$$\mu_{\epsilon}:=\frac{1}{2\pi}e_{\epsilon}(u_{\epsilon},\nabla_{\epsilon})\,\operatorname{vol}_g$$
converge subsequentially, in duality with $C^0(M)$, to the weight measure of a stationary, integral $(n-2)$-varifold $V$. Also, for all $0\le\delta<1$,
$$\operatorname{spt}(V)=\lim_{\epsilon \to 0}\{|u_{\epsilon}|\le\delta\}$$
in the Hausdorff topology. If $M$ is oriented, the $(n-2)$-currents dual to the curvature forms $\frac{1}{2\pi}\omega_{\epsilon}$ converge subsequentially to an integral $(n-2)$-cycle $\Gamma$, with $|\Gamma|\le\mu$.
\end{thm}

\subsection{Convergence to a stationary rectifiable varifold}\hfill

Let $(u_{\epsilon},\nabla_{\epsilon})$ be as in Theorem \ref{thm1}, and pass to a subsequence $\epsilon_j\to 0$ such that the energy measures $\mu_{\epsilon_j}$ converge weakly-* to a limiting measure $\mu$,
in duality with $C^0(M)$.

Note that, for $0<r<R<\operatorname{inj}(M)$, Theorem \ref{monotonicity} yields
\begin{align*}
	e^{CR}R^{2-n}\mu(\bar B_R(x))+CR
	&\ge\limsup_{\epsilon\to 0}e^{CR}R^{2-n}\mu_\epsilon(\bar B_R(x))+CR \\
	&\ge\liminf_{\epsilon\to 0}e^{Cr}r^{2-n}\mu_\epsilon(B_r(x))+Cr \\
	&\ge e^{Cr}r^{2-n}\mu(B_r(x))+Cr,
\end{align*}
so approximating $R$ with smaller radii we deduce
\begin{align}\label{mono.mu}
	&e^{CR}R^{2-n}\mu(B_R(x))+CR\ge e^{Cr}r^{2-n}\mu(B_r(x))+Cr
\end{align}
and in particular the $(n-2)$-density
\begin{align*}
	&\Theta_{n-2}(\mu,x):=\lim_{r\to 0}(\omega_{n-2}r^{n-2})^{-1}\mu(B_r(x))
\end{align*}
is defined.
As a first step toward the proof of Theorem \ref{thm1}, we show that this density is bounded from above and below on the support $\operatorname{spt}(\mu)$.

\begin{prop}\label{densprop} There exists a constant $0<C(M,\Lambda)<\infty$ such that
% I IMPROVED SLIGHTLY THE STATEMENT FOR LATER PURPOSES
\begin{equation}\label{mudensbds}
C(M,\Lambda)^{-1}\leq r^{2-n}\mu(B_r(x))\leq C(M,\Lambda)\quad\text{for }x\in \operatorname{spt}(\mu),\ 0<r<\operatorname{inj}(M)
\end{equation}
and thus $C(M,\Lambda)^{-1}\leq\Theta_{n-2}(\mu,x)\leq C(M,\Lambda)$ for all $x\in \operatorname{spt}(\mu)$.
\end{prop}

\begin{proof} The upper bound follows fairly immediately from the monotonicity formula in Theorem \ref{monotonicity}. In particular, for any $0<r<\operatorname{inj}(M)$, note that
$$r^{2-n}\mu_{\epsilon}(B_r(x))\leq \tilde{E}_{\epsilon}(x,r),$$
where $\tilde{E}_{\epsilon}(x,r):=e^{C_m r}r^{2-n}\int_{B_r(x)}e_\epsilon(u_\epsilon,\nabla_\epsilon)$,
and by Theorem \ref{monotonicity} we have
\begin{align*}
\tilde{E}_{\epsilon}(x,r)&=\tilde{E}_{\epsilon}(x,\operatorname{inj}(M))-\int_r^{\operatorname{inj}(M)}\tilde{E}_{\epsilon}'(s)ds\\
&\leq \tilde{E}_{\epsilon}(x,\operatorname{inj}(M))+\int_r^{\operatorname{inj}(M)}C_m\\
&\leq C(M,\Lambda)\Lambda+C\operatorname{inj}(M),
\end{align*}
so that
$$r^{2-n}\mu_{\epsilon}(B_r(x))\leq C'(M,\Lambda)$$
for all $0<r<\operatorname{inj}(M)$. %Taking the limit as $\epsilon \to 0$ and the $\limsup$ as $r\to 0$ then gives the upper bound in \eqref{mudensbds}.

To see the lower bound, let $\beta=\beta(M,\Lambda)\in (0,1)$ be the constant given by Corollary \ref{offzsubeqcor}, and again set
$$Z_{\beta}(u_{\epsilon}):=\{x\in M : |u_{\epsilon}(x)|^2\leq 1-\beta\}.$$
Let $\Sigma$ be the set of all limits $x=\lim_{\epsilon}x_\epsilon$, with $x_\epsilon\in Z_\beta(u_\epsilon)$; that is, take
$$\Sigma:=\bigcap_{\eta>0}\,\overline{\bigcup_{0<\epsilon<\eta}Z_\beta(u_\epsilon)}.$$
We then claim that
\begin{equation}\label{sptmuchar}
\operatorname{spt}(\mu)\subseteq \Sigma
\end{equation}
and
\begin{equation}\label{sigmadenslbd}
\mu(B_r(x))\ge c(M,\Lambda)r^{n-2}\quad\text{for }x\in\Sigma,\ 0<r<\operatorname{inj}(M)
\end{equation}
Once both \eqref{sptmuchar} and \eqref{sigmadenslbd} are established, the lower bound in \eqref{mudensbds} follows immediately.

To establish \eqref{sptmuchar}, fix some $p\in M\setminus \Sigma$; by definition of $\Sigma$, there must exist $\delta=\delta(p)>0$ such that 
$$\operatorname{dist}(p,Z_{\beta}(u_{\epsilon}))\geq 2\delta$$
for all $\epsilon$ sufficiently small. Applying Corollary \ref{decaycor} for all $x\in B_{\delta}(p)$, we deduce that
\begin{align*}
\mu(B_{\delta}(p))&\leq \lim_{\epsilon \to 0}\int_{B_{\delta}(p)}e_{\epsilon}(u_{\epsilon},\nabla_{\epsilon})\\
&\leq \lim_{\epsilon\to 0}\int_{B_{\delta}(p)}(C\epsilon^{-2}e^{-a \delta/\epsilon}+C\epsilon)\\
&=0.
\end{align*}
In particular, $p\notin M\setminus \operatorname{spt}(\mu)$, confirming \eqref{sptmuchar}.

To see \eqref{sigmadenslbd}, let $x\in \Sigma$. Note that, by definition of $\Sigma$, there exist points $x_{\epsilon}\in Z_{\beta}(u_{\epsilon})$ with $x_{\epsilon}\to x$ as $\epsilon \to 0$ (along a subsequence). We then see that
$$|u_{\epsilon}(x_{\epsilon})|^2\leq 1-\beta$$
and Corollary \ref{clearing.out} gives $c(M,\Lambda)$ such that
$$\mu_\epsilon(B_r(x_\epsilon))\ge c(M,\Lambda)r^{n-2}$$
for $\epsilon<r<\operatorname{inj}(M)$. Since for any $\delta>0$ we have $B_r(x_\epsilon)\subseteq\bar B_{r+\delta}(x)$ eventually,
it follows that $\mu(\bar B_{r+\delta}(x))\ge cr^{n-2}$, hence
$$\mu(B_{r}(x))\ge cr^{n-2}$$
for $0<r<\operatorname{inj}(M)$, which is \eqref{sigmadenslbd}.
%replacing $c$ with $\max\{c,\operatorname{inj}(M)^{2-n}c^{n-1}\}$.
%while by Proposition \ref{wcontrolprop}, we have
%$$|d|u||\leq |\nabla u|\leq C(M,\Lambda)/\epsilon.$$
%As a consequence, there must exist some $\delta(M,\Lambda)>0$ such that 
%$$|u_{\epsilon}(x')|^2\leq 1-\beta/2\text{ for all }x'\in B_{\delta \epsilon}(x_{\epsilon}),$$
%and in particular, it follows that
%\begin{eqnarray*}
%\tilde{E}_{\epsilon}(u_{\epsilon},x_{\epsilon},\delta \epsilon)&\geq &c(M,\Lambda) \epsilon^{2-n}\int_{B_{\delta \epsilon}(x_{\epsilon})}\frac{W(u_{\epsilon})}{\epsilon^2}\\
%&\geq & c \epsilon^{2-n}\int_{B_{\delta \epsilon}(x_{\epsilon})}\frac{\beta^2}{16\epsilon^2}\\
%&\geq & c_2(M,\Lambda)>0.
%\end{eqnarray*}
%By Theorem \ref{monothm}, we then deduce for any $r\in (\delta\epsilon,\operatorname{inj}(M))$ that
%$$r^{2-n}\mu_{\epsilon}(B_r(x_{\epsilon}))\geq c_3(M,\Lambda)-C_0(M,\Lambda) r.$$
%Finally, fixing an arbitrary $0<r<\frac{1}{2C_0(M,\Lambda)}$ and taking the limit as $\epsilon \to 0$, we arrive at the estimate
%$$r^{2-n}\mu(B_r(x))\geq c_4(M,\Lambda)>0,$$
%and taking $r\to 0$, we arrive finally at the desired lower bound \eqref{sigmadenslbd}.
\end{proof}

With Proposition \ref{densprop} in place, we will invoke a result by Ambrosio and Soner \cite{AS} to conclude that the limiting measure $\mu=\lim_{\epsilon\to 0}\mu_{\epsilon}$ coincides with the weight measure of a stationary, rectifiable $(n-2)$-varifold. Recall from Section \ref{monosec} the stress-energy tensors
$$T_{\epsilon}=e_{\epsilon}(u_{\epsilon},\nabla_{\epsilon})g-2\nabla_{\epsilon}u_{\epsilon}^*\nabla_{\epsilon}u_{\epsilon}-2\epsilon^2F_{\nabla_{\epsilon}}^*F_{\nabla_{\epsilon}}.$$
We record first the following lemma; in its statement, we canonically identify (and pair with each other) tensors of rank $(2,0)$, $(1,1)$, and $(0,2)$, using the underlying metric $g$.

\begin{lem}\label{tenslim} As $\epsilon\to 0$, the tensors $T_{\epsilon}$ converge (subsequentially) as $\operatorname{Sym}(TM)$-valued measures (in duality with $C^0(M,\operatorname{Sym}(TM))$) to a limit $T$ satisfying
\begin{equation}\label{tstat}
\langle T,DX\rangle=0\text{ for all }C^1\text{ vector fields }X\in C^1(M,TM),
\end{equation}
\begin{equation}\label{ttrace}
\langle T, \varphi g\rangle\geq (n-2)\langle \mu,\varphi\rangle\text{ for every }0\leq \varphi\in C^0(M),
\end{equation}
and
\begin{equation}\label{tdommu}
-\int_M |X|^2\,d\mu\leq \langle T,X\otimes X\rangle\leq \int_M |X|^2\,d\mu\text{ for all }X\in C^0(M,TM).
\end{equation}
\end{lem}
\begin{proof} For each $\epsilon>0$, note that, by definition of $T_{\epsilon}$, for every continuous vector field $X\in C^0(M,TM)$, we have 
\begin{align*}
\int_M\langle T_{\epsilon},X\otimes X\rangle&=\int_M e_{\epsilon}(u_{\epsilon},\nabla_{\epsilon})|X|^2\\
&\quad-\int_M 2|(\nabla_{\epsilon})_X u_{\epsilon}|^2\\
&\quad-\int_M 2\epsilon^2|\iota_X F_{\nabla_{\epsilon}}|^2.
\end{align*}
Evaluating \eqref{two.form.conv} in an orthonormal basis such that $X$ is a multiple of $e_1$, we see that $|\iota_X F_{\nabla_{\epsilon}}|^2\le|F_{\nabla_\epsilon}|^2|X|^2$, while $|(\nabla_{\epsilon})_X u_{\epsilon}|^2\le|\nabla_{\epsilon} u_{\epsilon}|^2|X|^2$. We deduce that
\begin{equation}\label{pretdommu}
-\int_M |X|^2e_{\epsilon}(u_{\epsilon},\nabla_{\epsilon})\leq \int_M \langle T_{\epsilon},X\otimes X\rangle\leq \int_Me_{\epsilon}(u_{\epsilon},\nabla_{\epsilon})|X|^2.
\end{equation}
As an immediate consequence, we see that the uniform energy bound $E_{\epsilon}(u_{\epsilon},\nabla_{\epsilon})\leq \Lambda$ gives a uniform bound on $\|T_{\epsilon}\|_{(C^0)^*}$ as $\epsilon\to 0$, so we can indeed extract a weak-* subsequential limit $T\in C^0(M,\operatorname{Sym}(TM))^*$, for which \eqref{tdommu} follows from \eqref{pretdommu}.

The stationarity condition \eqref{tstat} for the limit $T$ follows from \eqref{stateq2}. It remains to establish the trace inequality \eqref{ttrace}. For this, we simply compute, for nonnegative $\varphi\in C^0(M)$,
\begin{align*}
\int_M \langle T_{\epsilon},\varphi g\rangle&=\int_M \varphi(ne_{\epsilon}(u_{\epsilon},\nabla_{\epsilon})-2|\nabla_{\epsilon}(u_{\epsilon})|^2-4\epsilon^2|F_{\nabla_{\epsilon}}|^2)\\
&=\int_M (n-2)\varphi e_{\epsilon}(u_{\epsilon},\nabla_{\epsilon})+2\int_M\varphi \Big(\frac{W(u_{\epsilon})}{\epsilon^2}-\epsilon^2|F_{\nabla_{\epsilon}}|^2\Big)\\
&\geq (n-2)\langle \mu_{\epsilon},\varphi\rangle-2\int_M\varphi e_{\epsilon}(u_{\epsilon},\nabla_{\epsilon})^{1/2}\Big(|F_{\nabla_{\epsilon}}|-\frac{(1-|u_{\epsilon}|^2)}{2\epsilon}\Big)^+.
\end{align*}
Recalling from Proposition \ref{xibdprop} that
$$|F_{\nabla_{\epsilon}}|-\frac{(1-|u_{\epsilon}|^2)}{2\epsilon}\leq C(M,\Lambda),$$
we then see that
\begin{align*}
\langle T,\varphi g\rangle&=\lim_{\epsilon\to 0}\int_M\langle T_{\epsilon},\varphi g\rangle\\
&\geq (n-2)\langle \mu,\varphi\rangle-C\lim_{\epsilon \to 0}\int_M\varphi e_{\epsilon}(u_{\epsilon},\nabla_{\epsilon})^{1/2}.
\end{align*}
In particular, \eqref{ttrace} will follow once we show that $\lim_{\epsilon\to 0}\int_M e_{\epsilon}(u_{\epsilon},\nabla_{\epsilon})^{1/2}=0$. 

But this is straightforward: from Proposition \ref{densprop} we know that for $0<\delta<\operatorname{inj}(M)$ we have 
$$\mu(B_{\delta}(x))\geq c(M,\Lambda)\delta^{n-2}\quad\text{for }x\in \Sigma=\operatorname{spt}(\mu).$$
Since $\operatorname{vol}(B_{5\delta}(x))\le C(M)\delta^n$, a simple Vitali covering argument then implies that the $\delta$-neighborhood $B_{\delta}(\Sigma)$ of $\Sigma$ satisfies a volume bound
$$\operatorname{vol}(B_{\delta}(\Sigma))\leq C(M,\Lambda)\delta^2.$$
With this estimate in hand, we then see that
\begin{align*}
\int_M e_{\epsilon}(u_{\epsilon},\nabla_{\epsilon})^{1/2}&\leq \int_{B_{\delta}(\Sigma)}e_{\epsilon}(u_{\epsilon},\nabla_{\epsilon})^{1/2}+\int_{M\setminus B_{\delta}(\Sigma)}e_{\epsilon}(u_{\epsilon},\nabla_{\epsilon})^{1/2}\\
&\leq \operatorname{vol}(B_{\delta}(\Sigma))^{1/2}\Lambda^{1/2}+\operatorname{vol}(M)^{1/2}\mu_{\epsilon}(M\setminus B_{\delta}(\Sigma))^{1/2}.
\end{align*}
Fixing $\delta$ and taking the limit as $\epsilon\to 0$, we have $\mu_{\epsilon}(M\setminus B_{\delta}(\Sigma))\to 0$. Since $\operatorname{vol}(B_{\delta}(\Sigma))\leq C \delta^2$, we find that
$$\limsup_{\epsilon\to 0}\int_M e_{\epsilon}(u_{\epsilon},\nabla_{\epsilon})^{1/2}\leq C \delta \Lambda^{1/2}.$$
Finally, taking $\delta\to 0$, we conclude that $\int_Me_{\epsilon}(u_{\epsilon},\nabla_{\epsilon})^{1/2}\to 0$ as $\epsilon\to 0$, completing the proof.
\end{proof}

Estimate \eqref{tdommu} says that $|T|$ is absolutely continuous with respect to $\mu$, so by the Radon--Nikodym theorem we can write the limiting $\operatorname{Sym}(TM)$-valued measure $T$ from Lemma \ref{tenslim} as
\begin{equation}
\langle T,S\rangle=\int_M \langle P(x),S(x)\rangle\,d\mu(x)
\end{equation}
for some $L^{\infty}$ (with respect to $\mu$) section $P: M\to \operatorname{Sym}(TM)$. Moreover, it follows from \eqref{ttrace} and \eqref{tdommu} that $-g \leq P(x)\leq g$ and $\operatorname{tr}(P(x))\geq n-2$ at $\mu$-a.e. $x\in M$, so that $T$ defines in a natural way a generalized $(n-2)$-varifold in the sense of Ambrosio and Soner, namely a Radon measure on the bundle
\begin{align}\label{gen.var.def}
	&A_{n,n-2}(M):=\{S\in\operatorname{Sym}(TM):-ng\le S\le g,\,\operatorname{tr}(S)\ge n-2\}.
\end{align}
We refer the reader to \cite[Section~3]{AS}. Note that in \cite{AS} the authors work in the Euclidean space and require the trace to be equal to $n-2$ in \eqref{gen.var.def}; however, the main result on generalized varifolds, namely \cite[Theorem~3.8]{AS}, still holds in our setting (with the same proof).

Hence, in view of the stationary condition \eqref{tstat} and the density bounds of Proposition \ref{densprop}, we can apply \cite[Theorem~3.8(c)]{AS} to conclude that $T$ can be identified with a stationary, rectifiable $(n-2)$-varifold with weight measure $\mu$ (so, in particular, $\operatorname{spt}(\mu)$ is $(n-2)$-rectifiable), and that $P(x)$ is given $\mu$-a.e. by the orthogonal projection onto the $(n-2)$-subspace $T_x\operatorname{spt}(\mu)\subset T_xM$. We collect this information in the following statement.

\begin{prop}\label{varprop1} For a family $(u_{\epsilon},\nabla_{\epsilon})$ satisfying the hypotheses of Theorem \ref{thm1}, after passing to a subsequence $\epsilon_j\to 0$, there exists a stationary, rectifiable $(n-2)$-varifold $V=v(\Sigma^{n-2},\theta)$ such that
\begin{equation}
\lim_{\epsilon\to 0}\int_M\langle T_{\epsilon}(u_{\epsilon},\nabla_{\epsilon}), S\rangle=\int_{\Sigma}\theta(x) \langle T_x\Sigma,S(x)\rangle\,d\mathcal{H}^{n-2}
\end{equation}
for every continuous section $S\in C^0(M,\operatorname{Sym}(TM))$.
In particular, the energy measure $\mu$ is given by $\mu=\theta \mathcal{H}^{n-2}\mrestr\Sigma$. Also, we can choose $\Sigma:=\operatorname{spt}(\mu)$ and $\theta(x):=\Theta_{n-2}(\mu,x)$.
\end{prop}

\subsection{Integrality of the limit varifold and convergence of level sets}\hfill

We now show that the varifold $V$ is integer rectifiable.
Given $x\in\operatorname{spt}(\mu)$ and $s>0$, we define $M_{x,s}$ to be the ball of radius $s^{-1}\operatorname{inj}(M)$ in the Euclidean space $(T_xM,g_x)$ and define $\iota_{x,s}:M_{x,s}\to M$ by $\iota_{x,s}(y):=\exp_x(sy)$. We endow $M_{x,s}$ with the smooth metric $g_{x,s}:=s^{-2}\iota_{x,s}^*g$, which converges locally smoothly to the Euclidean metric $g_x$ as $s\to 0$.
By rectifiability, for $\mu$-a.e. $x$ the dilated varifolds $V_{x,s}:=(\iota_{x,s}^{-1})_*(V\mrestr B_{\operatorname{inj}(M)}(x))$ satisfy
\begin{align}\label{blow-up.shape}
&V_{x,s}\rightharpoonup v(T_x\Sigma,\Theta_{n-2}(x))
\end{align}
as $s\to 0$, in duality with $C_c(\mathbb R^n)$. Fix $x\in\operatorname{spt}(\mu)$ such that \eqref{blow-up.shape} holds.
The integrality of $V$ will follow once we prove that $\Theta_{n-2}(\mu,x)$ is an integer.

We can identify $(T_xM,g_x)$ with $\mathbb R^n$ by a linear isometry such that
$T_x\Sigma=\{0\}\times\mathbb R^{n-2}$. We also call $\mu_{x,s}$ the mass measure of $V_{x,s}$; equivalently, $\mu_{x,s}:=s^{2-n}(\iota_{x,s}^{-1})_*(\mu\mrestr B_{\operatorname{inj}(M)}(x))$.

%\begin{prop}
%	There exists $0<\gamma(M,\Lambda,\delta)<1$ and a Borel set $E$ with $\mu(E)\le C(M,\Lambda)\delta$ such that
%	\begin{align}\label{max.glob.bound}
%		&\int_{B_s(x)\cap\{|u_\epsilon|^2>1-\gamma\}}e_\epsilon(u_\epsilon,\nabla_\epsilon)<C(M,\Lambda)\delta s^{n-2},
%	\end{align}
%	for all $0<s<\operatorname{inj}(M)$ and $x\in\operatorname{spt}\setminus E$, if $\epsilon$ is small enough.
%\end{prop}
%
%\begin{proof}
%	Proposition \ref{glob.bound} gives a $0<\gamma<1$ such that the energy on $\{|u_\epsilon|^2>1-\gamma\}$ is bounded by $C(M,\Lambda)\delta^2$ for $\epsilon$ small enough.
%	Define the measure
%	\begin{align*}
%		&\nu_\epsilon(A):=\int_{A\cap\{|u_\epsilon|^2>1-\gamma\}}e_\epsilon(u_\epsilon,\nabla_\epsilon)
%	\end{align*}
%	and let $E$ be the set of points $x\in\operatorname{spt}(\mu)$ such that $\mu(B_s(x))<\delta^{-1}\nu(B_s(x))$ for some radius $s$.
%	Vitali covering lemma gives disjoint balls $\{B_{s_i}(x_i)\}$ such that $E\subseteq\bigcup_i B_{5s_i}(x_i)$, so
%	by virtue of \eqref{mudensbds}
%	\begin{align*}
%		\mu(E)
%		&\le\sum_i\mu(B_{5s_i}(x_i))
%		\le C(M,\Lambda)\sum_i (5s_i)^{n-2}
%		\le C(M,\Lambda)\sum_i\mu(B_{s_i}(x_i)) \\
%		&\le C(M,\Lambda)\delta^{-1}\nu(M)
%		\le C(M,\Lambda)\delta,
%	\end{align*}
%	while for $x\in\operatorname{spt}(\mu)\setminus E$ we have $\nu(B_s(x))\le\delta\mu(B_s(x))\le C(M,\Lambda)\delta s^{n-2}$.
%\end{proof}

With a diagonal selection, changing our sequence $\epsilon\to 0$ accordingly, we can find scales $s_\epsilon\to 0$ such that
we have the convergence of Radon measures
\begin{align*}
	&\lim_{\epsilon\to 0}\hat\mu_\epsilon=\lim_{s\to 0}\mu_{x,s}=\Theta\mathcal{H}^{n-2}\mrestr T_x\Sigma,
\end{align*}
where $(\hat u_\epsilon,\hat\nabla_\epsilon)$ is the pullback of $(u_{s_\epsilon\epsilon},\nabla_{s_\epsilon\epsilon})$ by means of $\iota_{x,s}$,
and $\hat\mu_\epsilon$ is the associated energy measure.
Note that $(\hat u_\epsilon,\hat\nabla_\epsilon)$ is stationary for $E_\epsilon$ in the line bundle $\iota_{x,s_\epsilon}^*L$, with respect to the base metric $g_{x,s_\epsilon}$. We introduce the notation
\begin{align*}
	&e_\epsilon^T(\hat u_\epsilon,\hat\nabla_\epsilon):=\sum_{i=3}^n(|(\nabla_\epsilon)_{\partial_i}u_\epsilon|^2+|\iota_{\partial_i}F_{\nabla_\epsilon}|^2).
\end{align*}
Balls will be denoted by $\cb_r(y)$ or $B_r^n(y)$, depending on whether they are with respect to $g_{x,s_\epsilon}$ or $g_{\mathbb R^n}$, respectively.
The volume $|E|$ of a set $E$ will be always understood with respect to the Euclidean metric.
%Observe that the energy measures of $(\hat u_\epsilon,\hat\nabla_\epsilon)$ converge to $\mu_{x,s}:=(\iota_{x,s}^{-1})_*(\mu\mrestr B_{\operatorname{inj}(M)}(x))$ and that
%\begin{align}\label{constr.mu}
%&\operatorname{spt}(\mu_{x,s})\to\{0\}\times\mathbb R^{n-2}
%\end{align}
%in the local Hausdorff topology.
%This readily follows from \eqref{blow-up.shape} and the lower bound in \eqref{mudensbds}.
The next proposition is the analogue of \cite[Lemma~2.4]{Lin} in this setting.

\begin{prop}\label{geo.lemma}
	As $\epsilon\to 0$ we have
	\begin{align*}
		&\lim_{\epsilon\to 0}\int_{B_2^2\times B_2^{n-2}}e_\epsilon^T(\hat u_\epsilon,\hat\nabla_\epsilon)=0.
	\end{align*}	
\end{prop}

\begin{proof}
	Let $C_m$ be the constant in Theorem \ref{monotonicity}. We first note that, given $y\in\{0\}\times\mathbb R^{n-2}$,
	\begin{align*}
		&\lim_{\epsilon\to 0}\hat\mu_\epsilon(\cb_r(y))=\Theta\omega_{n-2}r^{n-2};
	\end{align*}
	indeed, for any $\eta>0$, $B_{r-\eta}^n(y)\subseteq \cb_r(y)\subseteq B_{r+\eta}^n(y)$ eventually. We deduce that
	\begin{align}\label{almost.conic}\begin{aligned}
		&\lim_{\epsilon\to 0}(e^{C_M s_\epsilon r}(s_\epsilon r)^{2-n}\mu_{s_\epsilon \epsilon}(B_{s_\epsilon r}(\iota_{x,s_\epsilon }(y)))+C_M s_\epsilon r) \\
		&=\lim_{\epsilon\to 0}(e^{C_M s_\epsilon r}r^{2-n}\hat\mu_\epsilon(\cb_r(y))+C_M s_\epsilon r)\to\Theta_{n-2}(\mu,x)\omega_{n-2}.
	\end{aligned}\end{align}
	Pick $3\le i\le n$ and fix $R>0$. Choosing $y:=-2Re_i$, we can apply \eqref{mono.ineq} between the radii $R$ and $3R$ to obtain that
	\begin{align*}
		&\int_{B_{3s_\epsilon R}(p_i)\setminus B_{s_\epsilon R}(p_i)}d_{p_i}^{2-n}(|\nabla_{\nu_{R,i}} u_{s_\epsilon \epsilon}|^2+|\iota_{\nu_{R,i}} F_{\nabla_{s_\epsilon \epsilon}}|^2) \\
		&\le (e^{C_M (3s_\epsilon R)}(3s_\epsilon R)^{2-n}\mu_{s_\epsilon \epsilon}(B_{3s_\epsilon R}(p_i))+C_M(3s_\epsilon R)) \\
		&\quad-(e^{C_M (s_\epsilon R)}(s_\epsilon R)^{2-n}\mu_{s_\epsilon \epsilon}(B_{s_\epsilon R}(p_i))+C_M(s_\epsilon R)),
	\end{align*}
	where $p_i:=\iota_{x,s_\epsilon }(-2Re_i)$ and $\nu_{R,i}:=\operatorname{grad}d_{p_i}$. Now \eqref{almost.conic} and the comparability of $g_{x,s_\epsilon }$ with $g_{\mathbb R^n}$ give
	\begin{align*}
		&\lim_{\epsilon\to 0}\int_{B_{3R}(-2Re_i)\setminus B_{R}(-2Re_i)}(|\nabla_{\tilde\nu_{R,i}} \hat u_{\epsilon}|^2+|\iota_{\tilde\nu_{R,i}} F_{\hat\nabla_{\epsilon}}|^2)=0,
	\end{align*}
	where $\tilde\nu_{R,i}$ is the gradient of the distance function $d_{-2Re_i}$, both with respect to the metric $g_{x,s_\epsilon }$. Since eventually $B_{3R}(-2Re_i)\setminus B_{R}(-2Re_i)$ includes $B_2^2\times B_2^{n-2}$ for $R$ big enough, we get
	\begin{align}\label{approx.trans.inv}
		&\lim_{\epsilon\to 0}\int_{B_2^2\times B_2^{n-2}}(|\nabla_{\tilde\nu_{R,i}} \hat u_{\epsilon}|^2+|\iota_{\tilde\nu_{R,i}} F_{\hat\nabla_{\epsilon}}|^2)=0.
	\end{align}
	By monotonicity, eventually we have
	\begin{align}\label{en.bound.scaled}\begin{aligned}
	\limsup_{\epsilon\to 0}\int_{B_2^2\times B_2^{n-2}}e_\epsilon(\hat u_\epsilon,\hat \nabla_\epsilon)
	&\le\limsup_{\epsilon\to 0} 2s_\epsilon ^{2-n}\int_{B_{5s_\epsilon }(x)}e_{s_\epsilon \epsilon}(u_{s_\epsilon \epsilon}, \nabla_{s_\epsilon \epsilon}) \\
	&\le C(M,\Lambda).
	\end{aligned}\end{align}
	The smooth convergence $g_{x,s_\epsilon }\to g_{\mathbb R^n}$ gives $\tilde\nu_{R,i}(y)\to Y_{R,i}(y):=\frac{y+2Re_i}{|y+2Re_i|}$ uniformly on $B_2^2\times B_2^{n-2}$. Hence, the bound \eqref{en.bound.scaled} and \eqref{approx.trans.inv} give
	\begin{align}\label{approx.trans.inv2}
		&\lim_{\epsilon\to 0}\int_{B_2^2\times B_2^{n-2}}(|\nabla_{Y_{R,i}} \hat u_{\epsilon}|^2+|\iota_{Y_{R,i}} F_{\hat\nabla_{\epsilon}}|^2)=0.
	\end{align}
	Now $Y_{R,i}\to e_i=\partial_i$ as $R\to\infty$, and the statement follows from \eqref{approx.trans.inv2} and the uniform bound \eqref{en.bound.scaled}.	
\end{proof}

We now state the main technical result of the section, which will be shown later.
Fix a cut-off function $\chi\in C^\infty_c(B_2^2)$ with $\chi(z)=1$ for $|z|\le\frac{3}{2}$ and $0\le\chi\le 1$, and let $\hat\chi(z,t):=\chi(z)$.

\begin{prop}\label{slice.thesis.prop}
	There exists $F_\epsilon\subseteq B_1^{n-2}$ with $|F_\epsilon|\ge\frac{1}{4}|B_1^{n-2}|$ such that
	\begin{align}\label{slice.thesis}
	&\sup_{t\in F_{\epsilon}}\operatorname{dist}\Big(\int_{\mathbb R^2\times\{t\}}\chi(z)e_\epsilon(\hat u_\epsilon,\hat \nabla_\epsilon)(z,t),2\pi\mathbb N\Big)\to 0\quad\text{ as }\epsilon\to 0.
	\end{align}
\end{prop}

Before giving the proof, let us see how this implies the integrality of $V$.

\begin{proof}[Proof of Theorem \ref{thm1}]
	As $\epsilon\to 0$, we have both \eqref{slice.thesis} and
	\begin{align}\label{approx.density}
	&\lim_{\epsilon \to 0}\int_{\R^2\times B_2^{n-2}}\hat\chi(z)e_\epsilon(\hat u_\epsilon,\hat \nabla_\epsilon)=\lim_{\epsilon \to 0}\int_{\R^2\times B_2^{n-2}}\hat\chi\,d\hat\mu_\epsilon=\omega_{n-2}\Theta_{n-2}(\mu,x),
	\end{align}
	\begin{align}\label{small.annulus}
	&\int_{\R^2\times B_1^{n-2}}|d\hat\chi|\,d\hat\mu_\epsilon\le C\hat\mu_{\epsilon}((B_2^2\setminus B_2^2)\times B_1^{n-2})\to 0,
	\end{align}
	as $\hat\mu_\epsilon\rightharpoonup\Theta_{n-2}(\mu,x)\mathcal{H}^{n-2}\mrestr\{0\}\times\mathbb R^{n-2}$.
	In view of \eqref{en.bound.scaled} and \eqref{small.annulus}, for any vector field $(Y^3,\dots,Y^n)\in C^\infty_c(B_2^{n-2},\mathbb R^{n-2})$ we can integrate
	\eqref{stateq1} against $\chi(\sum_{i=3}^n Y^i\partial_i)$ and obtain, in the Euclidean metric,
	\begin{align*}
	&\Big|\int_{\mathbb R^2\times B_2^{n-2}}\hat\chi\langle T_\epsilon(u_\epsilon,\nabla_\epsilon),dY_i\otimes\partial_i\rangle\Big|
	\le \lambda_\epsilon(\|Y\|_{L^\infty}+\|DY\|_{L^\infty})
	\end{align*}
	for some sequence $\lambda_\epsilon\to 0$, thanks to the smooth convergence $g_{x,s_\epsilon}\to g_{\mathbb R^n}$.
	Invoking Proposition \ref{geo.lemma} and noting that $\|Y\|_{L^\infty}\le 2\|DY\|_{L^\infty}$, we can conclude that the nonnegative function $f_\epsilon(t):=\frac{1}{2\pi}\int_{\mathbb R^2\times\{t\}}\hat\chi e_\epsilon$ satisfies
	\begin{align*}
	&\Big|\int_{B_2^{n-2}}f_\epsilon\operatorname{div}(Y)\Big|\le \lambda_\epsilon\|DY\|_{L^\infty}
	\end{align*}
	for a possibly different sequence $\lambda_\epsilon\to 0$.
	Applying the Hahn--Banach theorem to the subspace $\{DY\mid Y\in C^\infty_c(B_1^{n-2},\R^{n-2})\}\subset C_0(B_1^{n-2},\R^{n-2})$ ($C_0$ denoting the closure of $C_c$), we can find real measures $(\nu_\epsilon)^i_j$ such that
	\begin{align*}
		&\partial_j f=\sum_{i=3}^{n}\partial_i(\nu_\epsilon)^i_j\quad\text{for all }j=3,\dots,n
	\end{align*}
	as distributions and $|(\nu_\epsilon)^i_j|(B_2^{n-2})\to 0$.
	Allard's strong constancy lemma \cite[Theorem~1.(4)]{allard} gives then
	\begin{align*}
	&\Big\|f-\frac{1}{\omega_{n-2}}\int_{B_1^{n-2}}f\Big\|_{L^1(B_1^{n-2})}\to 0.
	\end{align*}
	Since the sets $F_{\epsilon}$ of Proposition \ref{slice.thesis.prop} have positive measure, there clearly exists $t\in F_\epsilon$ such that
	\begin{align*}
	&\Big|f(t)-\frac{1}{\omega_{n-2}}\int_{B_{1}^{n-2}}f\Big|\le\frac{1}{|F_\epsilon|}\Big\|f-\frac{1}{\omega_{n-2}}\int_{B_{1}^{n-2}}f\Big\|_{L^1(B_1^{n-2})}.
	\end{align*}	
	Recalling \eqref{slice.thesis}, we deduce that
	\begin{align*}
	&\operatorname{dist}\Big(\frac{1}{\omega_{n-2}}\int_{B_{1}^{n-2}}f,\mathbb N\Big)\to 0.
	\end{align*}
	Hence, by \eqref{approx.density}, we get $\operatorname{dist}(\Theta_{n-2}(\mu,x),\mathbb N)=0$, which concludes the proof that $V$ is integral.
\end{proof}

\begin{proof}[Proof of Proposition \ref{slice.thesis.prop}]
%	Since $x\not\in E$, we have
%	\begin{align}\label{vert.en.bound.scaled}
%		&\int_{B_2^2\times B_2^{n-2}}1_{Z_\gamma(\hat u_\epsilon)^c}\,e_\epsilon(\hat u_\epsilon,\hat \nabla_\epsilon)\le C(M,L,\Lambda)\delta.
%	\end{align}	
	Taking into account Proposition \ref{geo.lemma}, the classical Hardy--Littlewood weak-(1,1) maximal estimate (applied to the function $t\mapsto\int_{B_2^2\times\{t\}}e_\epsilon^T(\hat u_\epsilon,\hat A_\epsilon)$) gives
	\begin{align}\label{vert.en.max}
		&\frac{1}{r^{n-2}}\int_{B_2^2\times B_r^{n-2}(t)}e_\epsilon^T(\hat u_\epsilon,\hat \nabla_\epsilon)\le C(n)\int_{B_2^2\times B_2^{n-2}}e_\epsilon^T(\hat u_\epsilon,\hat \nabla_\epsilon)\to 0
	\end{align}
	for all $t\in B_1^{n-2}\setminus E_1^\epsilon$ and $0<r<1$, where $E_1^\epsilon$ is a Borel set with $|E_1^\epsilon|\le\frac{1}{4}|B_1^{n-2}|$.
	Similarly, \eqref{en.bound.scaled} and \eqref{small.annulus} give
	\begin{align}\label{en.max}
		&\frac{1}{r^{n-2}}\hat\mu_\epsilon(B_2^2\times B_r^{n-2}(t))\le C(M,L,\Lambda),
	\end{align}
	\begin{align}\label{small.ann.max}
		&\frac{1}{r^{n-2}}\hat\mu_\epsilon((B_2^2\setminus B_1^2)\times B_r^{n-2}(t))\le C(n)\hat\mu_\epsilon((B_2^2\setminus B_1^2)\times B_2^{n-2})\to 0
	\end{align}
%	\begin{align}\label{small.nu}
%		&\frac{1}{r^{n-2}}\int_{B_2^2\times B_r^{n-2}(t)}1_{Z_\gamma(\hat u_\epsilon)^c}\,e_\epsilon(\hat u_\epsilon,\hat \nabla_\epsilon)\le C(M,L,\Lambda)\delta r^{n-2}
%	\end{align}
	for $t\in B_1^{n-2}\setminus(E_2^\epsilon\cup E_3^\epsilon)$ and $0<r<1$, with $|E_2^\epsilon|,|E_3^\epsilon|\le\frac{1}{4}|B_1^{n-2}|$.
	
	Pick any $t^\epsilon\in B_1^{n-2}\setminus(E_1^\epsilon\cup E_2^\epsilon\cup E_3^\epsilon)$ and, for $0<r<1$, define
	\begin{align*}
		&\mathcal{V}^\epsilon(r):=\{z\in B_1^2:\operatorname{dist}((z,t^\epsilon),\,Z_{\beta_d/2}(\hat u_\epsilon))<r\}
	\end{align*}
	(with the Euclidean distance). In other words, $\mathcal{V}^\epsilon$ is the $t^\epsilon$-slice of the neighborhood $B_r^n(Z_{\beta_d/2}(\hat u_\epsilon))$. We claim that for $0<r<1$, $\mathcal{V}^{\epsilon}(r)$ satisfies a uniform area bound
	\begin{align}\label{area.vort.nbhd}
	&|\mathcal{V}^\epsilon(r)|\le C(M,\Lambda)r^2,
	\end{align}
	provided $\epsilon<r$ is small enough.
	Indeed,	$\mathcal{V}^\epsilon(r)$ is covered by the balls $B_{2r}^n(y)$ with $y\in (B_2^2\times B_r^{n-2}(t^\epsilon))\cap Z_{\beta_d/2}(\hat u_\epsilon)$. Vitali's covering lemma gives a disjoint collection $\{B_{2r}^2(y_j)\mid j\in J\}$ such that $\mathcal{V}^\epsilon(r)\subseteq\bigcup_{j}B_{10r}^2(y_j)$.
	By Corollary \ref{clearing.out}, we have a bound on the cardinality $|J|$:
	\begin{align*}
	&\hat\mu_\epsilon(B_2^2\times B_{3r}^{n-2})
	\ge\sum_{j\in J}\hat\mu_\epsilon(B_{2r}^n(y_j))
	\ge\sum_{j\in J}\hat\mu_\epsilon(\mathcal{B}_r(y_j))
	\ge c(M,\Lambda)r^{n-2}|J|
	\end{align*}
	(since $\frac{1}{4}g_{\mathbb R^n}\le g_{x,s_\epsilon}\le 4g_{\mathbb R^n}$ for $\epsilon$ sufficiently small).
	Hence, using also \eqref{en.max}, we get
	\begin{align*}
	&|\mathcal{V}^\epsilon(r)|\le\sum_j|B_{5r}^2(y_j)|\le 25|J|r^2\le C(M,\Lambda)r^2,
	\end{align*}
	confirming \eqref{area.vort.nbhd}.
	
	Given $R>0$, let $\{z_1^\epsilon,\dots,z_{N(R,\epsilon)}^\epsilon\}$ be a maximal subset of $\mathcal{V}^\epsilon(R\epsilon)$
	with $|z_k^\epsilon-z_\ell^\epsilon|\ge 2\epsilon$.
%	Arguing as in the proof of Proposition \ref{clearing.out}, we get a uniform lower bound on
%	$\epsilon^{-2}\int_{B_\epsilon(z_k)^2\times\{t^\epsilon\}}W(\hat u_\epsilon)$. Since the balls $B_\epsilon^2(z_k)$ are disjoint, \eqref{slice.potential}
%	gives an upper bound on $N(\epsilon)$,
	Since $\bigcup_k (B_1^2\cap B_\epsilon^2(z_k))\subseteq\mathcal{V}^\epsilon((R+1)\epsilon)$ and the balls $B_\epsilon^2(z_k)$ are disjoint,
	\eqref{area.vort.nbhd} gives a uniform bound on $N(R,\epsilon)$ independent of $\epsilon$ (eventually),
	so up to subsequences we can assume that $N(R)=N(R,\epsilon)$ is constant
	and that $\epsilon^{-1}|z_k^\epsilon-z_\ell^\epsilon|$ has a limit $r_{k\ell}$ as $\epsilon\to 0$, for each $k,l$. We say that $k\sim\ell$ if $r_{k\ell}<\infty$; this is evidently an equivalence relation (as $r_{km}\le r_{k\ell}+r_{\ell m}$),
	so we can pick a set of representatives $\{k_1,\dots,k_P\}$ of the distinct equivalence classes $[k_1],\dots,[k_P]$ and conclude that
	\begin{align*}
		&\mathcal{V}^\epsilon(R\epsilon)\subseteq\bigcup_{j=1}^P B_{S\epsilon}^2(z_{k_j}^\epsilon)
	\end{align*}
	eventually, for any fixed $S\ge S_0(R):=\max\{\sum_{\ell\in[k_j]}r_{k_j\ell}+2\mid j=1,\dots,P\}$.
	Fix such an $S$ which is also bigger than the constants $C$ in \eqref{en.max} and $a_d^{-1},C$ in Corollary \ref{decaycor}.
%	if $|x_k^\epsilon-x_\ell^\epsilon|<2\kappa(\beta,N^{-1}\delta,S)^{-1}\epsilon$ for two indices $j<\ell$ then we throw away the point $x_\ell^\epsilon$ and
%	replace $S$ with $3\kappa(\beta,N^{-1}\delta,S)^{-1}$. Iterating at most $N-1$ times, we arrive at a set of points $\{z_{k_1}^\epsilon,\dots,z_{k_P}^\epsilon\}$ such that
%	\begin{align*}
%		&Z^\epsilon\subseteq\bigcup_{j=1}^P B_{S\epsilon}^2(z_{k_j}^\epsilon),
%	\end{align*}
%	the balls $B_{\kappa(\beta,N^{-1}\delta,S)^{-1}}^2(z_{k_j})$ being disjoint
%	(both $P$ and $S$ can be assumed independent of $\epsilon$ up to subsequences).
	For any fixed $\delta>0$, \eqref{vert.en.max} and \eqref{en.max} show that, for $\epsilon$ sufficiently small, Proposition \ref{2dlimit} below applies to the rescaled solutions $\widehat u_\epsilon(z_{k_j}^\epsilon+\epsilon\cdot,t^\epsilon+\epsilon\cdot)$ (with $\beta:=\beta_d$). Writing $K=K(\beta_d,\delta,S)>S$,
	note that the balls $B_{K\epsilon}^2(z_{k_j})$ are eventually disjoint and included in $\{\chi=1\}$.
	Hence, Proposition \ref{2dlimit} and \eqref{small.ann.max} give
	\begin{align*}
		&\operatorname{dist}\Big(\int_{\mathbb R^2\times\{t^\epsilon\}}\hat\chi e_\epsilon(\hat u_\epsilon,\hat \nabla_\epsilon),2\pi\mathbb N\Big)
		\le P\delta+\int_{B_2^2\setminus\bigcup_{j=1}^{P}B_{K\epsilon}^2(z_{k_j}^\epsilon)}e_\epsilon(\hat u_\epsilon,\hat\nabla_\epsilon)(\cdot,t^\epsilon) \\
		&\le P\delta+\int_{B_2^2\setminus\mathcal{V}^\epsilon(R\epsilon)}e_\epsilon(\hat u_\epsilon,\hat\nabla_\epsilon)(\cdot,t^\epsilon) \\
		&\le (P+1)\delta+\int_{B_1^2\setminus\mathcal{V}^\epsilon(R\epsilon)}e_\epsilon(\hat u_\epsilon,\hat\nabla_\epsilon)(\cdot,t^\epsilon)
	\end{align*}
	(for $\epsilon$ sufficiently small). Choosing $\delta=\delta(R)\leq \frac{1}{(P+1)R}$, we arrive at the estimate
	\begin{align*}
		&\operatorname{dist}\Big(\int_{\mathbb R^2\times\{t^\epsilon\}}\hat\chi e_\epsilon(\hat u_\epsilon,\hat \nabla_\epsilon),2\pi\mathbb N\Big)
		\le \frac{1}{R}+\int_{B_1^2\setminus\mathcal{V}^\epsilon(R\epsilon)}e_\epsilon(\hat u_\epsilon,\hat\nabla_\epsilon)(\cdot,t^\epsilon)
	\end{align*}
 To conclude the proof, it suffices to show that
	\begin{align}\label{small.energy}
		&\lim_{R\to 0}\limsup_{\epsilon\to 0}\int_{B_1^2\setminus\mathcal{V}^\epsilon(R\epsilon)}e_\epsilon(\hat u_\epsilon,\hat\nabla_\epsilon)(\cdot,t^\epsilon)\to 0.
	\end{align}
	Once we have this, we infer that $\liminf_{\epsilon\to 0}\operatorname{dist}\big(\int_{\mathbb R^2\times\{t^\epsilon\}}\hat\chi e_\epsilon(\hat u_\epsilon,\hat \nabla_\epsilon),2\pi\mathbb N\big)=0$ for the original sequence $(t^\epsilon)$.
	Noting that the estimates above are independent of the choice of $t^{\epsilon}\in F_{\epsilon}=B_1^{n-2}\setminus E_1^{\epsilon}\cup E_2^{\epsilon}\cup E_3^{\epsilon}$, the proposition then follows.
	
	To show \eqref{small.energy}, note that for $y\in B_1^2\times B_1^{n-2}$ the distance of $\iota_{x,s_\epsilon}(y)$ to the set $Z_{\beta_d/2}(u_{s_\epsilon\epsilon})$
	is (eventually) bounded below by $\frac{s_\epsilon}{2}\min\{1,r_\epsilon(y)\}$, where $r_\epsilon(y)$ is the (Euclidean) distance of $(y,t^\epsilon)$ to $Z_{\beta_d/2}(\hat u_\epsilon))$.
	Since $Z_{\beta_d/2}(u_{s_\epsilon\epsilon})\supseteq Z_{\beta_d}(u_{s_\epsilon\epsilon})$, for any $R>1$ Corollary \ref{decaycor} gives
	\begin{align*}
	\int_{B_1^2\setminus\mathcal{V}^\epsilon(R\epsilon)}e_\epsilon(\hat u_\epsilon,\hat\nabla_\epsilon)
	&\le C\epsilon^{-2}\int_{B_1^2\setminus\mathcal{V}^\epsilon(R\epsilon)}e^{-a_dr_\epsilon(y)/(2\epsilon)}+C\epsilon^{-2}e^{-a_d/(2\epsilon)}+Cs_\epsilon\epsilon \\
	&=C\epsilon^{-3}\int_{B_1^2\setminus\mathcal{V}_\epsilon(R\epsilon)}\int_{r_\epsilon(y)}^\infty \frac{a_d}{2} e^{-a_dr/(2\epsilon)}\,dr\,dy+C\epsilon^{-2}e^{-a_d/(2\epsilon)}+s_\epsilon\epsilon \\
	&=C\epsilon^{-3}\int_{R\epsilon}^\infty \frac{a_d}{2} e^{-a_d r/(2\epsilon)}|\mathcal{V}_\epsilon(r)|\,dr+C\epsilon^{-2}e^{-a_d/(2\epsilon)}+Cs_\epsilon\epsilon \\
	&\le C\epsilon^{-3}\int_{R\epsilon}^\infty e^{-a_d r/(2\epsilon)}r^2\,dr+C\epsilon \\
	&=C\int_R^\infty e^{-a_d t/2}t^2\,dt+C\epsilon,
	\end{align*}
	where we used Fubini's theorem in the second equality. The statement follows.
\end{proof}

The following key technical proposition, used in the proof of Proposition \ref{slice.thesis.prop}, relies ultimately on the quantization phenomenon for the energy of entire solutions in the plane, presented in \cite[Chapter~III]{JT}.
For the reader's convenience, we give a self-contained proof, including the relevant arguments from \cite{JT}.

\begin{prop}\label{2dlimit}
	Given $0<\beta,\delta<\frac{1}{2}$ and $S>1$,
	there exist $K(\beta,\delta,S)>S$ and $0<\kappa(\beta,\delta,S,n)<K^{-1}$ such that the following is true.
	Assume $(u,\nabla)$ is smooth and solves \eqref{ueq} and \eqref{feq}, with $|u|\le 1$ and $\epsilon=1$, on a line bundle $L$ over a cylinder $(Q,g)$, with $Q=B_{\kappa^{-1}}^2\times B_{\kappa^{-1}}^{n-2}$.
	If we have
	\begin{align}\label{confined.vort}
	Z_{\beta/2}(u)\cap(B_{\kappa^{-1}}^2\times\{0\})\subseteq \bar B_S^2\times\{0\},
	\end{align}
	the energy bounds
	\begin{align}\label{energy.bd}
	%&\int_{B_{\kappa^{-1}}^2\times B_r^{n-2}}e_1(u,\nabla)\le Sr^{n-2}\quad\text{for all }0<r<\kappa^{-1},
	&e_1(u,\nabla)\le S,
	\end{align}
	\begin{align}\label{small.vert}
	&\sum_{i=3}^n\int_{B_{\kappa^{-1}}^2\times B_r^{n-2}}(|\nabla_{\partial_i}u|^2+|\iota_{\partial_i}F_\nabla|^2)\le\kappa\quad\text{for all }0<r<\kappa^{-1},
	\end{align}
	as well as the decay
	\begin{align}\label{decay.assump}
	&\quad e_1(u,\nabla)(p)\le Se^{-S^{-1}r}+\kappa\quad\text{whenever }\cb_r(p)\subseteq Q\setminus Z_\beta,
	\end{align}
	and $\|g-g_{\mathbb R^n}\|_{C^2}\le\kappa$, then
	\begin{align*}
	&\Big|\int_{B_K^2\times\{0\}}e_1(u,\nabla)-2\pi|p|\Big|<\delta,
	\end{align*}
	where $p$ is the degree of $\frac{u}{|u|}(S\cdot,0)$, as a map from the circle to itself.
\end{prop}

\begin{proof}
	To begin with, fix a real number $K(\beta,\delta,S)>S$ so big that
	\begin{align}\label{choice.of.K}
	&\int_{K}^\infty (2\pi r) Se^{-S^{-1}(r-S)}<\delta.
	\end{align}	
	Arguing by contradiction, assume there exists a sequence $\kappa_j\to 0$ such that the statement admits a counterexample $(u_j,\nabla_j)$ (for $\kappa=\kappa_j$) for a (necessarily trivial) line bundle $L_j$ over $Q_j=B_{\kappa_j^{-1}}^2\times B_{\kappa_j^{-1}}^{n-2}$, with respect to a metric $g=g_j$ satisfying $\|g-g_{\mathbb{R}^n}\|_{C^2}\leq \kappa_j$.
	Fixing a trivialization of $L_j$ over $Q_j$, we can write $\nabla_j=d-iA_j$ for some real one-form $A_j$.

	By virtue of the uniform pointwise estimate \eqref{decay.assump} for $e_1(u_j,\nabla_j)\geq |d|u_j||^2$, we see that the functions $|u_j|$ are locally equi-Lipschitz. %We claim that, since the energy of $(u_j,A_j)$ is locally equibounded, the functions $|u_j|$ are locally equi-Lipschitz: this is readily seen exploiting gauge invariance of $|u_j|$ and bootstrapping in \eqref{ueq} and \eqref{feq} using a local Coulomb gauge, obtaining local $W^{2,q}$ bounds (for the couple written in this gauge), for all $q<\infty$; the embedding $W^{2,q}\hookrightarrow C^1$ for bounded regular domains (and $q>n$) then confirms our claim. 
In particular, we can apply the Arzel\`a--Ascoli theorem to extract a subsequence $|u_j|\to\rho_\infty$ converging in $C^0_{loc}$. Since $|d|u_j||\le|\nabla_j u_j|$, \eqref{small.vert} implies that $\rho_\infty$ depends only on the first two variables. Moreover, \eqref{confined.vort} gives $\rho_\infty^2\ge 1-\frac{\beta}{2}>1-\beta$ outside $B_S^2\times\mathbb{R}^{n-2}$.
	So, setting
	\begin{align*}
		&R_j:=\max\{r\le\kappa_j^{-1}:(B_r^2\setminus B_S^2)\times B_1^{n-2}\subseteq\{u_j\neq 0\}\},
	\end{align*}
	we have $R_j\to\infty$. Let $w_j:=\frac{u_j}{|u_j|}$ on $\{u_j\neq 0\}$.
	The degree $p_j$ is uniformly bounded as, for $r\ge S$ and $t\in\R^{n-2}$,
	\begin{align*}
	&2\pi p_j=\int_{\partial B_r^2\times\{t\}}w_j^*(d\theta)
	=\int_{B_r^2\times\{t\}}dA_j+\int_{\partial B_r^2\times\{t\}}(w_j^*(d\theta)-A_j)
	\end{align*}
	for $j$ sufficiently large, so averaging over $S<r<2S$ and $t\in B_1^{n-2}$ we get
	\begin{align*}
	2\pi|p_j|&\le C(S)\int_{B_{2S}^2\times B_1^{n-2}}|dA_j|+C(S)\int_{(B_{2S}^2\setminus B_S^2)\times B_1^{n-2}}|w_j^*(d\theta)-A_j|) \\
	&\le C(\beta,S)\Big(\int_{B_{2S}^2\times B_1^{n-2}}e_1(u_j,A_j)\Big)^{1/2},
	\end{align*}
	as $|u_j||w_j^*(d\theta)-A_j|\le|\nabla_j u_j|$. Thus, up to subsequences we can assume $p_j=p$ is constant.
	
	We now claim that, up to change of gauge, $(u_j,A_j)\to(u_\infty,A_\infty)$ subsequentially in $C^{1}_{loc}(\mathbb R^2\times B_1^{n-2})$.
	Let $\tilde u_j=e^{i\theta_j}u_j$ be the section in the %we should be careful to specify boundary conditions; as you said, may be good to do an appendix on some of this stuff
	Coulomb gauge on the domain $(\bar B_{5S}^n,g_j)$, with $A_j(\nu)=0$ on the boundary (as described in the Appendix). Note that $B_{5S}^n$ includes the cylinder $Q':=B_{4S}^2\times B_1^{n-2}$.
	and observe that, on $Q'':=(B_{4S}^2\setminus B_S^2)\times B_1^{n-2}$, $\tilde{u}_j$ has the form
$$\tilde u_j(re^{i\theta},t)=|u_j|e^{ip\theta+i\psi_j}$$
for a unique real function
	$\psi_j$ with $0\le\psi_j(2S,0)<2\pi$.
	
	Hence, $u_j=|u_j|e^{i(p\theta+\psi_j-\theta_j)}$ on $Q''$ and we can extend $\psi_j-\theta_j$ uniquely to a function $\sigma_j:(B_{R_j}^2\setminus B_S^2)\times B_1^{n-2}\to\mathbb R$ so that $u_j=|u_j|e^{ip\theta+i\sigma_j}$ holds true on all the domain of $\sigma_j$. Finally,
	we replace $(u_j,A_j)$ with $(e^{i\tau_j}u_j,A_j+d\tau_j)$, where
	\begin{align*}
	&\tau_j(z,t):=\begin{cases} \theta_j-\chi(|z|)\psi_j & |z|<4S \\ -\sigma_j & |z|>3S \end{cases}
	\end{align*}
	for a fixed smooth function $\chi:[0,\infty)\to[0,1]$ such that $\chi=0$ on $[0,2S]$ and $\chi=1$ on $[3S,\infty)$.
	Observe that the new couple obeys uniform local $W^{2,q}$ bounds in the cylinder $Q'=B_{4S}^2\times B_1^{n-2}$, for all $1\leq q<\infty$, thanks to the Coulomb gauge specification (per Proposition \ref{w2q.prop} in the Appendix.) %Again, may be good to provide some details here.

Moreover, in the exterior annular region $\mathcal{A}_j:=(B_{R_j}^2\setminus\bar B_{3S}^2)\times B_1^{n-2}$, we have that $u_j(re^{i\theta},t)=|u_j|e^{pi\theta}$ and we can obtain local $W^{2,q}$ bounds
%	rewriting \eqref{ueq} as
%	\begin{align*}
%	&-\Delta|u_j|+i |u_j|d^*(p d\theta-A_j)-2i\langle d|u_j|,pd\theta-A_j\rangle
%	+|pd\theta-A_j|^2=\frac{1-|u_j|^2}{2}|u_j|.
%	\end{align*}
%	Note that $A_j$ and $dA_j$ are now both bounded in $L^2_{loc}$ by \eqref{energy.bd}.
%	Solving \eqref{radial} for $d^*A_j$ and recalling that $|u_j|$ is already bounded in $H^{\lfloor n/2\rfloor+2}_{loc}$, we get a local $L^2$-bound for $dd^*A_j$. Rewriting \eqref{feq} as
%	\begin{align*}
%	&\Delta_H A_j=dd^*A+|u_j|^2(pd\theta-A_j)
%	\end{align*}
%	we get a local $H^2$-bound for $A_j$.
	noting that
	\begin{align*}
		&pd\theta-A_j=|u_j|^{-2}\langle \nabla_j u_j,u_j\rangle.
	\end{align*}
	Indeed, since the right-hand side is bounded by $e_1(u_j,\nabla_j)^{1/2}\leq S^{1/2}$ and $pd\theta$ is a fixed smooth one-form, we immediately obtain uniform $L^{\infty}$ bounds for $A_j$ locally in $\mathcal{A}_j$. Next, note that the identity \eqref{moduboch} applies to give us an estimate
$$|\Delta |u_j|^2|\leq Ce_1(u_j,\nabla_j)\leq CS$$
in $\mathcal{A}_j$, from which it follows that the modulus $|u_j|$ satisfies uniform $W^{2,q}$ bounds for every $q\in (1,\infty)$ locally in $\mathcal{A}_j$.
Taking the imaginary part of \eqref{ueq} gives
	\begin{align*}
		&|u_j|d^*(p d\theta-A_j)=2\langle d|u_j|,pd\theta-A_j\rangle,
	\end{align*}
	from which it follows that $d^*A_j$ satisfies uniform $L^{\infty}$ bounds locally in $\mathcal{A}_j$ as well; together with the obvious pointwise bound $|dA_j|\leq e_1(u_j,\nabla_j)^{1/2}\leq S^{1/2}$, this in particular yields uniform bounds on the full derivative $\|DA_j\|_{L^q}$ for every $q\in (1,\infty)$ on fixed compact subsets of $\mathcal{A}_j$. Finally, writing \eqref{feq} as
	\begin{align*}
		&\Delta_H A_j=dd^*A_j+|u_j|^2(pd\theta-A_j),
	\end{align*}
the preceding chain of identities and estimates give a uniform $L^q$ bound on the right-hand side over any fixed compact subset of $\mathcal{A}_j$, for any $q\in (1,\infty)$; in particular, this gives us the desired uniform local $W^{2,q}$ bounds for $A_j$ (while we already have the desired $W^{2,q}$ bounds for $u_j=|u_j|e^{pi\theta}$).
	
	Thanks to the compact embedding $W^{2,q}\hookrightarrow C^1$ on bounded regular domains (for $q>n$), we obtain a limit couple $(u_\infty,A_\infty)$ on $\mathbb R^2\times B_1^{n-2}$, as claimed, which solves \eqref{ueq} and \eqref{feq} with respect to the flat metric. Also, $|u_\infty|=\rho_\infty$ and
	\begin{align}\label{limit.stuff}
	&(\nabla_\infty)_{\partial_k}u_\infty=0,\quad\iota_{\partial_k}dA_\infty=0\quad\text{for }k=3,\dots,n.
	\end{align}
	The second part of \eqref{limit.stuff} implies that we can find a function $\alpha\in C^1(\mathbb R^2\times B_1^{n-2})$ with $\alpha(z,0)=0$ and $\partial_k\alpha=(A_\infty)_k$, for all $z\in\mathbb R^2$ and all $k\ge 3$. Set $\tilde u_\infty:=e^{-i\alpha}u_\infty$ and $\tilde A_\infty:=A_\infty-d\alpha$, so that
	\begin{align*}
	&(\tilde A_\infty)_k=0,\quad\partial_k(\tilde A_\infty)_\ell=\partial_k(A_\infty)_\ell-\partial^2_{k\ell}\alpha=\partial_\ell(A_\infty-d\alpha)_k=0
	\end{align*}
	for all $k=3,\dots,n$ and $\ell=1,\dots,n$ (using again the second part of \eqref{limit.stuff}). The first part gives instead $\partial_k\tilde u_\infty=0$ for $k=3,\dots,n$. Hence, $(\tilde u_\infty,\tilde A_\infty)$ depends only on the first two variables and therefore corresponds to a planar solution of \eqref{ueq} and \eqref{feq}.
	
	Also, from \eqref{decay.assump} we deduce that
	\begin{align}\label{decay.limit}
		&e_1(\tilde u_\infty,\tilde A_\infty)(z,t)=e_1(u_\infty,A_\infty)(z,t)=\lim_{j\to\infty}e_1(u_j,A_j)(z,t)\le Se^{-S^{-1}(|z|-S)}
	\end{align}
	for $|z|>S$, as eventually $\bar B_{|z|-S}^n(z,t)\cap Z_\beta(u_j)=\emptyset$.
	
	%Now \cite[Theorems~III.1.1-2]{JT} (see also the discussion preceding the statements) give
	%\begin{align*}
	%	&\int_{\mathbb{R}^2\times\{0\}}e_1(u_\infty,A_\infty)=\int_{\mathbb{R}^2\times\{0\}}e_1(\tilde u_\infty,\tilde A_\infty)= 2\pi|p|
	%\end{align*}
	Integrating \eqref{stateq1} on $\R^2=\R^2\times\{0\}$ against the position vector field we get
	\begin{align*}
		&\int_{\R^2}|d\tilde A_\infty|^2=\int_{\R^2}W(\tilde u_\infty).
	\end{align*}
	Thanks to the decay of $e_1(\tilde u_\infty,\tilde A_\infty)$, we can repeat the proof of \eqref{ptwisebal} obtaining $|d\tilde A_\infty|\le\sqrt{W(\tilde u_\infty)}$, so we must have $|d\tilde A_\infty|=\sqrt{W(\tilde u_\infty)}$ everywhere (cf. \cite{JT}, Section III.8). Observe that, by \eqref{moduboch} and the strong maximum principle, $|\tilde u_\infty|<1$ (unless $|\tilde u_\infty|=1$ everywhere, in which case $|d\tilde A_\infty|=\sqrt{W(\tilde u_\infty)}=0$ and $|\tilde \nabla_\infty \tilde u_\infty|=0$ by \eqref{moduboch}, thus $e_1(\tilde u_\infty,\tilde A_\infty)=0$). As a consequence, $|*d\tilde A_\infty|=W(\tilde u_\infty)>0$ and we get either $\frac{1-|\tilde u_\infty|^2}{2}=*d\tilde A_\infty$ everywhere or $\frac{1-|\tilde u_\infty|^2}{2}=-*d\tilde A_\infty$ everywhere. Thus, integrating by parts and using \eqref{ueq},
	\begin{align*}
	&\int_{\mathbb R^2} e_1(\tilde u_\infty,\tilde A_\infty)
	=\int_{\mathbb R^2}(|\tilde \nabla_\infty \tilde u_\infty|^2+2W(\tilde u_\infty))
	=\int_{\mathbb R^2}(\langle\tilde \nabla_\infty^*\tilde \nabla_\infty \tilde u_\infty,\tilde u_\infty\rangle+2W(\tilde u_\infty)) \\
	&=\int_{\mathbb R^2}\frac{1-|\tilde u_\infty|^2}{2}
	=\pm\int_{\mathbb R^2} d\tilde A_\infty=\pm\lim_{r\to\infty}\int_{\partial B_r^2}\tilde A_\infty
	=\pm\lim_{r\to\infty}\int_{\partial B_r^2}pd\theta
	=\pm 2\pi p.
	\end{align*}
	Hence, the energy of the two-dimensional solution $(\tilde u_\infty,\tilde A_\infty)$ is $2\pi|p|$.
	Our choice of $K$, namely \eqref{choice.of.K}, together with \eqref{decay.limit}, then ensures that
	\begin{align*}
	&\operatorname{dist}\Big(\int_{B_K^2\times\{0\}}e_1(u_\infty,A_\infty),2\pi\mathbb N\Big)<\delta.
	\end{align*}
	As a consequence, this must hold eventually also for $(u_j,A_j)$, giving the desired contradiction.
\end{proof}

\begin{remark}\label{stronger.nonzero}
	As a consequence, one also finds that
	\begin{align*}
		&\int_{B_K^2\times\{0\}}e_1(u,\nabla)<\delta
	\end{align*}
	if $|u|>0$ everywhere on the cylinder $Q$.
	Indeed, if $|u|>0$ everywhere, then the degree $p$ in the statement of Proposition \ref{2dlimit} clearly must vanish.
	%the limiting section in the preceding argument clearly must vanish, so that the integer in question is $0$.
\end{remark}

We are now able to address the statement on the convergence of level sets.

\begin{prop}
	For any $0\le\delta<1$ we have $\operatorname{spt}(\mu)=\lim_{\epsilon\to 0}\{|u_\epsilon|\le\delta\}$, in the Hausdorff topology.
\end{prop}

\begin{proof}
	If $x=\lim_{\epsilon\to 0}x_\epsilon$,
	for points $x_\epsilon\in\{|u_\epsilon|\le\delta\}$ defined along a subsequence,
	then the same argument used in the proof of Proposition \ref{densprop} shows that $x\in\operatorname{spt}(\mu)$.
	Hence, for all $\eta>0$, eventually $\{|u_\epsilon|\le\delta\}$ is included in the $\eta$-neighborhood of $\operatorname{spt}(\mu)$.
	
	To conclude the proof, it suffices to show that the converse inclusion $\operatorname{spt}(\mu)\subseteq B_\eta(\{u_\epsilon=0\})$ holds eventually. Arguing by contradiction, assume that there are points $p_\epsilon\in\operatorname{spt}(\mu)$ whose distance from $\{u_\epsilon=0\}$ is at least $\eta$, along some subsequence (not relabeled).
	Up to further subsequences, let $p_\epsilon\to p_0\in\operatorname{spt}(\mu)$.
	
	Since $\mu$ is $(n-2)$-rectifiable,
	there exists a point $q\in\operatorname{spt}(\mu)$ with $\operatorname{dist}(p_0,q)<\frac{\eta}{2}$, and such that $\mu$ blows up to $\Theta_{n-2}(\mu,q)\mathcal{H}^{n-2}\mrestr T_q\Sigma$ at $q$.
	Observe that eventually we have
	\begin{align}\label{nonzero.weird}
		&\operatorname{dist}(q,\{u_\epsilon=0\})\ge\frac{\eta}{2}.
	\end{align}
	Now, repeating all the preceding blow-up analysis at $q$, in view of Remark \ref{stronger.nonzero} we can improve
	\eqref{slice.thesis} to the uniform convergence
	$$\int_{\mathbb R^2\times\{t\}}\chi(z)e_\epsilon(\hat u_\epsilon,\hat \nabla_\epsilon)(z,t)\to 0$$
	for $t\in F_\epsilon$, which implies that $\Theta_{n-2}(\mu,q)=0$. However, since $q\in\operatorname{spt}(\mu)$, this is impossible, by Proposition \ref{densprop}.
\end{proof}

\subsection{Limiting behavior of the curvature}\hfill

As before, we identify the curvature $F_{\nabla}$ with a closed two-form $\omega_{\epsilon}$ by $F_{\nabla_{\epsilon}}(X,Y)=-i\omega_{\epsilon}(X,Y)$. Recall that the cohomology class $[\frac{1}{2\pi}\omega_{\epsilon}]$ represents the (rational) first Chern class $c_1(L)\in H^2(M;\mathbb{R})$ of the complex line bundle $L\to M$.

\begin{thm}\label{curvthm} Assume $M$ is oriented. Let $(u_{\epsilon},\nabla_{\epsilon})$ be a family as in Theorem \ref{thm1}. The curvature forms $\frac{1}{2\pi}\omega_{\epsilon}$ can be identified with $(n-2)$-currents that converge (weakly), as $\epsilon\to 0$, to an integer rectifiable cycle $\Gamma$ which is Poincar\'{e} dual to $c_1(L)$, and whose mass measure $|\Gamma|$ satisfies $|\Gamma|\leq \mu.$
\end{thm}

\begin{proof}
Recall from Section \ref{preliminary.sec} that
$$d\langle \nabla_{\epsilon} u_{\epsilon},iu_{\epsilon}\rangle=\psi(u_{\epsilon})-|u_{\epsilon}|^2\omega,$$
where $\psi(u_{\epsilon}):=\langle 2i\nabla u_{\epsilon},\nabla_{\epsilon} u_{\epsilon}\rangle$ is a two-form satisfying $|\psi(u)|\leq |\nabla u|^2$ pointwise. In particular, denoting by $J(u_{\epsilon},\nabla_{\epsilon})$ the two-form
$$J(u_{\epsilon},\nabla_{\epsilon}):=\psi(u_{\epsilon})+(1-|u_{\epsilon}|^2)\omega,$$
we can rewrite this identity as
\begin{align}\label{j.and.omega}
	&J(u_{\epsilon},\nabla_{\epsilon})-\omega_{\epsilon}=d\langle \nabla_{\epsilon} u_{\epsilon},iu_{\epsilon}\rangle,
\end{align}
and observe that
\begin{equation}\label{jebd}
|J(u_{\epsilon},\nabla_{\epsilon})|\leq |\nabla_{\epsilon} u_{\epsilon}|^2+\epsilon^2|\omega_{\epsilon}|^2+\frac{1}{4\epsilon^2}(1-|u_{\epsilon}|^2)^2=e_{\epsilon}(u_{\epsilon},\nabla_{\epsilon}).
\end{equation}
The dual $(n-2)$-currents given by
\begin{align*}
	&\langle\Gamma_\epsilon,\zeta\rangle:=\int_M J(u_\epsilon,\nabla_\epsilon)\wedge\zeta,
\end{align*}
for any $(n-2)$-form $\zeta\in \Omega^{n-2}(M)$, are thus bounded in mass by $\Lambda$. Up to subsequences, we can take a weak limit $\Gamma$. The bound $|\Gamma_\epsilon|\leq\mu_{\epsilon}$ implies that also $|\Gamma|\le\mu$. From \eqref{j.and.omega} and integration by parts we get
$$\int_M \omega_{\epsilon}\wedge \zeta=\int_M J(u_{\epsilon},\nabla_{\epsilon})\wedge \zeta-\int_M \langle \nabla_{\epsilon} u_{\epsilon},iu_{\epsilon}\rangle\wedge d\zeta.$$
Since (as discussed in the proof of Proposition \ref{densprop})
$$\int_M |\langle \nabla_{\epsilon} u_{\epsilon},iu_{\epsilon}\rangle|\leq \int e_{\epsilon}(u_{\epsilon},\nabla_{\epsilon})^{1/2}\to 0$$
as $\epsilon\to 0$, it follows that 
\begin{equation}
\langle \Gamma,\zeta\rangle=\frac{1}{2\pi}\lim_{\epsilon\to 0}\int_M J(u_{\epsilon},\nabla_{\epsilon})\wedge \zeta =\frac{1}{2\pi}\lim_{\epsilon\to 0}\int_M \omega_{\epsilon}\wedge \zeta
\end{equation}
for every smooth $(n-2)$-form $\zeta\in \Omega^{n-2}(M)$. Since $\mu$ is $(n-2)$-rectifiable, $\Gamma$ must be a rectifiable $(n-2)$-current: this can be seen by blow-up, applying \cite[Proposition~7.3.5]{krantz.parks}.
Since the two-forms $\omega_{\epsilon}$ are closed, for any $\xi\in\Omega^{n-3}(M)$ we have
\begin{align*}
	&\langle\partial\Gamma,\xi\rangle=\langle\Gamma,d\xi\rangle
	=\frac{1}{2\pi}\lim_{\epsilon\to 0}\int_M\omega_\epsilon\wedge d\xi
	=\frac{1}{2\pi}\lim_{\epsilon\to 0}\int_M d(\omega_\epsilon\wedge d\xi)=0,
\end{align*}
so $\Gamma$ is a cycle. By construction, $\Gamma$ is Poincaré dual to $c_1(L)$.

To complete the proof, it remains to show that $\Gamma$ has integer multiplicity.
%at those points $p\in \Sigma=\operatorname{spt}(\mu)$ with well-defined tangent plane $T_p\Sigma\subset T_pM$.
By means of a diagonal selection of a subsequence, as in the previous subsection, we can deduce integrality at those points $p\in\operatorname{spt}(\mu)$ where $\mu$ blows up to $\Theta_{n-2}(\mu,p)\mathcal{H}^{n-2}\mrestr T_p\Sigma$, using the following lemma. Note that its hypotheses are verified thanks to Corollary \ref{decaycor} and the fact that $Z_{\beta_d}(u_\epsilon)$ necessarily converges to $T_p\Sigma$ in the local Hausdorff topology, after rescaling (see the proof of Proposition \ref{densprop}).

Since $\mu$ is $(n-2)$-rectifiable, we deduce that the limiting current $\Gamma$ has integer multiplicity $\mathcal{H}^{n-2}$-a.e. on its support, as claimed.
\end{proof}

\begin{lem} On the Euclidean ball $B_4^n$, let $(u_{\epsilon},\nabla_{\epsilon})$ be a sequence of sections and connections in a trivial line bundle $L\to B_4^n$ (not necessarily satisfying any equation) for which $E_{\epsilon}(u_{\epsilon},\nabla_{\epsilon})\leq \Lambda$, $e_{\epsilon}(u_{\epsilon},\nabla_{\epsilon})\to 0$ in $C^0_{loc}(B_4^n\setminus P)$ and $*\omega_{\epsilon}\to \theta_1 [P]$ in $\mathcal{D}_{n-2}(B_4^n)$, where $P=\{0\}\times \mathbb{R}^{n-2}$. Then $\theta_1\in 2\pi\mathbb{Z}$.
\end{lem}
\begin{proof} To begin, fix a test function $\varphi\in C_c^1(B_1^2\times B_1^{n-2})$ of the form $\varphi(x^1,\ldots,x^n)=\psi(x^1,x^2)\eta(x^3,\ldots,x^n)$, with $\psi(x^1,x^2)=1$ for $|(x^1,x^2)|\leq \frac{1}{2}$. By assumption, we then have
$$\theta_1\int_P \eta dx^3\wedge\cdots\wedge dx^n=\lim_{\epsilon\to 0}\int \varphi\omega_{\epsilon}\wedge dx^3\wedge\cdots\wedge dx^n.$$
Fixing trivializations of $L$ over $B_2^n$, we write $\nabla_{\epsilon}=d-iA_{\epsilon}$ for some one-forms $A_{\epsilon}$, so that $\omega_{\epsilon}=dA_{\epsilon}$, and the right-hand term in the preceding limit becomes
\begin{align*}
\int\omega_{\epsilon}\wedge(\varphi dx^3\wedge \cdots\wedge dx^n)&=\int d(\varphi A_{\epsilon}\wedge dx^3\wedge\cdots \wedge dx^n)\\
&\quad+\int A_{\epsilon}\wedge d\varphi \wedge dx^3\wedge \cdots \wedge dx^n\\
&=\int\eta |u_{\epsilon}|^2A_{\epsilon}\wedge d\psi\wedge dx^3\wedge \cdots\wedge dx^n\\
&\quad+\int\eta(1-|u_{\epsilon}|^2)A_{\epsilon}\wedge d\psi\wedge dx^3\wedge\cdots \wedge dx^n.
\end{align*}
On $B_2^n$ we can choose our trivializations so that $d^*A_{\epsilon}=0$, and $A_{\epsilon}(\nu)=0$ on $\partial B_2^n$. We then have the $L^2$ control
\begin{equation}
\int_{B_2^n} |A_{\epsilon}|^2\leq C \int_{B_2^n}|dA_{\epsilon}|^2\leq C\epsilon^{-2}\Lambda,
\end{equation}
and consequently
\begin{align*}
\Big|\int \eta(1-|u_{\epsilon}|^2)A_{\epsilon}\wedge d\psi\wedge dx^3\wedge \cdots\wedge dx^n\Big|&\leq C\|1-|u_{\epsilon}|^2\|_{C^0(\operatorname{spt}(\eta d\psi))}\|A_{\epsilon}\|_{L^1(B_2^n)}\\
&\leq C\Lambda^{1/2}\|\epsilon^{-1}(1-|u_{\epsilon}|^2)\|_{C^0(\operatorname{spt}(\eta d\psi))}\\
&\leq C\Lambda^{1/2}\|e_{\epsilon}(u_{\epsilon},\nabla_{\epsilon})\|_{C^0(\operatorname{spt}(\eta d\psi))}^{1/2}\\
&\to 0
\end{align*}
as $\epsilon\to 0$, where we have used the fact that $d\psi(x^1,x^2) = 0$ for $|(x^1,x^2)|\leq \frac{1}{2}$, and the assumption that $e_{\epsilon}(u_{\epsilon},\nabla_{\epsilon})\to 0$ in $C^0_{loc}(B_2^n\setminus P)$.

Combining our computations thus far, we have arrived at the identity
$$\theta_1\int_P \eta dx^3\wedge\cdots\wedge dx^n=\lim_{\epsilon\to 0}\int \eta |u_{\epsilon}|^2A_{\epsilon}\wedge d\psi\wedge dx^3\wedge \cdots \wedge dx^n.$$
Noting next that 
$$||u_{\epsilon}|^2A_{\epsilon}-\langle du_{\epsilon},iu_{\epsilon}\rangle|=|\langle \nabla_{\epsilon} u_{\epsilon}, iu_{\epsilon}\rangle|\leq e_{\epsilon}(u_{\epsilon},\nabla_{\epsilon})^{1/2},$$
and using again the hypothesis that $e_{\epsilon}(u_{\epsilon},\nabla_{\epsilon})\to 0$ uniformly on $\operatorname{spt}(\eta d\psi)$, the preceding identity yields
\begin{align*}
\theta_1\int_P \eta dx^3\wedge\cdots\wedge dx^n&=\lim_{\epsilon\to 0}\int \eta \langle du_{\epsilon},iu_{\epsilon}\rangle\wedge d\psi \wedge dx^3\wedge \cdots \wedge dx^n\\
&=\lim_{\epsilon\to 0}\int \eta |u_{\epsilon}|^2(u_{\epsilon}/|u_{\epsilon}|)^*(d\theta)\wedge d\psi \wedge dx^3\wedge \cdots \wedge dx^n\\
&=\lim_{\epsilon\to 0}\int \eta (u_{\epsilon}/|u_{\epsilon}|)^*(d\theta)\wedge d\psi \wedge dx^3\wedge \cdots \wedge dx^n.
\end{align*}
Finally, since the one-form $(u_{\epsilon}/|u_{\epsilon}|)^*(d\theta)$ is closed on $\{u_\epsilon\neq 0\}$
and $d\eta\wedge dx^3\wedge\dots\wedge dx^n=0$, integrating by parts on $(\R^2\setminus B_{1/2}^2)\times\R^{n-2}$ we see that
\begin{align*}
	\int \eta (u_{\epsilon}/|u_{\epsilon}|)^*(d\theta)\wedge d\psi \wedge dx^3\wedge \cdots \wedge dx^n
	&=\int_{\R^{n-2}}\eta(t)\int_{\partial B_{1/2}^2\times\{t\}}(u_{\epsilon}/|u_{\epsilon}|)^*(d\theta)\,dt \\
	&=2\pi\operatorname{deg}(u_\epsilon,P)\int_P\eta,
\end{align*}
where $\operatorname{deg}(u_\epsilon,P)$ stands for the degree of $(u_\epsilon/|u_\epsilon|)(\frac{1}{2}e^{i\theta},0)$.
%from a simple application of the coarea formula (C.F., E.G., JERRARD-SONER?) we see that
%\begin{eqnarray*}
%\int \eta (u_{\epsilon}/|u_{\epsilon}|)^*(d\theta)\wedge d\psi \wedge dx^3\wedge \cdots \wedge dx^n&=&\pm\int_P \eta(x') \int_0^1\left(\int_{\partial \{\psi>t\}}(u_{\epsilon}/u_{\epsilon})^*(d\theta)\right)dtdx'\\
%&=&\pm 2\pi \deg(u_{\epsilon},P)\int_P\eta,
%\end{eqnarray*}
The statement follows.
\end{proof}

\section{Examples from variational constructions}\label{minmaxsec}

The goal of this section is to show that, for every closed manifold $M$ and every line bundle $L\to M$ endowed with a Hermitian metric, there exist critical couples $(u_{\epsilon},\nabla_{\epsilon})$ for the Yang--Mills--Higgs functional $E_\epsilon$, for $\epsilon$ small enough, in such a way that
\begin{align}\label{low.and.up.bounds}
	&0<\liminf_{\epsilon\to 0}E_\epsilon(u_\epsilon,\nabla_\epsilon)\le\limsup_{\epsilon\to 0}E_\epsilon(u_\epsilon,\nabla_\epsilon)<\infty.
\end{align}

This will be easier when the line bundle is nontrivial, as in this case we can just take $(u_\epsilon,\nabla_\epsilon)$ to be a global minimizer for $E_\epsilon$. The upper and lower bounds in \eqref{low.and.up.bounds} have the following immediate consequence---proved previously by Almgren \cite{Almvar} using GMT methods.

\begin{cor}
	Every closed Riemannian manifold $(M^n,g)$ supports a nontrivial stationary, integral $(n-2)$-varifold.
\end{cor}

\begin{proof}
	We can always equip $M$ with the trivial line bundle $L:=M\times\mathbb{C}$. As shown in the next subsection,
	there exists a sequence of critical couples $(u_{\epsilon},\nabla_{\epsilon})$ satisfying \eqref{low.and.up.bounds}.
	The statement now follows from Theorem \ref{thm1}.
\end{proof}

\subsection{Min-max families for the trivial line bundle}\hfill

In this section we will show how min-max methods may be applied to the functionals $E_{\epsilon}$ to produce nontrivial critical points in the trivial bundle $L=M\times \mathbb{C}$ on an arbitrary closed, oriented manifold $M$ of dimension $n\geq 2$. The min-max construction that we consider here is based on two-parameter families parametrized by the unit disk, similar to the constructions employed in \cite{Cheng} and \cite{Stern} for the Ginzburg--Landau functionals---with several technical adjustments to account for the gauge-invariance and other features particular to the Yang--Mills--Higgs energies. We remark that the families we consider induce a nontrivial class in $\pi_2(\mathcal{M})$ for the quotient 
$$\mathcal{M}:=\{(u,\nabla)\mid 0\not \equiv u\in \Gamma(L),\text{ }\nabla\text{ a hermitian connection}\}/\{\text{gauge transformations}\},$$
and the analysis that follows can be reformulated in terms of min-max methods applied directly to the Banach manifold $\mathcal{M}$.

Without loss of generality, we assume henceforth that $M$ is connected.
%Note that we are not assuming the orientability of $M$.

\begin{definition}
	Fix $n=\dim(M)<p<\infty$.
	In what follows, $\widehat X$ will denote the Banach space of couples $(u,A)$, where $u\in L^p(M,\mathbb C)$ and $A\in\Omega^1(M,\mathbb R)$, both of class $W^{1,2}$, with the norm
	\begin{align*}
		&\|(u,A)\|:=\|u\|_{L^p}+\|du\|_{L^2}+\|A\|_{L^2}+\|DA\|_{L^2}.
	\end{align*}
	Denote by $X:=\{(u,A)\in\widehat{X}:d^*A=0\}$ the subspace consisting of those couples for which the connection form $A$ is co-closed.
\end{definition}

Note that for $(u,A)\in X$, the full covariant derivative $\int_M|DA|^2$ is bounded by $C(M)\int_M(|A|^2+|dA|^2)$.

\begin{definition}
	Given a form $A\in\Omega^1(M,\mathbb R)$ in $L^2$, we denote by $h(A)$ the harmonic part of its Hodge decomposition, or equivalently the orthogonal projection of $A$ onto the (finite-dimensional) space $\mathcal{H}^1(M)$ of harmonic one-forms. %Similarly, given a function $v:M\to S^1$, $h(v)$ will refer to the harmonic part of $v^*(d\theta)\in L^2$.
\end{definition}

\begin{remark}\label{retraction}
	Selection of a Coulomb gauge gives a continuous retraction $\mathcal{R}:\widehat X\to X$: namely, given a couple $(u,A)\in\widehat X$, consider the unique solution $\theta\in W^{2,2}(M,\mathbb R)$ to the equation
	\begin{align*}
		&\Delta\theta=d^*A,
	\end{align*}
	with $\int_M\theta=0$, and set
	\begin{align*}
		&\mathcal R((u,A)):=(e^{i\theta}u,A+d\theta).
	\end{align*}
	Note that the continuity of $(u,A)\mapsto d(e^{i\theta}u)=e^{i\theta}(du+iud\theta)$, from $\widehat{X}$ to $L^2$, follows from the fact that $L^p\cdot L^{2^*}\subseteq L^2$, where $2^*=\frac{2n}{n-2}$.
\end{remark}

Throughout this section, $W(u)=f(|u|)$ will be a smooth radial function given by
$W(u)=\frac{(1-|u|^2)^2}{4}$ for $|u|\le 3/2$, and satisfying $W(u)>0$ for all $|u|>1$.
For technical reasons, we also find it convenient to require that
\begin{align*}
	&W(u)=|u|^p\text{ for }|u|\geq 2, \tag{G}
\end{align*}
which evidently gives the additional estimates $|u|W'(|u|)+|u|^2W''(|u|)\leq C|u|^p$ for $|u|\geq 2$, for some constant $C$. For future use, observe also that the potential $W(u)$ then satisfies a simple bound of the form
\begin{align}\label{w.lb}
	&(1-|u|)^2\le CW(u).
\end{align}
%Note that (G) also implies that the Hessian of $W$ is positive definite outside some ball, since the fact that $W$ is radial implies
%\begin{align*}
	%&rW''(re^{i\theta})[ie^{i\theta},ie^{i\theta}]=W'(re^{i\theta})[e^{i\theta}]=W'(r)>0.
%\end{align*}

\begin{prop}\label{c1}
	The functional $E_\epsilon$ is of class $C^1$ on $\widehat{X}$. Moreover, a couple $(u,A)$ is critical in $\widehat X$ for $E_\epsilon$ if and only if $\mathcal R((u,A))$ is critical in $X$.
\end{prop}

\begin{proof}
	Given a point $(u,A)\in\widehat X$ and a pair $(v,B)\in \widehat X$ with $\|(v,B)\|_{\widehat X}\le 1$, direct computation gives %Suppressing double brackets, the differential at $(u,A)$ is given by
	\begin{align*}
	E_\epsilon(u+v,A+B)=&E_\epsilon(u,A)+2\int_M\langle du-iu A,dv-ivA-iuB\rangle+ 2\epsilon^2\langle dA,dB\rangle \\
	&+\epsilon^{-2}\int_M W'(u)[v]+O(\|(v,B)\|_{\widehat X}^2),
	\end{align*}
	%where $\langle\ ,\ \rangle$ refers to the metric on $M$, rather than the canonical metric on $L$. [It's just the obvious metric on $L\otimes T*M$, no?]
	where we are using the fact that $\widehat X \cdot \widehat X\subseteq L^n\cdot L^{2^*}\subseteq L^2$ to see that
$$\|A v\|_{L^2}^2+\|Bu\|_{L^2}^2+\|B v\|_{L^2}^2+E_{\epsilon}(u,A)^{1/2}\|Bv\|_{L^2}=O(\|(v,B)\|_{\widehat X}^2),$$
and we invoke our assumptions on the structure of $W$ to see that
$$\int_M (W(u+v)-W(u))=\int_MW'(u)[v]+O(\|(v,B)\|_{\widehat X}^2)$$
for fixed $(u,A)\in \widehat X$.
	It follows immediately that $E_{\epsilon}$ is $C^1$ on $\widehat X$, with gradient
$$dE_{\epsilon}(u,A)[v,B]=2\int_M\langle du-iu A,dv-ivA-iuB\rangle+ 2\epsilon^2\langle dA,dB\rangle+\epsilon^{-2}W'(u)[v].$$

To confirm the second statement, assume without loss of generality that $v$ and $B$ are smooth, and observe that
	\begin{align*}
		&\mathcal R((u+tv,A+tB))=(e^{ti\psi}\tilde u+e^{i\theta+ti\psi}v,\tilde A+tB+td\psi),
	\end{align*}
	where $(\tilde u,\tilde A):=\mathcal R((u,A))=(e^{i\theta}u,A+d\theta)$ and $\psi$ solves $\Delta\psi=d^*B$. This easily gives
	\begin{align*}
		&\mathcal R((u+tv,A+tB))=\mathcal R((u,A))+t(e^{i\theta}v+i\psi\tilde u,B+d\psi)+o(t)\qquad\text{in } X
	\end{align*}
	and, using the gauge invariance $E_\epsilon=E_\epsilon\circ\mathcal R$, we deduce that
	\begin{align}\label{crit.X.X}
		&dE_\epsilon(u,A)[v,B]=dE_\epsilon(\tilde u,\tilde A)[e^{i\theta}v+i\psi\tilde u,B+d\psi].
	\end{align}	
	It follows that if $(\tilde u,\tilde A)$ is critical for $E_{\epsilon}$ in $X$ then $(u,A)$ is critical for $E_{\epsilon}$ in $\widehat X$, as claimed.
\end{proof}

We next show that the functionals $E_{\epsilon}$ satisfy a suitable variant of the Palais--Smale condition on $X$, giving compactness of critical sequences for $E_{\epsilon}$ after an appropriate change of gauge. (Cf. \cite{jost.peng.wang} for similar results in the Seiberg--Witten setting.)

\begin{prop}\label{palais.smale}
	The functional $E_\epsilon$ satisfies the following form of the Palais--Smale condition: every sequence $(u_j,A_j)$ in $X$ with bounded energy and $dE_\epsilon(u_j,A_j)\to 0$ in $X^*$ admits a subsequence converging strongly in $X$ to a critical couple $(u_\infty,A_\infty)$, up to possibly replacing $(u_j,A_j)$ with
	\begin{align*}
		&(\tilde u_j,\tilde A_j):=(u_jv_j,A_j+v_j^*(d\theta))
	\end{align*}
	for suitable smooth harmonic functions $v_j:M\to S^1$.
\end{prop}

\begin{proof}
	First, by assumption, we have that
	\begin{align}\label{ps.one}
		&dE_\epsilon(u_j,A_j)[u_j,0]=o(\|(u_j,0)\|)\qquad\text{as }j\to\infty;
	\end{align}
	that is,
	\begin{align*}
		&2\int_M|du_j-iu_jA_j|^2+\int_M W'(u_j)[u_j]=o(\|(u_j,0)\|).
	\end{align*}
	The first term is bounded by $2E_\epsilon(u_j,A_j)$, hence uniformly bounded as $j\to\infty$. Moreover, it is clear from (G) that $W'(u_j)[u_j]\ge p|u_j|^p-C$, so that
	\begin{align}\label{ps.one.cons}
		&\|u_j\|_{L^p}^p\le C+o(\|du_j\|_{L^2}).
	\end{align}
	% THE FOLLOWING IS BULLSHIT
	%Using also $dE_\epsilon((u_j,A_j))[(0,A_j)]=o(\|(0,A_j)\|)$, we get
	%\begin{align}\label{ps.two}\begin{aligned}
	%	\epsilon^2\|dA\|_{L^2}^2&\le 2\int_M|u_j||du_j||A_j|-2\int_M|u_j|^2|A_j|^2+o(\|A_j\|_{W^{1,2}}) \\
	%	&\le C\|du_j\|_{L^2}\|u_j\|_{L^p}\|A_j\|_{L^{2^*}}+o(\|A_j\|_{W^{1,2}})
	%\end{aligned}\end{align}
	%for any $\delta$.
	Denote by $\Lambda\subset \mathcal{H}^1(M)$ the lattice in the space of harmonic one-forms given by
	\begin{align*}
		\Lambda&:=\{-v_j^*(d\theta)\mid v_j: M\to S^1\text{ harmonic }\}\\
		&=\Big\{h\in \mathcal{H}^1(M) : \int_{\gamma}h\in 2\pi\mathbb{Z}\text{ for every }\gamma\in C^1(S^1,M)\Big\},
	\end{align*}
	and let $\lambda_j\in\Lambda$ be a closest integral harmonic one-form to $h(A_j)$ (with respect to the $L^2$ norm, say, on $\mathcal{H}^1(M)$). Then $\lambda_j=-v_j^*(d\theta)$ for a suitable harmonic map $v_j:M\to S^1$, and
	$$\|\lambda_j-h(A_j)\|_{L^2}\leq C(M).$$
	Replacing $(u_j,A_j)$ with the change of gauge $(v_ju_j,A_j-\lambda_j)\in X$, we can then assume that $h(A_j)$ is bounded. By standard Hodge theory we can write
	\begin{align*}
		&A_j=h(A_j)+d^*\xi_j
	\end{align*}
	for some closed $\xi_j\in W^{2,2}$ satisfying $\Delta_H\xi_j=dA_j$ and $\|d^*\xi_j\|_{W^{1,2}}\le C(M)\|dA_j\|_{L^2}$. %[IS A REFERENCE NEEDED HERE?]. I don't think so...
Thus, given the energy bound $E_{\epsilon}(u_j)\leq C$, we see that
	\begin{align*}
		&\|A_j\|_{W^{1,2}}^2
		\le C+2\|d^*\xi_j\|_{W^{1,2}}^2
		\le C+C\|dA_j\|_{L^2}^2
		\le C,
	\end{align*}
	whereby $A_j$ is bounded in $W^{1,2}$ and, consequently, in $L^{2^*}$. As a consequence, we see next that
	\begin{align*}
		&\|du_j\|_{L^2}^2\le 2\int_M|du_j-iu_jA_j|^2+2\int_M|u_jA_j|^2
		\le C+C\|u_j\|_{L^p}^2\|A_j\|_{L^{2^*}}^2\le C+C\|u_j\|_{L^p}^p;
	\end{align*}
	taking into account \eqref{ps.one.cons}, we infer then that $\|du_j\|_{L^2}$ is also bounded as $j\to\infty$. 

We have therefore shown that $(u_j,A_j)$ is uniformly bounded in $X$ as $j\to\infty$, so passing to subsequences we can assume that $(u_j,A_j)$ converges pointwise a.e. and weakly (in $X$) to a limiting couple $(u_\infty,A_\infty)$. In particular, defining $r$ by
	\begin{align*}
		&\frac{1}{r}:=\frac{1}{2}-\frac{1}{q}>\frac{1}{2}-\frac{1}{n}=\frac{1}{2^*},
	\end{align*}
	where $n<q<p$ is an arbitrary fixed exponent, it follows from the compactness of the embedding $W^{1,2}\hookrightarrow L^{r}$ that 
$$A_j\to A_\infty\text{ strongly in }L^r.$$
 Moreover, the boundedness of $u_j$ in $L^p$ and the pointwise convergence to $u_\infty$ give
\begin{align}\label{q.conv}
	&u_j\to u_\infty\text{ strongly in }L^q. %\quad W'(u_j)\to W'(u_\infty)\text{ in }L^1,
\end{align}
By definition of $r$, this implies in particular that
\begin{align*}
	&\lim_{j,k\to\infty}u_j A_k=u_\infty A_\infty\text{ strongly in }L^2.
\end{align*}
Next, compute
\begin{align*}
	dE_{\epsilon}(u_j,A_j)[u_j-u_k,A_j-A_k]&=\int 2\langle (d-iA_j)u_j, (d-iA_j)(u_j-u_k)-iu_j (A_j-A_k)\rangle\\
	&\quad+\int (2\epsilon^2\langle dA_j,d(A_j-A_k)\rangle+\epsilon^{-2}W'(u_j)[u_j-u_k]),
\end{align*}
and observe that the $L^2$ convergence $u_jA_k\to u_{\infty}A_{\infty}$ gives
\begin{align*}
	dE_{\epsilon}(u_j,A_j)[u_j-u_k,A_j-A_k]&=\int 2\langle (d-iA_j)u_j, d(u_j-u_k)\rangle+2\epsilon^2\langle dA_j,d(A_j-A_k)\rangle\\
	&\quad+\epsilon^{-2}W'(u_j)[u_j-u_k]+o(1)
\end{align*}
as $j,k\to\infty$. For the difference 
$$D_{j,k}:=d E_{\epsilon}(u_j,A_j)[u_j-u_k,A_j-A_k]-dE_{\epsilon}(u_k,A_k)[u_j-u_k,A_j-A_k],$$
we then see that 
$$D_{j,k}=\int 2 |d(u_j-u_k)|^2+2\epsilon^2|d(A_j-A_k)|^2+\epsilon^{-2}(W'(u_j)-W'(u_k))[u_j-u_k]+o(1)$$
as $j,k\to\infty$.

Now, by our assumptions (G) on the structure of $W(u)$, it is not difficult to check (see, e.g., \cite[Corollary~1]{HLM}) that the zeroth order term in our computation for $D_{j,k}$ satisfies a lower bound
$$(W'(u_j)-W'(u_k))[u_j-u_k]\geq |u_j-u_k|^p-C|u_j-u_k|$$
for some constant $C$. In particular, it follows now from the preceding computations and the $L^1$ convergence $u_j\to u_{\infty}$ that
$$D_{j,k}\geq \int 2|d(u_j-u_k)|^2+2\epsilon^2|d(A_j-A_k)|^2+\epsilon^{-2}|u_j-u_k|^p+o(1)$$
as $j,k\to\infty$. On the other hand, since $d E_{\epsilon}(u_j,A_j)\to 0$ and $(u_j-u_k,A_j-A_k)$ is bounded in $X$, we know also that
$$D_{j,k}\to 0\text{ as }j,k\to\infty,$$
and it then follows that $(u_j,A_j)$ is Cauchy in $X$. In particular, $(u_j,A_j)$ converges strongly to $(u_{\infty},A_{\infty})$, which necessarily satisfies
\begin{align*}
	&dE_{\epsilon}(u_{\infty},A_{\infty})=\lim_{j\to\infty}dE_{\epsilon}(u_j,A_j)=0. \qedhere
\end{align*}
\end{proof}

Having confirmed that the energies $E_{\epsilon}$ satisfy a Palais--Smale condition, we now argue in roughly the same spirit as \cite{Cheng}, \cite{Stern} to produce nontrivial critical points via min-max methods. To begin, note that the space $X$ splits as $\mathbb C\oplus Y$, where $\mathbb C$ is identified with the set of constant couples $(\alpha,0)$ and
\begin{align*}
	&Y:=\Big\{(u,A)\in X:\int_M u=0\Big\}.
\end{align*}

\begin{definition}
	Let $\Gamma$ denote the set of continuous families of couples $F:\bar D\to X$ parametrized by the closed unit disk $\bar D$, with
	\begin{align*}
		&F(e^{i\theta})=(e^{i\theta},0)
	\end{align*}
	for all $\theta\in\mathbb R$. Equivalently, under the above identification $\mathbb C\subset X$, we require $F|_{\partial D}=\operatorname{id}$. We denote by $\omega_\epsilon(M)$ the ``width" of $\Gamma$ with respect to the energy $E_\epsilon$, namely
	\begin{align*}
		&\omega_\epsilon(M):=\inf_{F\in\Gamma}\max_{y\in\bar D}E_\epsilon(F(y)).
	\end{align*}	
\end{definition}

Thanks to Proposition \ref{palais.smale}, we can apply classical min-max theory for $C^1$ functionals on Banach spaces (see e.g. \cite[Theorem~3.2]{Ghou}) to conclude that $\omega_\epsilon$ is achieved as the energy of a smooth critical couple $(u_\epsilon,A_\epsilon)$. In the following proposition, we show that $\omega_{\epsilon}(M)$ is positive, so that the corresponding critical couples $(u_{\epsilon},A_{\epsilon})$ are nontrivial. 

\begin{prop}\label{nontrivprop}
	We have $\omega_\epsilon(M)>0$.
\end{prop}

\begin{proof}
	We argue by contradiction, though the proof could be made quantitative. Since we are proving only the positivity $\omega_{\epsilon}(M)>0$ at this stage---making no reference to the dependence on $\epsilon$---in what follows we take $\epsilon=1$ for convenience.
	Assume that we have a family $F\in\Gamma$ with $\max_{y\in\bar D}E(F(y))<\delta$, with $\delta$ very small.
	Writing $F(y)=(u,A)$, this implies that
	\begin{align}\label{hodge.bounds}
		&\|A-h(A)\|_{W^{1,2}}\le C\|dA\|_{L^2}<C\delta^{1/2},\qquad\|DA\|_{L^2}\le C(\delta^{1/2}+\|h(A)\|).
	\end{align}
	When $b_1(M)\neq 0$, some additional work is required to deduce that the harmonic part $h(A)$ of $A$ must also be small for all couples $(u,A)=F(y)$ in the family. In particular, we will need to employ the following lemma, showing that $h(A)$ lies close to the integral lattice $\Lambda\subset \mathcal{H}^1(M)$ when $E(u,A)<\delta.$
\begin{lem} There exists $C(M)<\infty$ such that if $(u,A)\in X$ satisfies $E(u,A)<\delta$, then
$$\operatorname{dist}(h(A),\Lambda)\leq C\delta^{1/2}.$$
\end{lem}

\begin{proof} As in \cite{Stern}, it is convenient to define a box-type norm $|\cdot|_b$ on the space $\mathcal{H}^1(M)$ of harmonic one-forms as follows. Fix a collection $\gamma_1,\ldots,\gamma_{b_1(M)}\in C^{\infty}(S^1,M)$ of embedded loops generating $H_1(M;\mathbb{Q})$ and, for $h\in \mathcal{H}^1(M)$, set
\begin{equation}
|h|_b:=\max_{1\leq i\leq b_1(M)}\Big|\int_{\gamma_i}h\Big|.
\end{equation}
Since $\mathcal{H}^1(M)$ is finite-dimensional, this is of course equivalent to any other norm on $\mathcal{H}^1(M)$. Since $M$ is orientable, we may fix a collection of diffeomorphims $\Phi_i:B_1^{n-1}(0)\times S^1\to T(\gamma_i)$ onto tubular neighborhoods $T(\gamma_i)$ of $\gamma_i$, such that $\Phi_i(0,\theta)=\gamma_i(\theta)$. For every $t\in B_1^{n-1}$, set $\gamma_i^t(\theta):=\Phi_i(t,\theta)$. 

Suppose now that $(u,A)\in X$ satisfies the energy bound 
\begin{equation}\label{intlembd}
E(u,A)=\int_M|du-iAu|^2+|dA|^2+W(u)<\delta.
\end{equation}
As a consequence of the curvature bound $\|dA\|_{L^2}\leq \delta^{1/2}$ and the definition of $X$, it follows that
$$\|A-h(A)\|^2_{L^2}\leq C\delta$$
as well. As in the proof of Proposition \ref{palais.smale}, applying a gauge transformation $\phi\cdot (u,A)$ by an appropriate choice of harmonic map $\phi:M\to S^1$, we may assume moreover that 
$$|h(A)|_b=\operatorname{dist}(h(A),\Lambda)\leq \pi,$$
which together with the energy bound \eqref{intlembd} and the definition of $X$ leads us to the estimate
\begin{equation}\label{l2abd}
\int_M|A|^2\leq C(M).
\end{equation}
(Note that making a harmonic change of gauge preserves not only the energy $E(u,A)$, but also the distance $\operatorname{dist}(h(A),\Lambda)$, so it indeed suffices to establish the desired estimate in this gauge.)

Combining these estimates with a simple Fubini argument, we see that there exists a nonempty set $S$ of $t\in B_1^{n-1}$ for which
\begin{equation}\label{gc1}
\int_{\gamma_i^t}|du-iAu|^2+|dA|^2+W(u)<C\delta,
\end{equation}
\begin{equation}\label{gc2}
\int_{\gamma_i^t}|A-h(A)|^2<C\delta,
\end{equation}
and
\begin{equation}\label{gc3}
\int_{\gamma_i^t}|A|^2\leq C.
\end{equation}
Recalling the pointwise bound \eqref{w.lb} for $W(u)$, observe next that
$$|d(1-|u|)^2|=2(1-|u|)|d|u||\leq CW(u)+|du-iAu|^2,$$
so that, along a curve $\gamma_i^t$ satisfying \eqref{gc1}, it follows that
\begin{equation}\label{uc0bd}
\|(1-|u|)^2\|_{C^0}\leq C\|(1-|u|)^2\|_{W^{1,1}}\leq C\delta.
\end{equation}

Now, choose $\delta<\delta_1(M)$ sufficiently small that \eqref{uc0bd} gives
$$\|1-|u|\|_{C^0}\leq \eta<\frac{1}{2}$$
on $\gamma_i^t$, so that $\phi=u/|u|$ defines there an $S^1$-valued map $\phi: \gamma_i^t\to S^1$, whose degree is given by
$$2\pi \operatorname{deg}(\phi)=\int_{\gamma_i^t}|u|^{-2}\langle du,iu\rangle.$$
When \eqref{gc1}--\eqref{gc3} hold, we observe next that
$$\int_{\gamma_i^t}|u|^2|A-|u|^{-2}\langle iu,du\rangle|=\int_{\gamma_i^t}|\langle iu, iAu-du\rangle|\leq C\delta^{1/2}.$$
Since $|u|\geq \frac{1}{2}$ on $\gamma_i^t$, it follows that
\begin{equation}
\Big|2\pi \operatorname{deg}(\phi)-\int_{\gamma_i^t}A\Big|\leq \int_{\gamma_i^t} |A-|u|^{-2}\langle iu,du\rangle|\leq C\delta^{1/2}
\end{equation}
as well. Combining this with \eqref{gc2}, we then deduce that
$$\Big|2\pi \operatorname{deg}(\phi)-\int_{\gamma_i^t}h(A)\Big|\leq C\delta^{1/2}.$$
On the other hand, we already made a gauge transformation so that 
$$\Big|\int_{\gamma_i}h(A)\Big|=\Big|\int_{\gamma_i^t}h(A)\Big|\leq \pi,$$
so for $\delta$ chosen sufficiently small that $C\delta^{1/2}<\pi$, it follows that the degree $\operatorname{deg}(\phi)=0$. 
In particular, we can now conclude that
$$|h(A)|_b=\max_i\Big|\int_{\gamma_i}h(A)\Big|\leq C\delta^{1/2},$$
giving the desired estimate.
\end{proof}

Returning to the proof of Proposition \ref{nontrivprop}, suppose again that we have a family $\bar D\ni y \mapsto F(y)\in X$ in $\Gamma$ with
$$\max_{y\in \bar D}E(F(y))<\delta.$$
For $\delta<\delta_1(M)$ sufficiently small, it follows from the lemma that $\operatorname{dist}_b(h(A),\Lambda)<\pi$ for every couple $(u,A)=F(y)$ in the family. In particular, since the assignment $(u,A)\mapsto h(A)$ gives a continuous map $X\to \mathcal{H}^1(M)$, and since $h(A)=A=0$ for $y\in \partial \bar D$, it follows that $0$ is the nearest point in the lattice $\Lambda$ to $h(A)$ for every $y\in \bar D$, and the estimate therefore becomes
$$\|h(A)\|\leq C\delta^{1/2}.$$
In particular, combining this with \eqref{hodge.bounds}, we see now that 
\begin{equation}
\|A\|_{W^{1,2}}\leq C\delta^{1/2}
\end{equation}
for every couple $(u,A)=F(y)$ in the family.

	Now, for $(u,A)=F(y)$, our structural assumption (G) on $W(u)$ gives 
$$\|u\|_{L^p}^p\leq C+E(u,A)\leq C+\delta,$$
which together with the smallness 
$$\|A\|_{L^{2^*}}\leq C\|A\|_{W^{1,2}}\leq C\delta^{1/2}$$
of $A$ in $L^{2^*}$ (recalling that $p>n$) gives 
$$\int_M|uA|^2\leq C\delta.$$ 
Combining this with the fact that $\int |du-iuA |^2\leq E(u,A)<\delta$ by assumption, we then deduce that
$$\int_M |du|^2\leq C\delta$$
as well.

Finally, by \eqref{w.lb} and the Poincaré inequality, we have
	\begin{align*}
		1-\Big|\frac{1}{\operatorname{vol}(M)}\int_M u\Big|&\le C\int_M|1-|u||+C\int_M\Big|u-\frac{1}{\operatorname{vol}(M)}\int_M u\Big| \\
		&\le C\Big(\int_M W(u)\Big)^{1/2}+C\Big(\int_M|du|^2\Big)^{1/2}\\
		&\le C \delta^{1/2}.
	\end{align*}
	As a consequence, we find that $\int_M u_y$ is nonzero for all $(u_y,A_y)=F(y)$ in the family. But then the averaging map
\begin{align}
		&\bar D \to \mathbb{C},\qquad y \mapsto \frac{\int_M u_y}{|\int_M u_y|}.
	\end{align}
gives a retraction $\bar D\to \partial \bar D$, whose nonexistence is well known. This gives the desired contradiction.
\end{proof}

Having shown positivity $\omega_{\epsilon}(M)>0$ of the min-max energies, we can now deduce the lower bound in \eqref{low.and.up.bounds} from the following simple fact.

\begin{prop}\label{lower.bound}
	There exists $c(M)>0$ and $\epsilon_0(M)>0$ such that the following holds, for $\epsilon<\epsilon_0$.
	If $(u,\nabla)$ is critical for the functional $E_\epsilon$, then either $E_\epsilon(u,\nabla)\ge c$ or $E_\epsilon(u,\nabla)=0$.
\end{prop}

\begin{remark}\label{lb.nontriv} For future reference, we make the obvious observation that the trivial case $E_{\epsilon}(u,\nabla)=0$ can only occur when the bundle $L$ is trivial.
\end{remark}

\begin{proof}
	As discussed in the appendix, it is straightforward to see that critical points are smooth up to change of gauge.
	We claim that, whenever $E_\epsilon(u,\nabla)>0$, $u$ has to vanish at some point $x_0\in M$.
	Once we have this, Corollary \ref{clearing.out} implies that $r^{2-n}E_\epsilon(u,\nabla,B_r(x_0))$ has a lower bound independent of $\epsilon$ for any $\epsilon<r<\operatorname{inj}(M)$, and the statement follows.
	
	Indeed, if the claim fails, then $u$ is nowhere vanishing, so $L$ must be trivial and we can use the section $\frac{u}{|u|}$ to identify $L$ isometrically with the trivial line bundle $M\times\mathbb C$, equipped with the canonical Hermitian metric.
	Under this identification, $u:M\to\mathbb C$ takes values into positive real numbers. Writing $\nabla=d-iA$ and observing that $\langle\nabla u,iu\rangle=-|u|^2A$, \eqref{feq} becomes
	\begin{align*}
	&\epsilon^2 d^*dA+|u|^2A=0.
	\end{align*}
	Integrating against $A$ we get $\int_M(\epsilon^2|dA|^2+u^2|A|^2)=0$, so $A=0$ and $\nabla$ is the trivial connection.
	At a minimum point $y_0$ for $u$, \eqref{moduboch} gives
	\begin{align*}
	&0\le\frac{1}{2}\Delta|u|^2=|du|^2-\frac{1}{2\epsilon^2}(1-|u|^2)|u|^2=-\frac{1}{2\epsilon^2}(1-u^2)u^2,
	\end{align*}
	which forces $u(y_0)\ge 1$ and thus $u=1$ everywhere, giving the contradiction $E_\epsilon(u,\nabla)=0$.
\end{proof}

Finally, we turn to the uniform upper bound.
In the next statement, $L\to M$ is a Hermitian line bundle with a fixed Hermitian reference connection $\nabla_0$.
We identify any other Hermitian connection $\nabla$ with the real one-form $A$ such that $\nabla s=\nabla_0 s-iA\otimes s$ for all sections $s$.

\begin{prop}\label{u.to.couple}
	Given a smooth section $u:M\to L$, we can find a smooth couple $(u',A')$ such that
	\begin{align}\label{u.to.couple.est}
		&E_\epsilon(u',A')\le C\left(\epsilon^{-2}\operatorname{vol}\Big(\Big\{|u|\le\frac{1}{2}\Big\}\Big)+(1+\epsilon^2\|\nabla_0 u\|_{L^\infty}^2)\int_{\{|u|\le\frac{1}{2}\}}|\nabla_0 u|^2+\epsilon^2\int_M|\omega_0|^2\right)
		%C\epsilon^{-2}\operatorname{vol}(\{|u|\le 1\})(1+\epsilon^2[u]^2+\epsilon^4[u]^4+\epsilon^4\|\omega_0\|_{L^\infty}^2),
	\end{align}
	for a universal constant $C$.
	%where we set $[u]:=\sup_{\{u\neq 0\}}|u|\Big|\nabla_0\frac{u}{|u|}\Big|+\sup|\nabla_0 u|$.
\end{prop}

\begin{proof}
	On $\{u\neq 0\}$ we let
	\begin{align*}
	&w:=\frac{u}{|u|},\qquad A\otimes iw:=\nabla_0 w.
	\end{align*}
	The compatibility of $\nabla_0$ with the Hermitian metric on $L$ forces $\langle\nabla_0w,w\rangle=0$, so that $A$ is a real one-form. Equivalently, viewing $w$ as a map from $M$ to the circle bundle $U(L)$ of $L$, which is a principal $S^1$-bundle with induced connection form $\varpi\in\Omega^1(U(L),\mathbb R)$, we have $A=\varpi\circ dw$.
	
	We fix a smooth function $\rho:[0,\infty]\to [0,1]$ with
	\begin{align*}
	&\rho(t)=0\text{ for }t\le\frac{1}{4},\qquad\rho(t)=1\text{ for }t\ge\frac{1}{2}
	\end{align*}
	and we set
	\begin{align*}
	&(u',A'):=\rho(|u|)(w,A)
	\end{align*}
	(the right-hand side is meant to be zero on $\{u=0\}$).
	Writing $F_{\nabla_0}=-i\omega_0$, observe that
	\begin{align*}
		&|F_A|=|dA+\omega_0|=0\quad\text{on }\{u\neq 0\},
	\end{align*}
	so that $e_\epsilon(u',A')=0$ on $\{|u|>\frac{1}{2}\}$. From the estimates $|A|=|\nabla_0w|\le 2|u|^{-1}|\nabla_0 u|$ and $|d|u||\le|\nabla_0 u|$, it follows that also
	\begin{align*}
	|\nabla_0 u'|&\le C|\nabla_0 u|, \\
	|A'|&\le C|\nabla_0 u|, \\
	|dA'|&\le|\rho'(|u|)d|u|\wedge A|+|\omega_0|\le C|\nabla_0 u||d|u||+|\omega_0|,
	\end{align*}
	and the statement follows immediately.
\end{proof}

\begin{proof}[Proof of \eqref{low.and.up.bounds}]
	The method used in \cite[Section~3]{Stern} gives a continuous map $H:\bar D\to W^{1,2}\cap C^0(M,\mathbb C)$ such that $H(y)\equiv y$ for $y\in\partial D$ and
	\begin{align}\label{bounds.from.simplex}
		&\|dH(y)\|_{L^\infty}\le C\epsilon^{-1},\quad\int_{\{|H(y)|\le\frac{3}{4}\}}|dH(y)|^2\le C,\quad\operatorname{vol}\Big(\Big\{|u|\le\frac{1}{2}\Big\}\Big)\le C\epsilon^2
	\end{align}
	for all $y\in\bar D$ (the full Dirichlet energy having a worse bound $\int_M|dH(y)|^2\le C\log\epsilon^{-1}$, which is the natural one in the setting of Ginzburg--Landau). By approximation, we can assume that $H$ takes values in $C^\infty(M,\mathbb C)$, continuously in $y$, and still satisfies the same uniform bounds \eqref{bounds.from.simplex} (possibly increasing $C$ and replacing $\frac{3}{4}$ with $\frac{1}{2}$).
	
	To each section $H(y)$ of the trivial line bundle, Proposition \ref{u.to.couple} assigns in a continuous way an element $F(y)\in X$. From the way $F(y)$ is constructed, it is clear that $F\in\Gamma$. Finally, applying
	Proposition \ref{u.to.couple} with \eqref{bounds.from.simplex} gives
	\begin{align*}
		&\omega_\epsilon(M)\le\max_{y\in\bar D} E_\epsilon(F(y))\le C. \qedhere
	\end{align*}
\end{proof}

%\begin{remark}
%	If $M$ is not orientable, we can use an equivariant min-max on the oriented cover $\pi:\tilde M\to M$ of degree two.
%	We let $\tilde L:=\pi^* L$ and we endow $\tilde M$ and $\tilde L$ with the pullback metrics.
%	The Banach spaces $\widehat X$ and $X$ now consist of couples $(u,A)$ with $u:\tilde M\to\mathbb C$ and $A\in\Omega^1(M,\mathbb R)$, with the requirement that
%	\begin{align*}
%		&u\circ i=u,\qquad 
%	\end{align*}
%\end{remark}

\subsection{Minimizers for nontrivial line bundles}\hfill

%-Easy version: show that absolute minimizers on a nontrivial bundle satisfy uniform energy bounds as $\epsilon \to 0$. (Pick some smooth or polyhedral representative of $(n-2)$-homology class Poincare dual to $c_1(L)$; glue on vortex solutions transversally.)
%
%-More involved: show that absolute minimizers have energy concentrating on area-minimizing current in the homology class. (First, show that limiting energy is bounded above by mass of area-minimizing current; then maybe show that $*(1-|u|^2)F_{\nabla}$ converges to current in appropriate homology class, whose weight measure is bounded above by the limiting energy measure?)

Suppose now that $L$ is a nontrivial line bundle, equipped with a Hermitian metric. Fix a smooth Hermitian connection $\nabla_0$ and identify any other Hermitian connection $\nabla$ with the real one-form $A$ such that
\begin{align*}
	&\nabla=\nabla_0-iA.
\end{align*}
We can define $\widehat X$ and $X$ as in the previous subsection. With this notation, observe that the curvature of $\nabla$ is given by
\begin{align*}
	&F_\nabla=F_{\nabla_0}-idA.
\end{align*}
Hence, writing $F_{\nabla_0}=-i\omega_0$, we have
\begin{align*}
	&E_\epsilon(u,\nabla)=\int_M|\nabla_0 u-A\otimes iu|^2+\epsilon^{-2}\int_M W(u)+\epsilon^2\int_M|\omega_0+dA|^2.
\end{align*}

\begin{definition}
	For a fixed $n<p<\infty$, we define	$\widehat X$ to be the Banach space of couples $(u,A)$, where $u:M\to L$ is an $L^p$ section and $A\in\Omega^1(M,\mathbb R)$, both of class $W^{1,2}$, with the norm
	\begin{align*}
	&\|(u,A)\|:=\|u\|_{L^p}+\|\nabla_0 u\|_{L^2}+\|A\|_{L^2}+\|DA\|_{L^2}.
	\end{align*}
	We let $X:=\{(u,A)\in\widehat{X}:d^*A=0\}$.
\end{definition}

The analogous statements to Remark \ref{retraction} and Propositions \ref{c1} and \ref{palais.smale} hold, with identical proofs (replacing $du$ and $uA$ with $\nabla_0 u$ and $A\otimes u$, respectively).

Arguing as in the proof of Proposition \ref{palais.smale}, it is easy to see that a minimizing sequence of couples for $E_{\epsilon}$ converges---in the appropriate Coulomb gauge---to a global minimizer $(u_{\epsilon},A_{\epsilon})$. We now show that the energy of these minimizers enjoys uniform upper and lower bounds as $\epsilon\to 0$.

\begin{proof}[Proof of \eqref{low.and.up.bounds}]
	The lower bound in \eqref{low.and.up.bounds} follows directly from Proposition \ref{lower.bound}.	
	In order to obtain the upper bound, pick a smooth section $s:M\to L$ transverse to the zero section (see, e.g., \cite[Theorem~IV.2.1]{Kos}) and let $N:=\{s=0\}$, which is a smooth embedded $(n-2)$-submanifold of $M$.
	Proposition \ref{u.to.couple} applied to $\epsilon^{-1}s$ gives a couple $(u_\epsilon',A_\epsilon')$ with
	\begin{align*}
		&E_\epsilon(u_\epsilon',A_\epsilon')\le C\epsilon^{-2}\operatorname{vol}\Big(\Big\{|\epsilon^{-1}s|\le\frac{1}{2}\Big\}\Big)+C\epsilon^2\int_M|\omega_0|^2.
	\end{align*}
	By transversality of $s$, the set $\{|s|\le\frac{\epsilon}{2}\}$ is contained in a $C(s)\epsilon$-neighborhood of $N$, whose volume is bounded by $C(s)\epsilon^2$. We infer that
	\begin{align*}
		&E_\epsilon(u_\epsilon,A_\epsilon)\le E_\epsilon(u_\epsilon',A_\epsilon')\le\epsilon^{-2}\operatorname{vol}\Big(\Big\{|s|\le\frac{\epsilon}{2}\Big\}\Big)\le C. \qedhere
	\end{align*}
\end{proof}

\begin{remark}
	$N$ can be oriented in such a way that $[N]\in H_{n-2}(M,\mathbb R)$ is Poincaré dual to the Euler class $e(L)\in H^2(M,\mathbb R)$ of the line bundle, which equals the first Chern class $c_1(L)$. The fact that the energy of our competitors concentrates along $N$ suggests that, given a sequence of global minimizers $(u_\epsilon,A_\epsilon)$, up to subsequences the corresponding energy concentration varifold is induced by an integral mass-minimizing current whose homology class is Poincaré dual to $c_1(L)$. Theorem \ref{curvthm} provides the natural candidate $\Gamma$, which also satisfies $|\Gamma|\le\mu$.
\end{remark}

\appendix
\section*{Appendix. Interior regularity in the Coulomb gauge}
\renewcommand{\thesection}{A}
\setcounter{thm}{0}
\setcounter{equation}{0}

In this short appendix, we describe the essential ingredients needed to establish local regularity in the Coulomb gauge for finite-energy critical points $(u,A)$ of the ($\epsilon=1$) abelian Higgs energy $E(u,A)$, collecting some estimates which will be of use elsewhere in the paper.

Consider the manifold with boundary $(\bar\Omega^n,g)$ given by a smooth, contractible domain $\Omega^n\subset \mathbb{R}^n$ equipped with a $C^2$ metric $g$, and let $L\cong \mathbb{C}\times \Omega$ be the trivial line bundle over $\Omega$, with the standard Hermitian structure. With respect to the metric $g$, we then define the Yang--Mills--Higgs energies
$$E(u,A):=\int_{\Omega}e(u,A)=\int_{\Omega} |du-iA|^2+|dA|^2+W(u)$$
as in the preceding section. By Proposition \ref{c1} in the preceding section, it is easy to see that a pair $(u,A)$ in $W^{1,2}$ with
\begin{equation}\label{ulinfty}
|u|\leq 1
\end{equation}
is a critical point for $E$ (with respect to smooth perturbations supported in $\Omega$) if and only if the equations
\begin{align}\label{appaeq}
d^*dA&=\langle du-iAu, iu\rangle, \\[5pt]
\label{appueq}
\Delta u&=2\langle i du,A\rangle+|A|^2u-\frac{1}{2}(1-|u|^2)u-i(d^*A)u
\end{align}
are satisfied distributionally in $\Omega$, where all geometric quantities and operators are defined with respect to the metric $g$.

Now, given a pair $(u,A)$ in $W^{1,2}$ satisfying \eqref{appaeq}--\eqref{appueq} and
\begin{equation}\label{appebd}
E(u,A)\leq \Lambda<\infty,
\end{equation}
we can select a \emph{local Coulomb gauge} adapted to $A$ as follows. Denote by $\theta\in W^{2,2}(\Omega,\mathbb{R})$ the unique solution of the Neumann problem
\begin{equation}
\Delta \theta=d^*A\text{ in }\Omega;\quad\frac{\partial \theta}{\partial\nu}=-A(\nu)\text{ on }\partial\Omega
\end{equation}
with zero mean $\int_{\Omega}\theta=0$. Then the gauge-transformed pair
$$(\tilde{u},\tilde{A}):=(e^{i\theta}u,A+d\theta)$$
continues to satisfy \eqref{appaeq}--\eqref{appueq}, with 
$$E(\tilde{u},\tilde{A})=E(u,A)\leq \Lambda,$$
but now with the additional constraints
\begin{equation}\label{coul}
d^*\tilde{A}=0\text{ on }\Omega;\quad\tilde{A}(\nu)=0\text{ on }\partial\Omega.
\end{equation}

For the remainder of the section, we will assume that the pair $(u,A)$ is already in the Coulomb gauge on $\Omega$, so that $A$ satisfies (\ref{coul}). Note that the last term in the equation \eqref{appueq} then vanishes, so that we have
\begin{equation}\label{ueq.coul}
\Delta u=2\langle idu,A\rangle+|A|^2u-\frac{1}{2}(1-|u|^2)u.
\end{equation}
We now establish the local regularity for critical points $(u,A)$ in the Coulomb gauge, giving in particular estimates for $(u,A)$ in $W^{2,q}$ norms.

\begin{prop}\label{w2q.prop} Let $(u,A)$ solve \eqref{appaeq}--\eqref{appueq} in the Coulomb gauge \eqref{coul} on $(\Omega,g)$, with $|u|\leq 1$. If
\begin{equation}\label{e.linfty.bd}
E(u,A;\Omega)\leq \Lambda
\end{equation}
and
\begin{equation}\label{met.control}
\|g\|_{C^2}+\|g^{-1}\|_{C^2}\leq \Lambda,
\end{equation}
then for every compactly supported subdomain $\Omega'\subset\subset \Omega$ and $q\in (1,\infty)$ there exists $C_q(\Lambda,\Omega,\Omega')<\infty$ such that
\begin{equation}
\|u\|_{W^{2,q}(\Omega')}+\|A\|_{W^{2,q}(\Omega')}\leq C_q.
\end{equation}
\end{prop}

\begin{proof} To begin, note that \eqref{appaeq}, \eqref{coul}, and standard Bochner--Weitzenb\"{o}ck identities give the (weak) subequation
\begin{align*}
\Delta \frac{1}{2}|A|^2&\geq -\langle \Delta_HA,A\rangle+|DA|^2+\operatorname{Ric}(A,A)\\
&\geq -C(\Lambda)|A|^2-C|du-iAu|
\end{align*}
for $|A|^2$.
On the other hand, as in Section \ref{bochsec}, we also obtain from \eqref{appueq} the relation
$$\Delta \frac{1}{2}|u|^2=|du-iAu|^2-\frac{1}{2}(1-|u|^2)|u|^2,$$
and combining the two (recalling also that $|u|\leq 1$), we find an estimate of the form
\begin{equation}\label{a.u.subeq}
\frac{1}{2}\Delta (|A|^2+|u|^2)\geq -C(|A|^2+|u|^2)-C.
\end{equation}
Applying Moser iteration to \eqref{a.u.subeq}, we see in particular that
\begin{equation}
\|A\|_{L^{\infty}(\Omega_1)}^2\leq C(\Lambda,\Omega_1,\Omega)(1+\|A\|_{W^{1,2}(\Omega)}^2)
\end{equation}
for any $\Omega_1\subset\subset\Omega$. Moreover, by standard estimates for one-forms $A$ satisfying \eqref{coul} (see, e.g., \cite{ISS}), we have the global $L^2$ bound
$$\|A\|_{W^{1,2}(\Omega)}\leq \|dA\|_{L^2(\Omega)}\leq \Lambda^{1/2},$$
which together with the preceding $L^{\infty}$ estimate gives
\begin{equation}
\|A\|_{L^{\infty}(\Omega_1)}\leq C(\Lambda,\Omega_1,\Omega)
\end{equation}
for any subdomain $\Omega_1\subset\subset\Omega$.

Now, fixing some intermediate domain $\Omega'\subset\subset\Omega_1\subset\subset \Omega$ between $\Omega'$ and $\Omega$, the equation \eqref{ueq.coul} together with the $L^{\infty}(\Omega_1)$ estimate for $A$ give pointwise bounds of the form
\begin{equation}\label{lap.u.bd}
|\Delta u|\leq C(\Omega_1,\Omega,\Lambda)(|du|+1)\text{ in }\Omega_1.
\end{equation}
And since 
$$|du|\leq |du-iAu|+|A|\leq e(u,A)+C$$
in $\Omega_1$, we obtain from the energy bound $E(u,A)\leq \Lambda$ and \eqref{lap.u.bd} the simple estimate
$$\|\Delta u\|_{L^2(\Omega_1)}\leq C(\Omega_1,\Omega,\Lambda),$$
and consequently
$$\|u\|_{W^{2,2}(\Omega_2)}\leq C$$
for any $\Omega'\subset\subset\Omega_2\subset\subset\Omega_1$. Returning to the pointwise bound \eqref{lap.u.bd}, we can now employ a simple iteration argument---combining $L^q$ regularity theory with the Sobolev embedding $W^{2,r}\hookrightarrow W^{1,\frac{rn}{n-r}}$---over successive domains between $\Omega'$ and $\Omega$, to arrive at the desired $W^{2,q}$ estimates for $u$. 

Returning finally to the equation \eqref{appaeq} for $A$, in the Coulomb gauge, we see that
$$\Delta_HA=\langle du-iAu,iu\rangle,$$
and it therefore follows from the preceding estimates that
$$\|A\|_{L^{\infty}(\Omega'')}+\|\Delta A\|_{L^{\infty}(\Omega'')}\leq C(\Omega'',\Omega,\Lambda)$$
for some intermediate domain $\Omega'\subset\subset\Omega''\subset\subset\Omega$. In particular, this gives us upper bounds for $\|\Delta A\|_{L^q(\Omega'')}$ for every $q\in (1,\infty)$, and $L^q$ regularity theory therefore gives us the desired estimates for $A$ in $W^{2,q}(\Omega')$.
\end{proof}

Finally, we remark that higher regularity of $u$ and $A$ in the Coulomb gauge follows in a standard way---e.g., via Schauder theory---from the $W^{2,q}$ estimates obtained in the preceding proposition.

\end{document}